\documentclass[reqno,11pt]{amsart}
\usepackage{amscd,amsfonts,mathrsfs,amsthm,enumerate}
\usepackage{amssymb, amsmath,bigints}
\usepackage{stmaryrd}
\usepackage{epsfig}
\usepackage[sorted-cites,nobysame]{amsrefs}
\usepackage{latexsym}
\usepackage[usenames,dvipsnames]{color}
\usepackage[T1]{fontenc}
\usepackage{calligra}
\DeclareFontShape{T1}{calligra}{m}{n}{<->s*[2.3]callig15}{}
\DeclareMathAlphabet{\mathcalligra}{T1}{calligra}{m}{n}
\usepackage[colorlinks=true,linkcolor=blue,citecolor=BrickRed,urlcolor=RoyalBlue]{hyperref}
\newcommand{\R}{\mathbb{R}}

\def\si{\zeta}

\def\comment#1 {{\color{red}(Comment: #1) }}

\def\gind {\operatorname{g}}
\def\C {\mathbb{C}}

\newtheorem{theorem}{Theorem}[section]
\newtheorem{mythm}{Theorem}

\newtheorem{lemma}[theorem]{Lemma}

\newtheorem{corollary}[theorem]{Corollary}

\newtheorem{definition}[theorem]{Definition}
\theoremstyle{definition}
\newtheorem*{assumption*}{$\lambda_{1}$-Condition}

\newtheorem{remark}[theorem]{Remark}
\newtheorem{example}[theorem]{Example}

\def\pproof#1{\@ifnextchar[\opargproof
{\opargproof[\it Proof of #1.]}}
\def\opargproof[#1]{\par\noindent {\bf #1 }}

\makeatother

\numberwithin{equation}{section}

\begin{document}

\title[Degenerate curve shortening flow]{Curve shortening flow on Riemann surfaces with conical singularities}
\author[N. Roidos]{\textsc{Nikolaos Roidos}}
\address{Nikolaos Roidos\newline
Department of Mathematics\newline
University of Patras\newline
26504 Rio Patras, Greece\newline
{\sl E-mail address:} {\bf roidos@math.upatras.gr}
}
\author[A. Savas-Halilaj]{\textsc{Andreas Savas-Halilaj}}
\address{%
	Andreas Savas-Halilaj\newline
	Department of Mathematics\newline
	University of Ioannina\newline
	45110 Ioannina, Greece\newline
	{\sl E-mail address:} {\bf ansavas@uoi.gr}
}
\date{}
\thanks{The first author was supported from the Deutsche Forschungsgemeinschaft, Grant SCHR 319/9-1.
The second author was supported from the General Secretariat for Research and Innovation (GSRI) and the Hellenic Foundation for Research and Innovation (HFRI), Grant No: 133.}

\begin{abstract}
We study the curve shortening flow on Riemann surfaces with finitely many conformal conical singularities.
If the initial curve is passing through the singular points, then the evolution is governed by a degenerate quasilinear parabolic equation. In this case,
we establish short time existence, uniqueness, and regularity of the flow. We also show that the evolving curves stay fixed at the singular points
of the surface and obtain some contracting and convergence results.
\end{abstract}

\maketitle
\setcounter{tocdepth}{1}
\tableofcontents

\section{Introduction}

There is a long history of curve shortening flow (CSF for short) on smooth Riemann surfaces.
The existence, regularity, and long-time behavior of CSF in $\R^2$ have been studied extensively by
Gage and Hamilton \cite{GaHa1}, Grayson\cite{grayson3}, \cite{grayson2}, Angenent \cite{Ang3}, and Huisken \cite{huisken}.
Later, Grayson \cite{Gr1} and Gage \cite{Ga1}, investigated the long-time behavior of CSF on smooth surfaces. In a series of papers, Angenent developed a more general theory for parabolic equations of immersed curves on smooth
surfaces; see \cite{Ang3,Ang2,Ang1,Ang4,Ang5}.

The subject of this paper, is the CSF on conical Riemann surfaces $(\Sigma,\gind)$.
A point $p\in\Sigma$ is called {\em conformal conical point of order $\beta\in\R$}, if
the metric $\gind$ has in local complex coordinates a representation of the form
$$
e^{2h}|z-z(p)|^{2\beta}|dz|^2,
$$
where here $h$ is a (real) analytic function. If $h$ is zero, then the metric $\gind$
is flat and $p$ is called {\em conical}. 
The analysis and geometry of surfaces with conical singularities was investigated among others by
Finn \cite{Finn}, Huber \cite{Huber}, Cheeger \cite{Cheeger3,Cheeger2,Cheeger1}, McOwen \cite{mcowen}, Hulin and
Troyanov \cite{Tro4,Tro1,Tro3,Tro2}, Br\"uning and Seeley \cite{BS3,BS2,BS1},
Melrose \cite{Melrose} and Schulze \cite{Schu2,Schu}.

The primary goal of this article is to prove the following:

\begin{mythm}\label{THMA}
Let $(\Sigma,\gind)$ be a Riemann surface where $\gind$ is a metric
with conformal conical $($not necessarily distinguished\,$)$ points
$\{p_1,\dots,p_n\}$, $n\ge 2$, of orders $-1<\beta_1\le\cdots\le \beta_n<0$,
respectively, and let $\gamma\colon[0,1]\to\Sigma$ be
a continuous, immersed, not necessarily closed curve satisfying the conditions:
\begin{enumerate}
\item[\rm(c$_1$)]
There is a partition $s_1=0<s_2<\cdots <s_{n}=1$ of the interval $[0,1]$
so that $\gamma(s_j)=p_j$, $1\le j\le n$.
\smallskip
\item[\rm(c$_2$)]
The image $\gamma([0,1]\backslash\{s_1,\dots,s_n\})$ does not meet any singular point of
the surface $\Sigma$.
\smallskip
\item[\rm(c$_3$)]
The curve $\gamma$ is $C^{\infty}$-smooth on each interval $[s_j,s_{j+1}]$, $1\le j\le {n-1}$.
\end{enumerate}
\vspace{-2pt}
Then, there exist $T>0$ and a unique $\varGamma\colon[0,1]\times[0,T]\to\Sigma$
such that:
\begin{enumerate}
\item[\rm(1)]
$\varGamma$ is continuous on $[0,1]\times[0,T]$,
 $C^{\infty}$-smooth on
$(s_j,s_{j+1})\times(0,T)$ and $C^{1}$-smooth on $[s_j,s_{j+1}]\times[0,T]$, $1\le j\le n-1$.
\smallskip
\item[\rm(2)]
$\varGamma$ stays fixed at the singular points, i.e. $\varGamma(s_j,t)=p_j$, $1\le j\le n$, and
$t\in[0,T]$.
\smallskip
\item[\rm(3)]
$\varGamma$ satisfies the evolution equation
\begin{equation}\label{DCSF}
\left\{
\begin{array}{ll}
\varGamma^{\perp}_t(s,t)={\bf k}_{\gind}(\varGamma(s,t)),& (s,t)\in[0,1]\times(0,T),
\\[0.2cm]
\varGamma(s,0)=\gamma(s), &s\in[0,1],
\end{array}
\right.
\end{equation}
where $\{\cdot\}^{\perp}$ denotes the orthogonal projection on the normal space and
${\bf k}_{\gind}$ the geodesic curvature vector of the evolved curve $s\mapsto \varGamma(s,\cdot\,)$.
\end{enumerate} 
\end{mythm}

\begin{remark}
Theorems \ref{THMA} hold for initial curves that are not necessarily
$C^\infty$-smooth up to the singular points. As a matter of fact,
we solve \eqref{DCSF} with initial data lying in suitable Mellin-Sobolev spaces.
For the precisely required regularity assumptions on the initial data, the uniqueness and the asymptotic behavior
of the solution, we refer to Theorem
\ref{wexreg}, Corollary \ref{pb1}, and Theorem \ref{kgregularitys}.
We emphasize that,
even if we start with a $C^{\infty}$-smooth curve, the evolving curves may
instantaneously
become only piecewise $C^1$-smooth up to the singular points; see Fig. \ref{DCSFFFF}.
\end{remark}
\begin{figure}[ht]
	\includegraphics[scale=0.8]{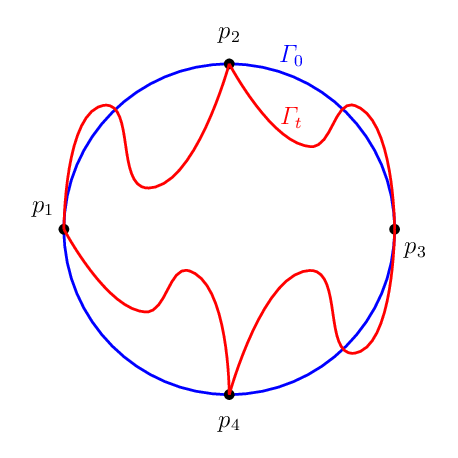}\caption{Degenerate curve shortening flow.}\label{DCSFFFF}
	\end{figure}

The next theorem describes the asymptotic behavior of the curvature near the singular points
of evolved curves under the degenerate curve shortening flow \eqref{DCSF} (DCSF for short).

\begin{mythm}\label{THMB}
Let $(\Sigma,\gind)$ and $\gamma\colon[0,1]\to\Sigma$ be
as in Theorem \ref{THMA}.
Then, for each $\varepsilon>0$ there exists a time $T_{\varepsilon}\in(0,T]$ such that for the solution $\varGamma$ of \eqref{DCSF}
on $[0,T_{\varepsilon}]$, close to $p_j$, $j\in\{1,\dots,n\}$, we have the pointwise estimate
\begin{equation*}
|{\bf k_{\gind}}|_{\gind}\leq\left\{
\begin{array}{lll}
c_j|s-s_j|^{1+\beta_j-\varepsilon} &\text{if} &\beta_n\leq -1/2,\\
c_j|s-s_j|^{-\beta_n\frac{1+\beta_j}{1+\beta_n}-\varepsilon} &\text{if}& \beta_n>-1/2,
\end{array}
\right.
\end{equation*}
where $c_{j}\geq 0$ is a constant depending only on $\varepsilon$ and $T_\varepsilon$, and $|\cdot|_{\gind}$ denotes the norm
with respect to $\gind$.
\end{mythm}

\begin{remark}
Let us make some comments regarding Theorems \ref{THMA} and \ref{THMB}.
\begin{enumerate}[\rm(1)]
\item
If $\Sigma$ is a surface as in Theorem \ref{THMA}
and
$\gamma$ is a $C^2$-smooth curve passing through singular points, then its geodesic curvature
${\bf k_{\gind}}$
vanishes at the singular points; see Lemma \ref{lemm ana p}. This is a strong indication that the 
evolved curves under the 
DCSF will stay fixed at the singular points.
\medskip
\item Let us give another motivating heuristic of why there exists a solution of 
\eqref{DCSF} which satisfies (2) of Theorem \ref{THMA}, i.e. a solution which stays fixed at the singular points. Consider $\mathbb{C}$
equipped with the singular metric
$$\gind=|z|^{2\beta}|dz|^2,\,\,\, \text{where}\,\,\,\beta\in(-1,0).$$
Then $\mathbb{C}$ minus a closed half-line starting from the origin of $\C$ is isometric with the (open)
sector
$$
\varLambda=\big\{z\in\C\backslash\{0\}:|{\rm arg}(z)|<\pi(1+\beta)\big\}\subset(\C,\gind_{\rm euc}),
$$
equipped with the euclidean metric $\gind_{\rm euc}$; see Subsection \ref{m1} and Fig. \ref{cone1}.
Note that it is not possible to map isometrically an open neighbourhood containing $0\in(\C,\gind)$
to an open subset of $(\C,\gind_{\rm euc})$.
Let us now restrict for simplicity to the case
$\beta=-1/2$ and consider an embedded curve $\Gamma_0$ in $\varLambda$ whose both ends are arriving
at the origin of $(\C,\gind_{\rm euc})$ with zero curvature. Since $0$ does not
belong to the sector $\varLambda$, the under consideration initial curve $\varGamma_0$ is actually an open curve and
the standard CSF in $(\varLambda,\gind_{\rm euc})$ with initial data $\varGamma_0$ has
infinitely many solutions. This is due to the fact that there are infinitely many ways
to extend $\varGamma_0$ to a closed curve containing $0$, and each such extension produces a solution to the
CSF.  However, among all these solutions of the CSF, there exists one solution which stays at the
origin. Indeed, extend $\varGamma_0$
to $(\C,\gind_{\rm euc})$ by reflecting it with respect to (and including) the origin of $\C$. In this way we obtain a closed symmetric $C^2$-smooth
figure-eight curve and the CSF in $(\C,\gind_{\rm euc})$ exits for short time and preserves these symmetries; see for instance \cite{Smoczyk}.
Thus, the parts of the evolved curves within $\varLambda$ stays fixed at the origin. Going back to
$(\C,\gind)$, we get a solution of the DCSF staying fixed at the singular point.
\medskip
\item
We would like to
emphazise that the procedure described in the above example cannot be carried out to prove
Theorem \ref{THMA} in the general case,
since the shape of the initial curve might be more complicated, as it might not be contained in
any sector, or $\Sigma$ can be non-compact of high genus and with several singular points.
\medskip
\item
The behavior of the CSF is quite different when we consider singular surfaces. If the initial curve
is passing through singular points, the flow is governed by a degenerate
PDE. In this case, the CSF reduces to the investigation of a degenerate quasilinear differential equation of the form
$$
w_t-s^{\alpha_1}(1-s)^{\alpha_1}P(s,w,w_s)w_{ss}=s^{\alpha_2}(1-s)^{\alpha_2}Q(s,w,w_s),
$$
where $\alpha_1,\alpha_2\in\R$, $s\in(0,1)$ and $P$, $Q$ are appropriate functions with $P$ being positive.
There has been a lot of research for such evolution equations but,
most of the results ensure only generalized solutions (e.g. weak, mild, strong, viscosity)
rather than classical ones; see for example the papers \cite{Chen,DoRu,evans2,Gi1,Gi2,GoLin, JaLu,Vo,VoHu}.
\smallskip
\item
To prove Theorems \ref{THMA} and \ref{THMB}, we write the evolving curves for a short
time as a graph over the initial curve; see \eqref{hanoofa}.
As a matter of fact, we evolve separately
each piece of the curve joining two singular points.
Following Huisken and Polden
\cite{HuPo}, we express the geodesic curvature
in terms of the height function; see \eqref{s5}. Then, the problem \eqref{DCSF} becomes equivalent to the degenerate parabolic differential equation \eqref{kwer1}-\eqref{kwer2}.
To solve such problems, we employ maximal $L^{q}$-regularity theory for quasilinear parabolic equations in combination with $H^{\infty}$-calculus results for cone differential operators;
see Sections \ref{sec4} and \ref{sec5} for more details.
\end{enumerate}
\end{remark}

Due to a deep result of Grayson \cite{Gr1} a closed embedded curve in a compact Riemann surface must either
shrink to a round point in finite time or else it converges to a simple geodesic in infinite time. In the case of planar curves,
using the Gauss-Bonnet formula, the maximal
time of the existence of the solution of CSF can be explicitly computed in terms of the enclosed area $\mathcal{A}(0)$ of the initial curve, i.e.
$$
T_{\max}={\mathcal{A}(0)}/{2\pi}.
$$
So it is natural to ask if we can estimate the maximal time of the existence of the solution and detect the
final shape of a curve moving by DCSF. In this direction, we derive the following partial result in the case
where $\Sigma$ is conformally equivalent to $\C$.

\begin{mythm}\label{THMC}
Let $(\C,\gind)$ be a flat singular Riemann surface
with conical singular points $\{p_1,\dots,p_n\}$ of orders
$-1<\beta_1\le\cdots\le \beta_n<0,$
respectively. Let $\gamma\colon[0,1]\to\C$ be
a $C^{\infty}$-smooth closed embedded curve passing through the singular {\rm (}distinct{\rm)} points $\{p_1,\dots,p_m\}$, $m\le n$, and containing the
rest into its interior.
Then,
the enclosed areas $\mathcal{A}$ of the evolved curves under the {\rm DCSF}
satisfy
$$
\mathcal{A}_t(t)=-2\pi-\sum_{j=1}^m(\pi-\alpha_j(t))\beta_j+\sum_{j=1}^m\alpha_j(t)-2\pi\sum_{j=m+1}^n\beta_j,
$$
where $\{\alpha_1,\dots,\alpha_m\}$ are the {\rm(}time dependent{\rm)} exterior 
angles of the evolved curves formed at the singular points $\{p_1,\dots,p_m\}$.
\end{mythm}

Additionally, based on a result of Grayson \cite{grayson3}, we prove the following result.

\begin{mythm}\label{THMD}
Let $\gamma:[0,1]\to(\mathbb{C},|z|^{2\beta}|dz|^2)$, $\beta\in(-1,0)$, be an embedded smooth closed
regular curve passing through the origin $O$ of $\mathbb{C}$. Assume that at least one of the following two
conditions is satisfied:
\vspace{-2pt}
\begin{enumerate}
\item[\rm (a)]
$\beta\le -1/2$ and the curve avoids a half-line starting from the origin.
\smallskip
\item [\rm (b)] $\beta>-1/2$ and the curve avoids a sector of opening angle $\pi(1+2\beta)$. 
\end{enumerate}
\vspace{-2pt}
Then, the {\rm DCSF} with $\gamma$ as initial data will contract in finite time at $0\in\C$.
\end{mythm}

Finally, we derive an analogue of Grayson's Theorem \cite{Gr1} in the case where the curve is not passing through any
singular point.

\begin{mythm}\label{THME}
Suppose that $(\Sigma,\gind)$ is a compact Riemann surface with conical singularities $\{p_1,\dots,p_n\}$, with each singularity being of order
less or equal than $-1$.
If $\gamma\colon[0,1]\to M=\Sigma\backslash\{p_1,\dots,p_n\}$ is an embedded smooth closed curve,
then the {\rm CSF} either contract to a non-singular point in finite time or else converge to a simple
closed geodesic of $M$ as time tends to infinity.
\end{mythm}

The paper is structured as follows: In Section \ref{CRSSM}, we set up the notation and investigate the local geometry of
singular surfaces and their geodesics. In Section \ref{sec3}, we present the Gauss-Bonnet formula for singular
surfaces with boundary. In Section \ref{sec4}, we discuss sectorial operators and basic facts about
maximal regularity for parabolic differential equations. Section \ref{sec5} is devoted to cone differential operators. Mellin-Sobolev
spaces are introduced and their properties are presented.
In Section \ref{sec6} of the paper, we derive the short time existence, uniqueness, and regularity of DCSF. In Sections \ref{sec7} \& \ref{sec8}, we derive
evolution equations for curves evolving by DCSF and obtain various geometric properties. In Section \ref{sec9}, we prove our main results.

\section{Curves on Riemann surfaces with singular metrics}\label{CRSSM}
In this section, we set up the notation and review some basic facts about singular metrics on
surfaces following the exposition in \cite{Hul,Tro1,Tro2}.
Throughout this paper $\Sigma$ will be a Riemann surface, i.e. an oriented two dimensional manifold.
It is well-known that $\Sigma$ becomes a complex manifold. We also assume that $\Sigma$ is equipped with a complete Riemannian metric
$\gind_0$.

\subsection{Geometry of singular surfaces}
Let
$\gind$ be a (weakly) conformal to $\gind_0$ metric, which is positive definite away from given points $\{p_1,\dots,p_n\}$ in $\Sigma$,
such that around each point $p_\ell$ there exists a non-singular complex chart $z_{\ell}\colon U_{\ell}\to\C$ defined in a neighborhood $U_{\ell}$ of $p_{\ell}$ with
$z_{\ell}(p_{\ell})=0$
and
$$
\gind=e^{2h_{\ell}}|z_{\ell}|^{2\beta_{\ell}}|dz_{\ell}|^2.
$$
Here $h_{\ell}\colon U_{\ell}\to\R$ is considered to be a (real) analytic function. The pair $(U_{\ell},z_{\ell})$ is called {\em conformal coordinate chart}. Often the euclidean metric will be denoted
by $\langle\cdot\,,\cdot\rangle$ and $\gind$ by $\langle\cdot\,,\cdot\rangle_{\gind}$.
The singular metric $\gind$ is called {\em conformal conical metric on} $\Sigma$ and the point $p_{\ell}$ is called {\it conformal conical singular point of order} $\beta_{\ell}$.
Notice that the definition
of singularity is independent of the choice of the conformal coordinate chart.
The Gaussian
curvature $K_{\gind}$ of the metric $\gind$ near a singular point $p_{\ell}$ of order $\beta_{\ell}$ is given by
$$
K_{\gind}=-\frac{\Delta\log \big(|z|^{2\beta_{\ell}}e^{2h_{\ell}}\big)}{2|z|^{2\beta_{\ell}}e^{2h_{\ell}}}
=-|z|^{-2\beta_{\ell}}e^{-2h_{\ell}}\Delta h_{\ell},
$$
where $\Delta$ stands for the Laplacian with respect to the euclidean metric. If $h_j\equiv 0$, the point $p_j$ will be simply called {\em conical singular point}.

\begin{example}
On every oriented surface $\Sigma$ there are plenty of such singular metrics. For instance, let $\gind_0$ be a Riemannian metric
on $\Sigma$ and $f:\Sigma\to\C$ a meromorphic function. Then,
$$\gind=|f|\gind_0$$
is a singular Riemannian metric with the singular points being the zeroes and the poles of $f$. 
\end{example}

\subsection{Curves on singular surfaces}\label{secsetup}
Let $\Sigma$ be a Riemann surface equipped with a conformal conical metric $\gind$
and singular points $\{p_1,\dots,p_n\}$.
The Levi-Civita connection associated with $\gind$ on the tangent bundle of $\Sigma\backslash\{p_1,\dots,p_n\}$
will be denoted by $\nabla$.

Let $I$ be an interval of $\R$ and let $\gamma\colon I\to(\Sigma,\gind)$ be a smooth immersed curve
passing through singular points
of the surface, that is the differential of $\gamma$ is nowhere zero. For such a given curve, away from the singular points, we write
$\tau_{\gind}$ and $\xi_{\gind}$ for its {\it unit tangent} and {\it unit normal vectors}, respectively. We shall always assume that
$\{\tau_{\gind}, \xi_{\gind}\}$ forms a positively oriented basis of $\Sigma$. The {\it geodesic curvature vector} ${\bf k}_{\gind}$
and the {\it scalar geodesic curvature} $k_{\gind}$ of $\gamma$
are given by the formulas
$$
{\bf k}_{\gind}=\nabla_{\tau_{\gind}}\tau_{\gind}\quad\text{and}\quad{\bf k}_{\gind}=k_{\gind}\,\xi_{\gind}.
$$
In the following lemma, we compute the geodesic curvature vector $\bf k_{\gind}$ of a regular curve $\gamma$
passing through a singular point.
\begin{lemma}\label{kg}
The geodesic curvature vector $\bf{k}_{\gind}$ of a smooth immersed curve
$\gamma\colon I\to U\subset\C$ passing through a singular point $p_\ell\in\Sigma$ of order $\beta_\ell$,
in a conformal coordinate chart around $p_{\ell}$, is given by the formula
$$
{\bf k}_{\gind}=|\gamma|^{-2\beta_\ell}e^{-2h_\ell}\big(k-\beta_\ell \langle\gamma,\xi\rangle|\gamma|^{-2}
-\langle Dh_\ell,\xi\rangle\big)\xi,
$$
where $\xi$ is the unit normal and $k$ is the curvature of $\gamma$ with respect to the euclidean metric of $\C$.
\end{lemma}
\begin{proof}
Denote by $D$ the usual connection of $\R^2$.
Then, from the formula relating the connections of conformally related metrics (see e.g. \cite[pp. 181-182]{carmo}), we have 
\begin{eqnarray*}
\lefteqn{\nabla_vw\,-\,D_vw}\\
&&\,\,\,\,=\frac{1}{2r^{2\beta_\ell}e^{2h_\ell}}\big(v(r^{2\beta_\ell}e^{2h_\ell})w+w(r^{2\beta_\ell}e^{2h_\ell})v-
\langle v,w\rangle D (r^{2\beta_\ell}e^{2h_\ell})\big)\\
&&\,\,\,\,=\beta_\ell r^{-1}\big(v(r)w+w(r)v-
\langle v,w\rangle Dr\big)\\
&&\quad\quad+v(h_\ell)w+w(h_\ell)v-
\langle v,w\rangle D h_\ell,
\end{eqnarray*}
for any vectors $v,w$, where $r$ is the euclidean distance from the origin.
The unit tangent $\tau_{\gind}$
and the unit normal $\xi_{\gind}$ with respect to $\gind$ are related with the corresponding
$\tau$ and $\xi$ with respect to the euclidean metric by the formulas
$$\tau_{\gind}=|\gamma|^{-\beta_\ell}e^{-h_j}\tau\quad\text{and}\quad\xi_{\gind}=|\gamma|^{-\beta_\ell}e^{-h_\ell}\,\xi.$$
Hence, the geodesic curvature $k_{\gind}$ of $\gamma$ with respect to $\gind$ is given by
\begin{eqnarray*}
{\bf k}_{\gind}&=&\langle\nabla_{\tau_{\gind}}\tau_{\gind},\xi_{\gind}\rangle_{\gind}\xi_{\gind}
=|\gamma|^{-2\beta_\ell}e^{-2h_\ell}\big\langle\nabla_{\tau}\tau,\xi\big\rangle\xi\\
&=&|\gamma|^{-2\beta_\ell}e^{-2h_\ell}\big(k-\beta_{\ell} \langle\gamma,\xi\rangle|\gamma|^{-2}
-\langle Dh_\ell,\xi\rangle\big)\xi
\end{eqnarray*}
and this completes the proof.
\end{proof}

Consider the function
$$\varrho=\frac{\langle \gamma,\xi\rangle}{|\gamma|^2}.$$
In \cite[Lemma 3.3]{gssz} is was shown that
$\varrho$ is well defined for any smooth curve passing through the origin.
\begin{lemma}\label{lemm ana p}
	Let $\gamma\in C^{2}([0,1);\mathbb{C})\cap C^{\infty}((0,1);\mathbb{C})$ be a regular curve such that $\gamma(0)=0$. Then, $\varrho$ can be
	extended everywhere. In particular,
	$$\lim_{s\to 0} \varrho(s)=-{k(0)}/{2},$$
	where $k$ is the curvature of $\gamma$ with respect to the euclidean metric. Hence, the geodesic curvature of a $C^2$-smooth
	curve passing through a singular point $p\in\Sigma$ of order $\beta<0$ is zero at $p$.
\end{lemma}
\begin{proof}
	Consider cartesian coordinates for $\C$ such that at the origin the $s$-axis is tangent to $\gamma$. Then, locally, we can express $\gamma$
	as a graph over the $s$-axis. Without loss of generality we may assume that after a reparameterization the curve is locally given as
the graph of a smooth function $y\colon [0,1) \to \R$ with
$y(0)=y_{s}(0)=0$.
If $y$ is identically zero, then $\varrho$ vanishes. So let us suppose that $y$ is not identically zero.
Since $y\in C^{2}([0,1);\mathbb{R})\cap C^{\infty}((0,1);\mathbb{R})$, from Taylor's expansion
we have that close to zero, $y$ has the form
$$
y(s)=s^2\big(y_{ss}(0)/2+h(s)\big)=s^2u(s),
$$
where $h$ tends to zero as $s$ approaches zero. Observe that $2u(s)\to y_{ss}(0)$ as $s$ tends to zero.
Because,
$$
y_{ss}(0)=\lim_{s\to 0}s^{-1}y_{s}(s)=\lim_{s\to 0}\big(2u(s)+su_{s}(s)\big)
$$
we have that $su_{s}(s)\to 0$ as $s$ tends to zero.
For $s>0$, we have
	$$\varrho(s)=\frac{\langle \gamma(s),\xi(s)\rangle}{|\gamma(s)|^2}=\frac{-sy_{s}(s)+y(s)}{(s^2+y^2(s))\sqrt{1+(y_{s}(s))^2}}.$$
	Consequently,
	$$\varrho(s)=\frac{-u(s)-su_{s}(s)}{(1+s^{2}u^2(s))\sqrt{1+(y_{s}(s))^2}}.$$
	Clearly, the function $\varrho$ is well defined for all values of $s\in(0,1)$. Moreover, 
	$$\lim_{s\to 0}\varrho(s)=-u(0)=-{y_{ss}(0)}/{2}=-{k(0)}/{2}.$$
	Combining this fact with Lemma \ref{kg} we complete the proof.
\end{proof}

Since we are interested in evolution of curves in the direction of their geodesic curvature vector, we would like to regard
the curvature as a differential operator and investigate its type. For this reason, let us consider a $C^2$-smooth regular curve with parametrization $\gamma\colon\mathbb{S}^1\to\Sigma$.
This parametrization can be extended to an immersion $F\colon\mathbb{S}^1\times (-\varepsilon,\varepsilon)\to\Sigma$,
for some constant $\varepsilon>0$, where
$$F(\cdot,0)=\gamma(\cdot).$$
Indeed, because of compactness of the set $\gamma(\mathbb{S}^1)$,
there exists a finite covering $\{(U_1,z_1),\dots,(U_m,z_m)\}$ of $\gamma(\mathbb{S}^1)$ by conformal coordinate charts of $(\Sigma,\gind_0)$.
For any $\ell\in\{1,\dots,m\}$ there exists $\varepsilon_{\ell}>0$ such that the graph
$$
F_{\ell}(s,t)=\gamma_{\ell}(s)+t\xi_{\ell}(s),\quad (s,t)\in\mathbb{S}^1\times (-\varepsilon_{\ell},\varepsilon_{\ell}),
$$
is an immersion, where $\gamma_{\ell}$ is the representation of $\gamma$ in the chart $U_{\ell}$
and $\xi_{\ell}$ is the unit normal of $\gamma_{\ell}$ with respect to the euclidean metric of the chart. The map $F$ is produced
by taking
$$\varepsilon=\min\{\varepsilon_1,\dots,\varepsilon_m\}$$
and by gluing together all the graphs $F_1,\dots,F_m$.
It is clear now that any immersed curve $\sigma$, which is sufficiently close to $\gamma$ with respect to the
distance arising from $\gind_0$, can be parametrized in the form
$$
\sigma(s)=F(s,w(s)),
$$
where $w\in C^{1}(\mathbb{S}^1)$. According the above mentioned considerations, we may represent any curve $\sigma$, which is sufficiently close to $\gamma$,
as the global graph of a function $w$ over $\gamma$.
\begin{remark}\label{degenerate k}
Let $\gamma\in C^{3}([0,1);\Sigma)$ be a regular curve such that $\gamma(0)$ is a singular point of order $\beta$.
For simplicity let us assume that the singular metric of $\Sigma$ is represented locally around the singular point in the form
$$
\gind=e^{2h}|z|^{2\beta}|dz|^2.
$$
and that $\gamma(0)=0$.
Consider now a variation $\sigma$ of the curve $\gamma$ given by
\begin{equation}\label{hanoofa}
\sigma=\gamma+w\xi,
\end{equation}
where $w\in C^{2}([0,1);\mathbb{R})$ and $\xi$ is the unit normal of $\gamma$ with respect to the
euclidean metric of the chart. We may assume that $\gamma$ is
parametrized by arc-length and that
$$|kw|<1/2.$$
Using the Serret-Frenet formulas we deduce
\begin{equation}\label{s1}
\sigma_s=(1-kw)\gamma_s+w_s\,\xi.
\end{equation}
Hence, the unit tangent $\tau_\sigma$ and the unit normal vector $\xi_\sigma$ along $\sigma$, with respect to the euclidean
metric, are given by the formulas
\begin{eqnarray}\label{s2}
\tau_\sigma&=&\frac{1-kw}{\sqrt{(1-kw)^2+w^2_s}}\gamma_s+\frac{w_s}{\sqrt{(1-kw)^2+w^2_s}}\,\xi,\\\label{s3}
\xi_\sigma&=&\frac{-w_s}{\sqrt{(1-kw)^2+w^2_s}}\gamma_s+\frac{1-kw}{\sqrt{(1-kw)^2+w^2_s}}\,\xi,
\end{eqnarray}
respectively. By a straightforward computation we obtain that the curvature $k_\sigma$ of $\sigma$ is given by
\begin{equation}\label{gget}
k_\sigma=\frac{1-kw}{\big((1-kw)^2+w_s^2\big)^{3/2}}\,w_{ss}+\frac{2kw^2_s+k_sw_sw+k(1-kw)^2}{\big((1-kw)^2+w_s^2\big)^{3/2}},
\end{equation}
where $k$ is the curvature of $\gamma$. On the other hand
\begin{equation}\label{s4}
\xi=\frac{w_s}{\sqrt{(1-kw)^2+w^2_s}}\,\tau_\sigma+\frac{1-kw}{\sqrt{(1-kw)^2+w^2_s}}\,\xi_\sigma.
\end{equation}
Moreover,
\begin{eqnarray*}
\varrho_\sigma&=&\frac{\langle\sigma,\xi_\sigma\rangle}{|\sigma|^2}\\
&=&\frac{\langle \gamma+w\xi,-w_s\gamma_s+(1-kw)\xi\rangle}{|\gamma+w\xi|^2\sqrt{(1-kw)^2+w^2_s}}\\
&=&\frac{-w_s\langle\gamma,\gamma_s\rangle+(1-kw)\langle\gamma,\xi\rangle+w(1-kw)}{|\gamma+w\xi|^2\sqrt{(1-kw)^2+w^2_s}}.
\end{eqnarray*}
From Lemma \ref{kg}, the geodesic curvature of $\sigma$ is given by the formula
\begin{eqnarray}\label{s5}
k_{\gind}(w)&=&|\gamma+w\xi|^{-\beta}e^{-h(\gamma+w\xi)}\frac{1-kw}{\big((1-kw)^2+w_s^2\big)^{3/2}}\,w_{ss}\\
&+&|\gamma+w\xi|^{-\beta}e^{-h(\gamma+w\xi)}\,\frac{2kw^2_s+k_sw_sw+k(1-kw)^2}{\big((1-kw)^2+w_s^2\big)^{3/2}}\nonumber\\
&-&|\gamma+w\xi|^{-\beta}e^{-h(\gamma+w\xi)}\beta\,
\frac{-w_s\langle\gamma,\gamma_s\rangle+(1-kw)\langle\gamma,\xi\rangle+w(1-kw)}{|\gamma+w\xi|^2\big((1-kw)^2+w^2_s\big)^{1/2}}
\nonumber\\
&+&|\gamma+w\xi|^{-\beta}e^{-h(\gamma+w\xi)}\frac{w_s\langle D h|_{\gamma+w\xi},\gamma_s\rangle-(1-kw)\langle Dh|_{\gamma+w\xi},\xi\rangle}{\big((1-kw)^2+w^2_s\big)^{1/2}}\nonumber.
\end{eqnarray}
Observe that if $w$ is of the form
$$w=su,$$
then the coefficient of the leading term of $k_{\gind}(w)$ vanishes if $\beta<0$ and explodes if $\beta>0$.
\end{remark}

\subsection{Conical singularities}\label{m1}
Let us turn our attention now in the case of surfaces with conical singularities.

We start by giving a geometric interpretation of the
order of a singular point. Consider polar coordinates $(r,t)$ in $\C$ and let us fix a real nonzero number $\theta$. Denote by $\sim$ the equivalence relation
given by
$(r_1,t_1)\sim(r_2,t_2)$ if and only if $r_1=r_2$ and $t_1-t_2\in\theta\mathbb{Z}.$
The quotient space
$$
V_\theta=\C/\sim
$$
is called the {\em cone of total opening angle $\theta$}. The justification for the name {\em total opening angle} comes from the observation that
in the case where $0<\theta<2\pi$, the cone $V_\theta$ is exactly the standard cone in $\R^3$ which is made by gluing together the edges of a sector of angle $\theta$ in $\R^2$; see Fig. \ref{cone1}.

\begin{figure}[ht]
	\includegraphics[scale=.7]{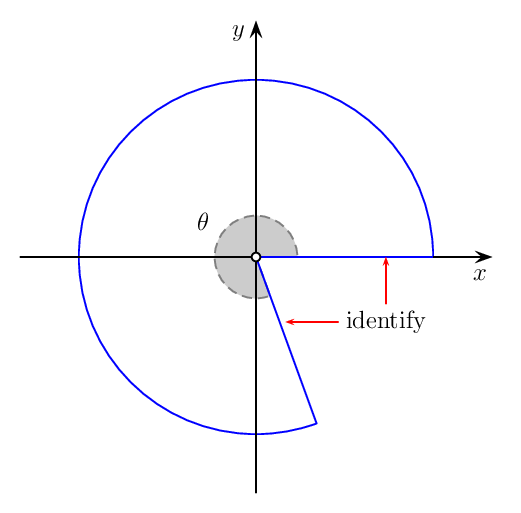}\caption{The cone $V_{\theta}$.}\label{cone1}
	\end{figure}

Denote by the closed half-line of non-negative real numbers by $L$.
It turns out that if $\beta\in(-1,0)$, the Riemann surface
$(\mathbb{C}\backslash L,|z|^{2\beta}|dz|^2)$ is isometric with the sector
$$\varLambda_{\theta}=\{(r,t)\in\R^2:r>0\quad\text{and}\quad 0<t<2\pi(1+\beta)\},$$
of angle $\theta=2\pi(1+\beta)$, equipped with the warped metric
$\gind_\theta=dr^2+r^2dt^2.$
The desired isometry is prescribed by the map $F:\C\backslash{L}\to \varLambda_{\theta}$ given by
\begin{equation}\label{isometry}
\left\{
\begin{array}{lll}
x&\rightarrow&(1+\beta)^{\frac{1}{1+\beta}}r^{\frac{1}{1+\beta}}\cos\big(\frac{t}{1+\beta}\big),\\
y&\rightarrow&(1+\beta)^{\frac{1}{1+\beta}}r^{\frac{1}{1+\beta}}\sin\big(\frac{t}{1+\beta}\big). 
\end{array}
\right.
\end{equation}
Replacing the coordinate $z$ of $\C$ by $\zeta=1/z$, one also sees that if $\beta<-1$, then a punctured
neighborhood of $0$ is a conical end with total opening angle
$$\theta=-2\pi(1+\beta).$$
Finally, let us mention that the case $\beta=-1$ is more complicated and striking geometric phenomena appears. For more details see \cite{Hul} and \cite{Tro2}.

In the following lemma, let us explain the structure of geodesic curves in neighborhoods around conical singular points.

\begin{lemma}\label{geod}
Consider the complex plane $\C$ endowed with the singular metric
$\gind_{\beta}=|z|^{2\beta}|dz|^2$, $\beta\in\mathbb{R}$. In polar coordinates $(r,\varphi)$ the metric takes the form
$
\gind_\beta=r^{2\beta}dr^2+r^{2\beta+2}d\varphi^2.
$
Moreover, straight lines emanating from the origin of $\C$ are geodesics. In addition:
\begin{enumerate}[\rm(A)]
\item
If $\beta\neq-1$, then the geodesic curves are given in polar coordinates by the formula
$$
r^{\beta+1}\cos\big((\beta+1)\varphi-m_2\big)=m^{\beta+1}_1,
$$
where $m_1$ and $m_2$ are real numbers, with $m_1$ being positive.
\medskip
\item
If $\beta=-1$, then the geodesics are either circles centered at the origin
or spirals. In polar coordinates, they are represented by the formula
$
r=m_1e^{m_2\varphi}
$
where $m_1$ and $m_2$ are real numbers, with $m_1$ being positive.
\end{enumerate}
\end{lemma}
\begin{proof}
We can represent, at least locally and away from the origin, our curve $\gamma$ in polar coordinates, i.e. we may write 
$$\gamma(\varphi)=(r(\varphi)\cos\varphi,r(\varphi)\sin\varphi).$$
By a straightforward computation, we see that the curvature $k$ of $\gamma$ with respect to the euclidean metric
and the function $\varrho$ are given by the expressions 
$$
k=\frac{-rr_{\varphi\varphi}+2r^2_{\varphi}+r^2}{\big(r^2+r^2_{\varphi})^{3/2}}\quad\text{and}\quad
\varrho=-\frac{1}{(r^2+r^2_{\varphi})^{1/2}}.
$$
Consequently, $\gamma$ is a geodesic of $\gind_\beta$ if and only if
the distance function $r$ satisfies the ODE
\begin{equation}\label{geod1}
rr_{\varphi\varphi}-(\beta+2)r^2_{\varphi}-(\beta+1)r^2=0.
\end{equation}
Note that \eqref{geod1} is invariant under dilations.
Setting $r_{\varphi}=r\omega(r)$,
then \eqref{geod1} reduces to
\begin{equation}\label{geod2}
r\omega\omega_r=(\beta+1)(\omega^2+1).
\end{equation}
{\em Case A.} Let us suppose that $\beta+1$ is non-zero. By integration we obtain that
$$
\omega^2=(r/m_1)^{2\beta+2}-1,
$$
where $m_1$ is a positive constant. Consequently,
$$
r^2_{\varphi}=r^2\big((r/m_1)^{2\beta+2}-1\big).
$$
By another integration we see that the solution is of the form
$$
r^{\beta+1}\cos\big((\beta+1)\varphi-m_2\big)=m^{\beta+1}_1,
$$
where $m_2$ is another real constant. It is clear that
$$
|(\beta+1)\varphi-m_2|<\frac{\pi}{2}.
$$
It is not difficult to see that any such geodesic is convex and is produced by rotating and dilating around the origin
the curve
$$
r^{\beta+1}\cos\big((\beta+1)\varphi\big)=1.
$$
\begin{enumerate}[\rm (a)]
\item
If $|\beta+1|\ge 1/2$, i.e. if $\beta\in(-\infty,-3/2]\cup[-1/2,+\infty)$, then the function $r$ changes monotonicity only once.
In particular, such a curve is contained in the sector prescribed by the lines
$$
\vartheta_{\pm}=\pm\frac{\pi}{2(\beta+1)},
$$
i.e. in a sector of total opening angle
$$\vartheta=\frac{\pi}{|\beta+1|}\le 2\pi.$$
Observe that if $\beta>0$, then the corresponding geodesics are convex (with respect to the euclidean metric), they do
not contain the origin within its convex hull and satisfy $\min r=m_1.$
For values of $\beta$ in the interval $[-1/2,0)$, the corresponding geodesics are again convex, they contain the origin
within its convex hull and satisfy $\min r=m_1$.
Finally, in the case where $\beta\in(-2,-3/2]$, the curves are not convex and for $\beta\le-2$ they are.
In both latter cases, $\max r=m_1$; see Fig. \ref{Cat2} and \ref{Cat1}.

\begin{figure}[ht]
	\includegraphics[scale=.7]{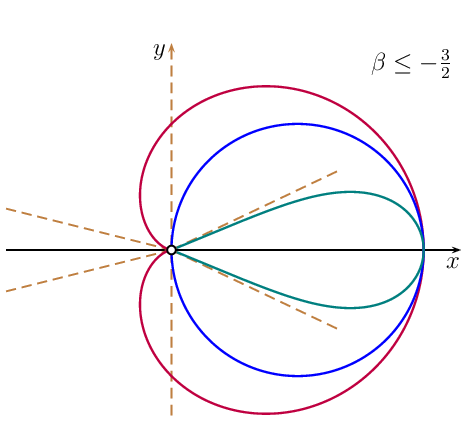}\caption{Geodesic curves for $\beta\in(-\infty,-3/2].$}\label{Cat2}
	\end{figure}
	
\begin{figure}[ht]
	\includegraphics[scale=.6]{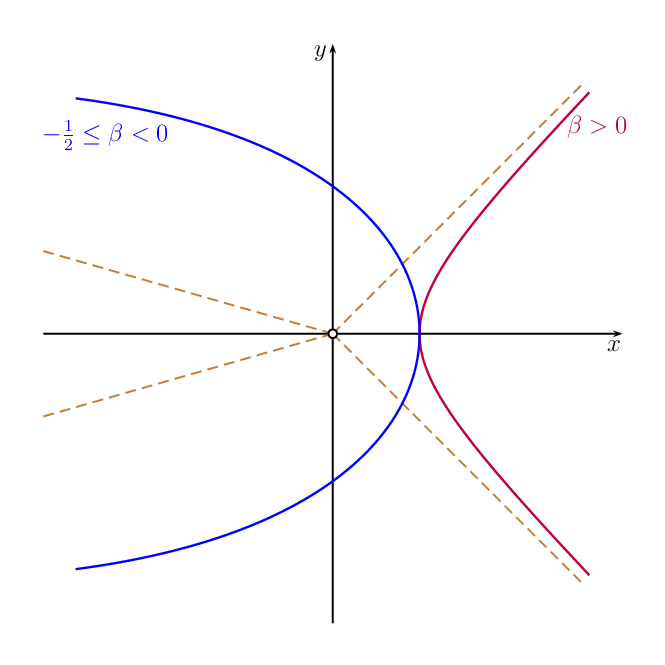}\caption{Geodesic curves for $\beta\in[-1/2,+\infty)$.}\label{Cat1}
	\end{figure}
	
\begin{figure}[ht]
	\includegraphics[scale=.6]{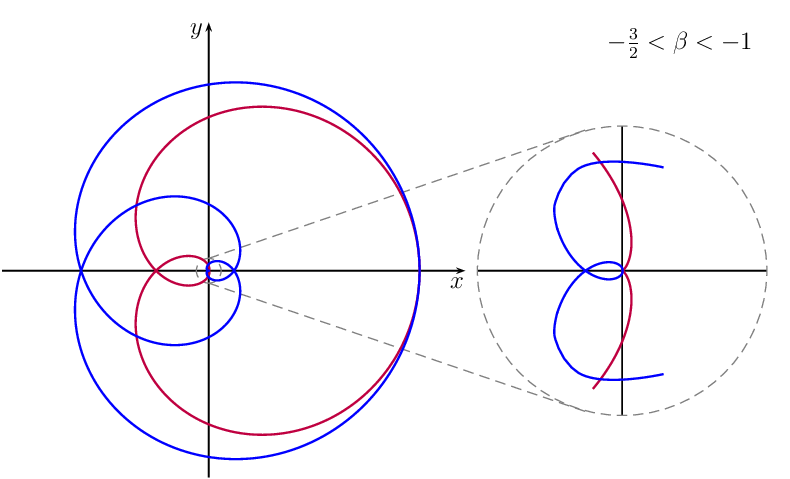}\caption{Geodesic curves for $\beta\in(-3/2,-1).$}\label{Cat3}
	\end{figure}

\item
Suppose now that $\beta\in(-3/2,-1)\cup(-1,-1/2)$. Then the geodesics
have self-intersections. If $\beta\in(-3/2,-1)$, then the curves are closing while if
$\beta\in(-1,-1/2)$ the curves are open; see Fig. \ref{Cat4}.
\begin{figure}[ht]
	\includegraphics[scale=.5]{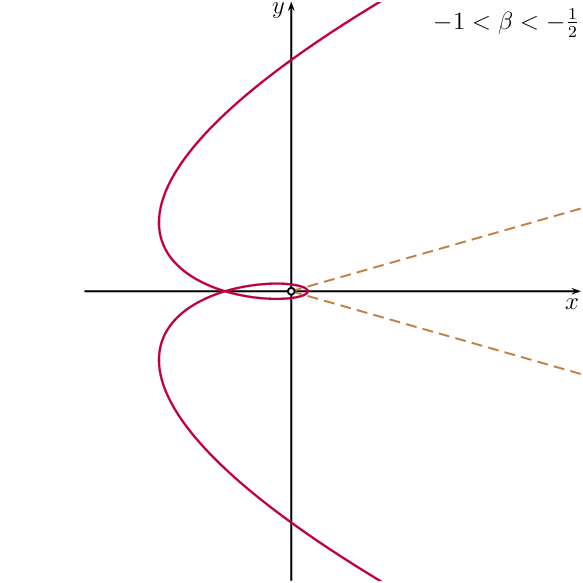}\caption{Geodesic curves for $\beta\in(-1,-1/2).$}\label{Cat4}
	\end{figure}
\end{enumerate}
{\em Case B:} If $\beta=-1$, then \eqref{geod2} becomes
$r\omega\omega_r=0$
and so $r=m_1e^{m_2\varphi}$, where $m_1$ and $m_2$ are constants, with $m_1$ positive. Hence,
the geodesics are either circles with center at the origin or spirals; see Fig. \ref{Spiral}.

\begin{figure}[ht]
	\includegraphics[scale=.5]{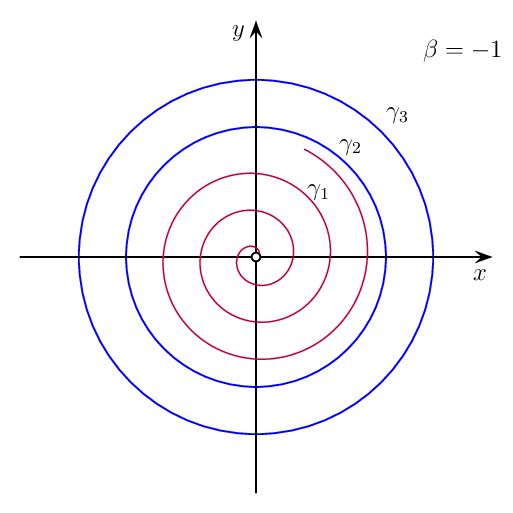}\caption{Geodesic curves for $\beta=-1$.}\label{Spiral}
	\end{figure}

This completes the proof.
\end{proof}

\begin{remark}
According to Lemma \ref{kg}, the geodesic curvature of a regular curve $\gamma$ passing through the origin of $(\C,\gind_\beta)$ is given by the formula
$$
k_{\gind}=|\gamma|^{-\beta}(k-\beta\varrho).
$$
There are plenty of embedded loops through the origin that are convex. For example, consider the {\em tear drop
curve} (see Fig. \ref{TDC}) $\gamma:[0,2\pi]\to\mathbb{C}$ given by
$$
\gamma(s)=\frac{1}{1+\cos^2(s)}\big(-\sin(s),\sin(s)\cos(s)\big).
$$
As can be easily computed,
$k=3|\gamma|=-3\varrho.$
Thus,
$$
k_{\gind}=|\gamma|^{1-\beta}(\beta+3).
$$
Hence, for any $\beta\in(-3,0)$, the above curve is convex with the euclidean as well as with the metric $\gind_{\beta}$.
\begin{figure}[ht]
	\includegraphics[scale=0.7]{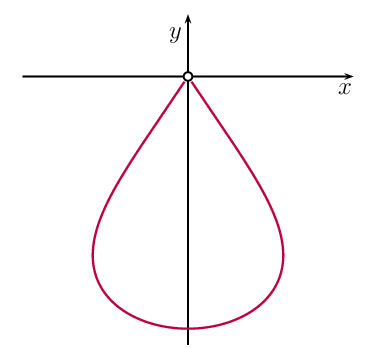}\caption{Tear drop curve.}\label{TDC}
	\end{figure}
\end{remark}

\begin{remark}
If $\gamma=u+iv:\mathbb{S}^1\to\mathbb{C}$ is a curve then $f:\mathbb{S}^1\times\mathbb{S}^{m-1}\to\mathbb{C}^m$, $m>1$, defined by $f(s,x)=\gamma(s)\cdot x$,
produces a Lagrangian submanifold which is invariant under the group of isometries in $\mathbb{O}(m)$,
where the action of $A\in\mathbb{O}(m)$ on $\mathbb{C}^m$ is given by $A(x,y):=(Ax,Ay)$, for all $x,y\in\R^m$. Such a submanifold is called {\em equivariant Lagrangian}. The evolution
under the mean curvature flow
of an equivariant Lagrangian is fully determined in terms of the evolution of the profile curve $\gamma$ under the
equation
\begin{equation}\label{equimeancurvshorty}
\gamma'=\big(k-(m-1)|\gamma|^{-2}\langle \gamma,\xi\rangle\big)\xi,
\end{equation}
where $k$ is the curvature and $\xi$ the unit vector of the curve; see \cite[Section 3]{savas}. The
right hand side in \eqref{equimeancurvshorty} is quite similar with the expression of ${\bf k}_{\gind}$ in Lemma \ref{kg}.
Moreover, the geodesics of $\gind_{m+1}=|z|^{2(m+1)}|dz|^2$ give us the stationary points of the equivariant
Lagrangian mean curvature flow or, equivalently, the minimal equivariant Lagangian submanifolds. Minimal 
submanifolds of this type are called {\em Lawson's minimal Lagragian catenoids.} For further details about 
equivariant Lagrangians and equivariant Lagrangian mean curvature flow see \cite{evans,gssz,harvey,savas,wood}.
\end{remark}

\section{A Gauss-Bonnet formula}\label{sec3}
The next theorem is a variant of the Gauss-Bonnet formula in the case of Riemann surfaces with
conformal conical singularities. In the proof we will need the following elementary lemma. For reader's
convenience we include a proof.

\begin{lemma}\label{ANG}
Let $R$ be a curvilinear triangular region in the plane, whose edges are $C^{1}$-smooth up to the vertices regular curves. At the vertex $O$ the following formula holds true
$$
\lim_{\varepsilon\to 0}\frac{\operatorname{length}(R\cap\partial B_O(\varepsilon))}{\varepsilon}=\alpha,
$$
where $B_O(\varepsilon)$ is the ball of radius $\varepsilon$ centered at the point $O$ and $\alpha$
the angle of the triangle.
\end{lemma}
\begin{proof}
It suffices to prove the lemma in the special case where $O$ is the origin of $\R^2$, the first triangle
side is the $x$-axis and the second side is a curve which, close to the origin, can be represented in the
form
$x=yu(y),$
where $u\colon[0,\delta)\to\R$ is continuous and positive for $y\in[0,\delta)$, for some $\delta>0$. In this case, the angle of the
triangle at the vertex $O$ is $\pi/2$; see Fig. \ref{CLT}.

\begin{figure}[ht]
	\includegraphics[scale=.8]{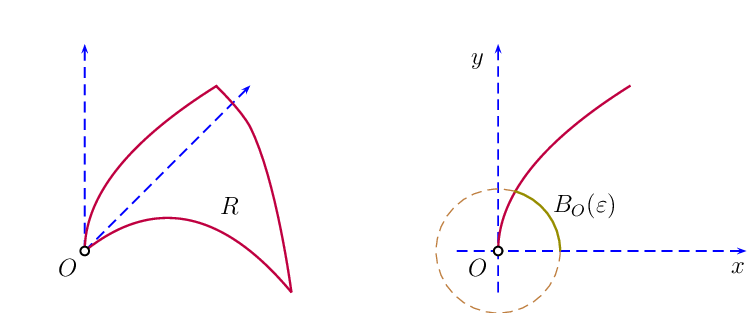}\caption{Curvilinear triangle.}\label{CLT}
	\end{figure}

Let us parametrize the circle of radius $\varepsilon>0$ centered at the origin
by the map $\gamma_\varepsilon:[0,2\pi]\to\R$ given by
$\gamma_\varepsilon(t)=\varepsilon(\cos t,\sin t)$.
Clearly, for sufficiently small values of $\varepsilon$, the circle intersects the triangle only at two points; the $x$-axis
for $t=0$ and the $y$-axis for $t=t_\varepsilon>0$, where $t_\varepsilon$ is the unique solution of
$
\cos(t_\varepsilon)=\varepsilon\sin(t_\varepsilon)u(\varepsilon\sin (t_\varepsilon)).
$
Passing to the limit as $\varepsilon$ tends to zero, we see that
$\cos(t_\varepsilon)\to 0$ as $\varepsilon$ tends to zero and so $t_{\varepsilon}\to\pi/2$.
Consequently,
$$
\lim_{\varepsilon\to 0}\frac{\operatorname{length}(R\cap\partial B_O(\varepsilon))}{\varepsilon}
=\lim_{\varepsilon\to 0}\frac{1}{\varepsilon}\int_{\gamma_\varepsilon}ds=\lim_{\varepsilon\to 0}\int_0^{t_\varepsilon}dt=
\frac{\pi}{2}.
$$
This completes the proof.
\end{proof}

We give now the proof of the Gauss-Bonnet formula, following closely ideas of Troyanov \cite[Proposition 1]{Tro1}.
\begin{theorem}\label{SGB}
Let $(\Sigma,\gind_0)$ be a Riemann surface and $\gind$ a singular metric
conformal to $\gind_0$
with conformal conical singularities. Suppose that $S$ is a compact region
of $\Sigma$ containing the singular points $\{p_1,\dots,p_n\}$ with corresponding
orders $\beta_1,\dots,\beta_n$, such that the first $m$ points belong to $\partial S$. Assume further that
the boundary $\partial S$ is a $C^{1}$-smooth simple regular curve up to the singular points of the surface.
Then, the following formula holds true:
$$
\frac{1}{2\pi}\int_S K_{\gind}+\frac{1}{2\pi}\int_{\partial S} k_{\gind}=\chi(S)
+\frac{1}{2\pi}\sum_{j=1}^m(\pi-\alpha_j)\beta_j+\sum_{j=m+1}^n\beta_j-\frac{1}{2\pi}\sum_{j=1}^m\alpha_j,
$$
where $\chi(S)$ is the topological Euler characteristic of $S$ and $\{\alpha_1,\dots,\alpha_m\}$ are the exterior 
angles of $\partial S$, with respect to $\gind_0$, at the vertices $\{p_1,\dots,p_m\}$.
\end{theorem}
\begin{proof}
According to our set-up there exists a smooth function $v$, defined away from the singular points of the surface, such that
$\gind=e^{2v}\gind_0.$
Moreover, in a conformal coordinate chart around a singular point $p_\ell$ of order $\beta_\ell$, the function $v$ has the form
$$
v=\beta_\ell\log|z|+h_\ell,
$$
where $h_\ell$ is smooth. The Gaussian curvatures $K_{\gind}$
and $K_{\gind_0}$ of the metrics $\gind$ and $\gind_0$, respectively, are related by the formula
$$
K_{\gind}=e^{-2v}\big(K_{\gind_0}-\Delta_{\gind_0} v),
$$
where here $\Delta_{\gind_0}$ is the Laplace operator with respect to $\gind_0$. Moreover, the area elements $dA_{\gind}$
and $dA_{\gind_0}$, the length elements $ds_{\gind}$ and $ds_{\gind_0}$ and the gradients $D^{\gind}$ and
$D^{\gind_0}$ are related by the expressions
$$
dA_{\gind}=e^{2v}dA_{\gind_0},\quad ds_{\gind}=e^{v}ds_{\gind_0}\quad\text{and}\quad D^{\gind}
=e^{-2v}D^{\gind_0}.
$$
Hence,
$$
K_{\gind}dA_{\gind}=K_{\gind_0}dA_{\gind_0}-(\Delta_{\gind_0} v) dA_{\gind_0}.
$$
Moreover, as in Lemma \ref{kg}, we obtain that the geodesic curvatures $k_{\gind}$ and $k_{\gind_0}$ are related by the formula
$$
k_{\gind}ds_{\gind}=k_{\gind_0}ds_{\gind_0}-\langle D^{\gind_0} v,\eta\rangle_{\gind_0} ds_{\gind_0},
$$
where $\eta$ is the unit inward pointing normal of $\partial S$ with respect to $\gind_0$. Hence,
\begin{eqnarray*}
&&\frac{1}{2\pi}\int_SK_{\gind}dA_{\gind}+\frac{1}{2\pi}\int_{\partial S}k_{\gind}ds_{\gind}
=\frac{1}{2\pi}\int_SK_{\gind_0}dA_{\gind_0}+\frac{1}{2\pi}\int_{\partial S}k_{\gind_0} ds_{\gind_0}\\
&&\quad\quad\quad\quad\quad\quad-\frac{1}{2\pi}\int_{S}\Delta_{\gind_0}v dA_{\gind_0}
-\frac{1}{2\pi}\int_{\partial S}\langle D^{\gind_0} v,\eta\rangle_{\gind_0} ds_{\gind_0}.
\end{eqnarray*}
From the standard Gauss-Bonnet formula, we deduce
\begin{eqnarray}\label{GB1}
&&\hspace{-35pt}\frac{1}{2\pi}\int_SK_{\gind}dA_{\gind}+\frac{1}{2\pi}\int_{\partial S}k_{\gind}ds_{\gind}=\chi(S)-\frac{1}{2\pi}\sum_{j=1}^m\alpha_j\nonumber\\
&&\quad\quad\quad-\frac{1}{2\pi}\int_{S}\Delta_{\gind_0}v dA_{\gind_0}
-\frac{1}{2\pi}\int_{\partial S}\langle D^{\gind_0} v,\eta\rangle_{\gind_0} ds_{\gind_0}.
\end{eqnarray}
For any $j\in\{1,\dots,n\}$, let $B_j(\varepsilon)$ be the euclidean disk of radius $\varepsilon>0$ around the singular point $p_j$.
Set
$$
S_\varepsilon=S\backslash\cup_{j=1}^{n}B_j(\varepsilon).
$$
The boundary $\partial S_{\varepsilon}$ is the union of the following three parts
$$
M_1(\varepsilon)=\partial S\backslash\cup_{j=1}^mB_j(\varepsilon),\,\,
M_2(\varepsilon)=S\cap\cup_{j=1}^m\partial B_j(\varepsilon),\,\,
M_3(\varepsilon)=\cup_{j=m+1}^n\partial B_j(\varepsilon);
$$
see Fig. \ref{M1M2M3}.
\begin{figure}[ht]
	\includegraphics[scale=0.66]{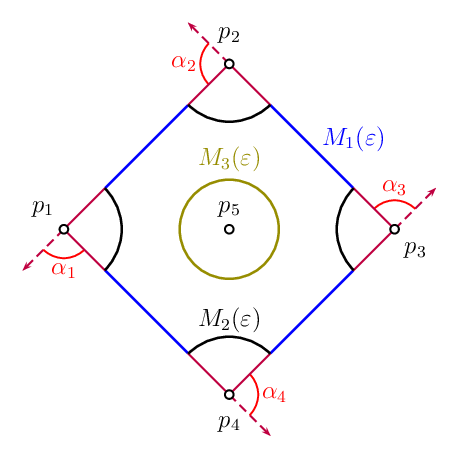}\caption{The sets $M_1(\varepsilon)$, $M_2(\varepsilon)$, $M_3(\varepsilon)$.}\label{M1M2M3}
	\end{figure}

By Green's Theorem we obtain
\begin{eqnarray*}
\int_{S_\varepsilon}\Delta_{\gind_0}v dA_{\gind_0}&=&-\int_{M_1(\varepsilon)}\langle D^{\gind_0} v,\eta\rangle_{\gind_0}ds_{\gind_0}\\
&&-\int_{M_2(\varepsilon)}\langle D^{\gind_0} v,\eta\rangle_{\gind_0}ds_{\gind_0}
-\int_{M_3(\varepsilon)}\langle D^{\gind_0} v,\eta\rangle_{\gind_0}ds_{\gind_0}.
\end{eqnarray*}
The minus sign is due to the fact that we use inward pointing normals for the boundary curves.
Note that
$$
\lim_{\varepsilon\to 0}\int_{M_1(\varepsilon)}\langle D^{\gind_0} v,\eta\rangle_{\gind_0}ds_{\gind_0}=
\int_{\partial S}\langle D^{\gind_0} v,\eta\rangle_{\gind_0}ds_{\gind_0}.
$$
Moreover, along the arcs of the discs $B_j(\varepsilon)$, $j\in\{1,\dots,m\}$, we have that
$$
\langle D^{\gind_0} v,\eta\rangle_{\gind_0}ds_{\gind_0}=\frac{\beta_j}{\varepsilon}ds+\langle {Dh_j,\xi}\rangle ds,
$$
where $ds$ the length element with respect to the euclidean metric and $\xi$ is the (euclidean) unit normal of $\partial B_j(\varepsilon)$.
Because $h_j$ is smooth, the term $|Dh_j|$
is uniformly bounded. Integrating, passing to the limit as $\varepsilon$ tends to zero and using Lemma \ref{ANG}, we deduce that
$$
\int_S\Delta_{\gind_0}v dA_{\gind_0}+\int_{\partial S}\langle D^{\gind_0} v,\eta\rangle_{\gind_0}ds_{\gind_0}
=-\sum_{j=1}^m\beta_j(\pi-\alpha_j)-2\pi\sum_{j=m+1}^n\beta_j.
$$
Plugging the last equality to \eqref{GB1}, we obtain the desired formula.
\end{proof}

\section{Functional analytic methods for parabolic problems}\label{sec4}
The curve shortening flow on singular Riemann surfaces is a quasilinear
parabolic problem whose linearized term is a degenerate operator and at the same time
the reaction term apriori could blow up at the singular points. This forces us to employ maximal $L^{q}$-regularity theory. In particular, we search for classical solutions in suitable weighted $L^{p}$-spaces. The main ingredient in this approach is a general short time existence theorem of Cl\'ement and Li \cite{CL}. 

The main requirement in the theorem of Cl\'ement and Li is the maximal $L^{q}$-regularity property for the linearization of the quasilinear operator, which is a functional analytic property for generators of holomorphic semigroups in Banach spaces. To make the paper self-contained, we review some basic facts from the linear theory. We follow closely the exposition in \cite{Am}, \cite{DHP}, \cite{KuW1} and \cite{clpr}.

\subsection{Sectorial operators and functional calculus}
Here we will briefly discuss the class of sectorial operators defined in Banach spaces and recall
some of their functional analytic properties.
In the rest of the paper, we will always assume that $X_1$ and $X_0$ are complex Banach spaces and that $X_{1}$ is densely and continuously
injected in $X_{0}$. Moreover, denote by $\mathcal{L}(X_{1},X_{0})$ the class of bounded operators from $X_{1}$ to $X_{0}$, For simplicity, we denote $\mathcal{L}(X_{0},X_{0})$ by $\mathcal{L}(X_{0})$. Additionally, we denote by $\mathcal{D}(\cdot)$ and $\rho(\cdot)$ the domain and the resolvent set of an operator, respectively.

A {\em $C_{0}$-semigroup or a strongly continuous one-parameter semigroup} on $X_{0}$ is a map $L:[0,+\infty)\rightarrow \mathcal{L}(X_{0})$ that satisfies:\\
(a) The map $t\mapsto L(t)x$ is continuous for each $x\in X_{0}$,\\
(b) $L(s+t)=L(s)L(t)$, $s,t\geq0$,\\
(c) $L(0)=I$.\\
The {\em infinitesimal generator} or simply {\em generator} $A$ of $L$ is defined as the operator on $X_{0}$ given by 
$$
Ax=\lim_{t\rightarrow0}t^{-1}(L(t)x-x),
$$
with domain 
$$
 \mathcal{D}(A)= \{\text{$x\in X_{0}$ such that ${\lim}_{t\rightarrow0}t^{-1}(L(t)x-x)$ exists in $X_{0}$}\}.
$$

A $C_{0}$-semigroup $L$ on $X_{0}$ is called {\em analytic $C_{0}$-semigroup (of angle $\theta$)}, if there exists a $\theta\in (0,\pi/2]$ such that $L$ admits an analytic extension to the sector $\{\lambda\in\mathbb{C}\backslash\{0\}\, |\, |\arg(\lambda)|<\theta\}$ with 
$$
\sup\{\|L(\lambda)\|_{\mathcal{L}(X_{0})}\, |\, \arg(\lambda)|<\phi \,\, \text{and} \,\, |\lambda|\leq1\}<\infty
$$
for each $\phi\in(0,\theta)$. In this case, the conditions (a) and (b) are extended to the above sector.

In the sequel, we will focus on linear Cauchy problems whose solutions are expressed through analytic $C_{0}$-semigroups. We start by recalling certain functional analytic properties concerning their generators.

\begin{definition}[Sectorial operators]\label{defsec}
Let $\mathcal{P}(K,\theta)$, $\theta\in[0,\pi)$, $K\geq1$, be the class of all closed densely defined linear operators $A$ in $X_{0}$ such that 
$$
S_{\theta}=\big\{\lambda\in\mathbb{C}\,|\, |\arg(\lambda)|\leq\theta\big\}\cup\{0\}\subset\rho{(-A)},$$
and 
$$(1+|\lambda|)\|(A+\lambda)^{-1}\|_{\mathcal{L}(X_{0})}\leq K \quad \text{when} \quad \lambda\in S_{\theta}.$$
The elements in
$$\mathcal{P}(\theta)=\bigcup_{K\geq1}\mathcal{P}(K,\theta)$$
are called {\em (invertible) sectorial operators of angle $\theta$}.
\end{definition}
\begin{remark}
The openness of the resolvent set implies that
$\mathcal{P}(K,\theta)\subset \mathcal{P}(2K+1,\phi)$
for some $\phi\in(\theta,\pi)$, see e.g. \cite[Chapter III, (4.6.4)-(4.6.5)]{Am}.
Thus, if $A\in \mathcal{P}(\theta)$ we can always assume $\theta>0$.
\end{remark}
In Definition \ref{defsec}, after replacing the condition on the boundedness of the family
$\lambda(A+\lambda)^{-1}$ with {\em Rademacher's boundedness}, we obtain the following stronger condition.
\begin{definition}[$R$-sectorial operators]
Denote by $\mathcal{R}(K,\theta)$, $\theta\in[0,\pi)$, $K\geq1$, the class of all operators $A\in \mathcal{P}(\theta)$ in $X_{0}$ such that for any choice of $\lambda_{1},\dots,\lambda_{n}\in S_{\theta}\backslash\{0\}$ and $\{x_{1},\dots,x_{n}\}\in X_0$, $n\in\mathbb{N}$, we have
\begin{eqnarray*}
\Big\|\sum_{k=1}^{n}\epsilon_{k}\lambda_{k}(A+\lambda_{k})^{-1}x_{k}\Big\|_{L^{2}(0,1;X_0)} \leq K \Big\|\sum_{k=1}^{n}\epsilon_{k}x_{k}\Big\|_{L^{2}(0,1;X_0)},
\end{eqnarray*}
where here $\{\epsilon_{k}\}_{k\in\mathbb{N}}$ is the sequence of the Rademacher functions. Elements in the space
$$\mathcal{R}(\theta)=\bigcup_{K\geq1}\mathcal{R}(K,\theta)$$ are called {\em $R$-sectorial operators of angle $\theta$}. 
\end{definition}

Sectorial operators admit holomorphic functional calculus, which is defined by the Dunford integral formula;
see e.g. \cite[Section 1.4]{DHP}. More precisely, let $\theta\in[0,\pi)$, $A\in\mathcal{P}(\theta)$ and let $H_{0}^{\infty}(\theta)$ be the space of all bounded holomorphic functions $f:\mathbb{C}\backslash S_{\theta}\rightarrow \mathbb{C}$ satisfying 
$$
|f(\lambda)|\leq c \left(\frac{|\lambda|}{1+|\lambda|^{2}}\right)^{\eta}
$$
for any $\lambda\in \mathbb{C}\backslash S_{\theta}$ and some $c,\eta>0$ depending on $f$. 
For any $\rho\geq0$ and $\vartheta\in(0,\pi)$, consider the counterclockwise oriented path 
\begin{eqnarray}\nonumber
\lefteqn{\Gamma_{\rho,\vartheta}=\{re^{-i\vartheta}\in\mathbb{C}\,|\,r\geq\rho\}}\\\label{gammapath}
&&\quad \cup\,\, \{\rho e^{i\phi}\in\mathbb{C}\,|\,\vartheta\leq\phi\leq2\pi-\vartheta\}\cup\{re^{i\vartheta}\in\mathbb{C}\,|\,r\geq\rho\},
\end{eqnarray}
see Fig. \ref{Path}. For simplicity we denote $\Gamma_{0,\vartheta}$ by $\Gamma_{\vartheta}$.
\begin{figure}[ht]
	\includegraphics[scale=.75]{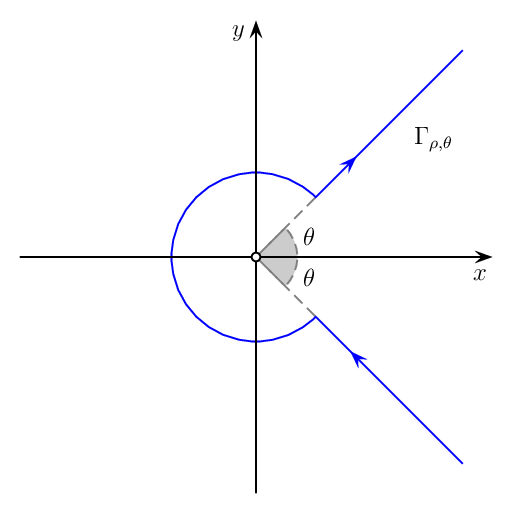}\caption{The path $\Gamma_{\rho,\theta}$.}\label{Path}
	\end{figure}
Then any function $f\in H_{0}^{\infty}(\theta)$ defines an element $f(-A)\in \mathcal{L}(X_{0})$ given by 
\begin{equation}\label{hgsta}
f(-A)=\frac{1}{2\pi i}\int_{\Gamma_{\theta+\varepsilon_{0}}}f(\lambda)(A+\lambda)^{-1} d\lambda,
\end{equation}
where $\varepsilon_{0}\in(0,\pi-\theta)$ is sufficiently small and depends only on $A$ and $\theta$.

The {\em complex powers} of a sectorial operator are a typical example of the holomorphic functional calculus. For $\mathrm{Re}(z)<0$ they are defined by
\begin{equation}\label{cp}
A^{z}=\frac{1}{2\pi i}\int_{\Gamma_{\rho,\theta+\varepsilon_{0}}}(-\lambda)^{z}(A+\lambda)^{-1}d\lambda,
\end{equation}
where $\rho>0$ is sufficiently small.

The family of operators $\{A^{z}\}_{\mathrm{Re}(z)<0}$ together with
$A^{0}=I$ is an analytic $C_{0}$-semigroup on $X_{0}$, see e.g.
\cite[ Chapter III, Theorems 4.6.2 and 4.6.5]{Am}. Moreover, each operator $A^{z}$, $\mathrm{Re}(z)<0$, is an injection and the complex powers
for the positive real part $A^{-z}$ are defined by $A^{-z}=(A^{z})^{-1}$, see e.g. \cite[Chapter III, (4.6.12)]{Am}.

By Cauchy's theorem, we can deform the path in formula \eqref{cp} and define the
{\em imaginary powers} $A^{it}$, $t\in\mathbb{R}\backslash\{0\}$, as the closure of the operator
$$
\mathcal{D}(A)\ni x \mapsto A^{it}x=\frac{\sin(i\pi t)}{i\pi t}\int_{0}^{+\infty}s^{it}(A+s)^{-2}Axds,
$$
see e.g. \cite[Chapter III, (4.6.21)]{Am}. For properties of the complex powers of sectorial operators we refer to 
 \cite[Theorem III.4.6.5]{Am}.
 
Although the imaginary powers $A^{it}$, $t\in\mathbb{R}\backslash\{0\}$, of the operator $A$ are in general unbounded operators, the 
following well-known property holds:

{\em Let $A\in\mathcal{P}(0)$ in $X_{0}$ and assume that there exist numbers $\delta,M>0$ such that $A^{it}\in \mathcal{L}(X_{0})$ and
$$\|A^{it}\|_{\mathcal{L}(X_{0})}\leq M$$
for any $t\in(-\delta,\delta)$. Then, $A^{it}\in \mathcal{L}(X_{0})$ for all $t\in\mathbb{R}$ and there exist numbers
$\phi,\widetilde{M}>0$ such that $$\|A^{it}\|_{\mathcal{L}(X_{0})}\leq \widetilde{M}e^{\phi|t|},$$
for any $t\in\mathbb{R}$.}

\begin{definition}[Bounded imaginary powers]
If an operator $A$ satisfies the above property, we say that $A$ {\em has bounded imaginary powers} and denote
the space of such operators by $\mathcal{BIP}(\phi)$.
\end{definition}

Let us conclude this section by recalling the following important property for operators in the class $\mathcal{P}(\theta)$, which is stronger than the boundedness of the imaginary powers.

\begin{definition}[Bounded $H^{\infty}$-calculus]
We say that the linear operator $A\in \mathcal{P}(\theta)$, $\theta\in[0,\pi)$, {\em has bounded $H^{\infty}$-calculus of angle $\theta$}, and we denote by $A\in \mathcal{H}^{\infty}(\theta)$, if there exists a constant $C>0$ such that
\begin{equation*}
\|f(-A)\|_{\mathcal{L}(X_{0})}\leq C\sup_{\lambda\in\mathbb{C}\backslash S_{\theta}}|f(\lambda)|
\end{equation*}
for any $f\in H_{0}^{\infty}(\theta)$.
\end{definition}

\subsection{Maximal $L^q$-regularity for parabolic equations}

We will see in this section how the notions of the previous section are deeply related to the regularity theory of parabolic equations.
Let us consider the following abstract parabolic first order linear Cauchy problem
\begin{eqnarray}\label{app1}
u_{t}(t)+Au(t)&=&f(t), \quad t\in(0,T),\\\label{app2}
u(0)&=&0,
\end{eqnarray}
where $-A:\mathcal{D}(A)=X_{1}\rightarrow X_{0}$ is the infinitesimal generator of an analytic $C_{0}$-semigroup on $X_{0}$ and
$f\in L^q(0,T;X_{0})$, where $q\in(1,\infty)$ and $T>0$. Denote by $H_{q}^{s}$, $s\in\mathbb{R}$, the usual Bessel potential space and $H_{q}^{\infty}=\cap_{s>0}H_{q}^{s}$.
\begin{definition}
The operator $A$ has {\em maximal $L^q$-regularity} if, for any map $f\in L^q(0,T;X_{0})$, there exists a unique
$$u\in H_{q}^{1}(0,T;X_{0})\cap L^{q}(0,T;X_{1})$$
solving the problem \eqref{app1}-\eqref{app2}.
\end{definition}

\begin{remark}
Recall the following embedding of the maximal $L^q$-regularity space
\begin{equation}\label{interpemb}
H_{q}^{1}(0,T;X_{0})\cap L^{q}(0,T;X_{1})\hookrightarrow C\big([0,T];(X_{1},X_{0})_{1/q,q}\big),
\end{equation}
where $T>0$, $q\in(1,\infty)$ and by $(\cdot,\cdot)_{\eta,q}$, $\eta\in(0,1)$, we denote the real interpolation functor of exponent $\eta$ and parameter $q$; see \cite[Chapter III, Theorem 4.10.2]{Am} and \cite[Example I.2.4.1]{Am}. If $T\in [T_{1},T_{2}]$, for some fixed $0<T_{1}<T_{2}<\infty$, then the norm of the embedding \ref{interpemb} is independent of $T$.
\end{remark}

In the case where $A$ has maximal $L^q$-regularity, we can replace the initial condition \eqref{app2} with $u(0)=u_{0}$, for any function $u_{0}\in (X_{1},X_{0})_{1/q,q}$. The solution $u$ depends continuously on $f$ and $u_{0}$, i.e. there exists a constant $C$, depending only on $A$ and $q$, such that
\begin{equation}\label{maxregineq}
\|u\|_{H_{q}^{1}(0,T;X_{0})}+\|u\|_{L^{q}(0,T;X_{1})}\leq C\big(\|f\|_{L^q(0,T;X_{0})}+\|u_{0}\|_{(X_{1},X_{0})_{1/q,q}}\big).
\end{equation}
The maximal $L^q$-regularity property is independent of $q$ and $T$, and the analytic $C_{0}$-semigroup generation property for $-A$ turns out to be necessary; see \cite{Doremax}. Recall that $A$ generates an analytic $C_{0}$-semigroup on $X_{0}$ if and only if there exists some $c>0$ such that $A+c\in \mathcal{P}(\pi/2)$, see \cite[Proposition 3.1.9, Proposition 3.7.4, Theorem 3.7.11]{ArBaHiNe}.

If $X_{0}$ is a Hilbert space, any operator $A$ such that $-A$ generates an analytic $C_{0}$-semigroup has maximal $L^{q}$-regularity;
see \cite{DeS}. However, in the case of Banach spaces the situation is more complicated. 

\begin{definition}
A Banach space $X_{0}$ is called {\em of class $\mathcal{HT}$}, if for some (and then all) $p\in (1,\infty)$, the Hilbert transform $H: L^{p}(\mathbb{R};X_{0})\rightarrow L^{p}(\mathbb{R};X_{0})$, given by
$$
(Hu)(t)=\frac{1}{\pi}PV\int_{\mathbb{R}}\frac{1}{t-s}u(s)ds,
$$
is a bounded map.
\end{definition}

According to a
well-known theorem, Banach spaces of class $\mathcal{HT}$ coincide with the class of all Banach spaces satisfying the
{\em unconditional martingale difference property} (for short UMD spaces), see e.g. \cite[Chapter III, Section 4.4]{Am}.

\begin{remark} Let us make some comments about the operator class $\mathcal{BIP}$ and the class of UMD spaces:
\begin{enumerate}[\rm (a)]
\item
For each $\theta\in(0,\pi)$ and $\phi\in(\pi-\theta,\pi)$, the following inclusion holds
$$\mathcal{H}^{\infty}(\theta)\subset \mathcal{BIP}(\phi).$$
\item
There is an abundance of Banach spaces having the UMD property. For example, Hilbert spaces, $L^p$ spaces with $p\in(1,\infty)$, Bessel potential spaces $H_{p}^{s}$ with $p\in(1,\infty)$ and $s\in\mathbb{R}$, and real (or even complex) interpolation spaces of UMD spaces, have this property. For more details see
\cite[Chapter 4]{Hy}. In the case where the underlying space $X_{0}$ is UMD, the following property holds true
$$\mathcal{BIP}(\phi)\subset \mathcal{R}(\pi-\phi),$$
see e.g. \cite[Theorem 4]{ClPr}.
\end{enumerate}
\end{remark}

The following classical result of Dore and Venni \cite[Theorem 3.2]{DorVe} provides a necessary condition for an operator in a UMD Banach space to satisfy the maximal $L^{q}$-regularity property. 

\begin{theorem}[Dore and Venni]\label{DoVe}
Let $X_{0}$ be a UMD Banach space and let $A\in\mathcal{BIP}(\phi)$ in $X_{0}$ with $\phi<\pi/2$. Then $A$ has maximal $L^{q}$-regularity. 
\end{theorem}

Kalton and Weis \cite[Theorem 6.5]{KW1} improved the result of Dore and Venni to the following.

\begin{theorem}[Kalton and Weis]\label{KalWeiTh}
Let $X_{0}$ be a UMD Banach space and let $A\in\mathcal{R}(\theta)$ in $X_{0}$ with $\theta>\pi/2$. Then $A$ has maximal $L^{q}$-regularity. 
\end{theorem}
We point out that the maximal $L^{q}$-regularity is characterized by the $\mathcal{R}$-sectoriality in UMD spaces; see \cite[Theorem 4.2]{Weis}.

Let us see how the above property is applied to quasilinear equations. Let $q\in(1,\infty)$, $T_{0}\in(0,\infty)$, $U$ be an open subset of the
interpolation space $(X_{1},X_{0})_{{1}/{q},q}$ and let $A(\cdot): U\rightarrow \mathcal{L}(X_{1},X_{0})$, $F(\cdot\,,\cdot): U\times [0,T_{0}]\rightarrow X_{0}$
be two (possibly non-linear) maps. Consider the problem
\begin{eqnarray}\label{aqpp1}
u_{t}(t)+A(u(t))u(t)&=&F(u(t),t)+G(t),\quad t\in(0,T),\\\label{aqpp2}
u(0)&=&u_{0},
\end{eqnarray}
where $T\in(0,T_{0})$, $u_{0}\in U$ and $G\in L^{q}(0,T_{0};X_{0})$. One important fact is that maximal $L^q$-regularity for the linearization $A(u_{0})$ 
together with appropriate Lipschitz continuity conditions imply existence and uniqueness of a short time solution to \eqref{aqpp1}-\eqref{aqpp2}.

\begin{definition}
Let $Y$ and $Z$ be two Banach spaces and $U\subseteq Y$. We say that a map 
$B(\cdot):U\to Z$ is {\em locally Lipschitz} if, for any $u_1,u_2\in U$, it holds
$$
\|B(u_1)-B(u_2)\|_{Z}\le C\|u_1-u_2\|_{Y},
$$
where $C$ is a constant depending on $B$, $U$ and $q$. The space of such maps is denoted by $C^{1-}(U;Z).$
\end{definition}

Now we are ready to state a result of Cl\'ement and Li \cite[Theorem 2.1]{CL} for short time existence, uniqueness, and maximal $L^q$-regularity for solutions
of quasilinear parabolic equations.

\begin{theorem}[Cl\'ement and Li]\label{ClementLi}
Suppose that $A$ and $F$ in \eqref{aqpp1} satisfy the following conditions:
\begin{enumerate}[\rm (H1)]
\item $A(\cdot)\in C^{1-}(U;\mathcal{L}(X_{1},X_{0}))$.
\item $F(\cdot\,,\cdot)\in C^{1-}(U\times [0,T_{0}];X_{0})$.
\item $A(u_{0})$ has maximal $L^q$-regularity.
\end{enumerate}
Then, there exists a $T\in(0,T_{0})$ and a unique
$$u\in H_{q}^{1}(0,T;X_{0})\cap L^{q}(0,T;X_{1})\hookrightarrow C\big([0,T];(X_{1},X_{0})_{{1}/{q},q}\big)$$
solving \eqref{aqpp1}-\eqref{aqpp2}. 
\end{theorem}

For a continuous maximal regularity approach, we refer to \cite{Am2}.

\begin{remark}
Suppose that the conditions (H1), (H2), (H3) of Theorem \ref{ClementLi} are satisfied, and let $u$, $T$, $T_{0}$ be as in Theorem \ref{ClementLi}. Denote by $T_{\max}$ the supremum of all such $T$. If
$$
T_{\max}<T_{0} \quad \text{and} \quad \|u\|_{H_{q}^{1}(0,T_{\max};X_{0})\cap L^{q}(0,T_{\max};X_{1})}<\infty,
$$ 
then, by \eqref{interpemb}, the solution $u$ extends up to $T_{\max}$ as an element $u_{T_{\max}}\in (X_{1},X_{0})_{{1}/{q},q}$. In this case, there is no open set $V\subset (X_{1},X_{0})_{{1}/{q},q}$ such that $u_{T_{\max}}\in V$ and the conditions (H1), (H2), (H3) of Theorem \ref{ClementLi} are satisfied by replacing $u_{0}$ with $u_{T_{\max}}$ and $U$ with $V$.
\end{remark}

\section{Cone differential operators}\label{sec5}
In this section, we review basic facts from the theory of {\em cone differential operators} or {\em Fuchs type operators}. We focus on the theory of Schulze's cone 
calculus, toward the direction of regularity theory for PDEs; for further details we refer to \cite{GKM2,GKM,GM,Le,
SS1,SS,Schu2,Schu,Schu3,Sei}.
Since our problem concerns evolution of curves, we restrict 
ourselves and adapt the situation to the one-dimensional case. However, most of the results hold in any dimension.

\subsection{Cone operators}\label{secconebas}
Let us denote by $\mathbb{B}$ the interval $[0,1]$.

\begin{definition}
A {\em cone differential operator} $A$ of order $\mu\in\mathbb{N}\cup\{0\}$ is an $\mu$-th order differential operator with smooth $\mathbb{C}$-valued coefficients on $(0,1)$
such that near the boundary points $j\in\{0,1\}$ it has the form 
\begin{equation}\label{Aop}
A_{j}=((-1)^{j}(x-j))^{-\mu}\sum_{k=0}^{\mu}a_{k}(x)((-1)^{j+1}(x-j)\partial_{x})^{k},
\end{equation}
where each $a_{k}$ is $C^{\infty}$-smooth up to $j$.
\end{definition}
The usual {\em homogeneous principal pseudodifferential symbol} $\sigma_\psi(A)$ is given by
$$
\sigma_\psi(A)(x,\zeta)=a_\mu(x)(-i\zeta)^\mu,
$$
for any $x\in(0,1)$ and $\zeta\in\mathbb{R}$.
Beyond its usual homogenous principal symbol, we define the {\em principal rescaled symbol} $\widetilde{\sigma}_{\psi}(A)$ by 
$$
\widetilde{\sigma}_{\psi}(A)(j,\zeta)=a_{\mu}(j)(-i\zeta)^{\mu},
$$
where $j\in\{0,1\}$ and $\zeta\in\mathbb{R}$.
Moreover, the {\em conormal symbol} $\sigma_M(A)$ of $A$ is defined by the following polynomial
$$
\sigma_{M}(A)(j,\lambda)=\sum_{k=0}^{\mu}a_{k}(j)\lambda^{k},
$$
where $j\in\{0,1\}$ and $\lambda\in\mathbb{C}$. The notion of ellipticity is extended to the case of conically degenerate differential operators as follows. 

\begin{definition}
A cone differential operator $A$ is called {\em $\mathbb{B}$-elliptic} if $\sigma_{\psi}(A)$ and $\widetilde{\sigma}_{\psi}(A)$ are pointwise invertible.
\end{definition}

Let $\omega\in C^{\infty}(\mathbb{B})$ be a fixed cut-off function, which is equal to one near the boundary points of $\mathbb{B}$
and zero away from them. Decompose $\omega$ as $\omega=\omega_{0}+\omega_{1}$, where $\omega_{j}$ is supported near $j\in\{0,1\}$. Denote by $C_{c}^{\infty}$ the space of smooth compactly supported functions.

\begin{definition}[Mellin-Sobolev spaces]
For any $\gamma\in\mathbb{R}$ consider the map 
$M_{\gamma}: C_{c}^{\infty}(\mathbb{R}_{+})\rightarrow C_{c}^{\infty}(\mathbb{R})$ given by
$$(M_\gamma u)(x)=e^{\big(\gamma-\frac{1}{2}\big)x}u(e^{-x}). 
$$
Furthermore, for any $r\in\mathbb{R}$ and $p\in(1,\infty)$, let $\mathcal{H}^{r,\gamma}_p(\mathbb{B})$ be the space of all distributions $u$ on $(0,1)$ such that 
$$
\|u\|_{\mathcal{H}^{r,\gamma}_p(\mathbb{B})}=\|M_{\gamma}(\omega u)\|_{H^{r}_p(\mathbb{R})}+\|(1-\omega)u\|_{H^{r}_p(\mathbb{B})}<\infty.
$$
The space $\mathcal{H}^{r,\gamma}_p(\mathbb{B})$ is called {\em (weighted) Mellin-Sobolev space}. We denote $\mathcal{H}^{\infty,\gamma}_p(\mathbb{B})= \cap_{r>0}\mathcal{H}^{r,\gamma}_p(\mathbb{B})$.
\end{definition}

The space $\mathcal{H}^{r,\gamma}_p(\mathbb{B})$ is UMD and moreover, it is independent of the choice of the cut-off function $\omega$.
Equivalently, if $r\in \mathbb{N}\cup\{0\}$ then $\mathcal{H}^{r,\gamma}_p(\mathbb{B})$ is the space of all functions $u$ such that
$u\in H^{r}_{p,loc}(\mathbb{B}^{\circ})$ and
$$
|x-j|^{\frac{1}2-\gamma}((x-j)\partial_x)^{k}(\omega u) \in L^{p}\big(\mathbb{B}, (x-j)^{-1}dx\big),
$$
for any $j\in\{0,1\}$ and $k\le r$.
For each $n\in\mathbb{N}\cup\{0\}$, let us define the space
$$
\mathbb{S}_{\omega}^{n}=\{c_{0}\omega_{0}s^{n}+c_{1}\omega_{1}(1-s)^{n}\, |\, c_{0},c_{1}\in\mathbb{C}\},
$$ 
endowed with the norm $v\mapsto (|c_{1}|^{2}+|c_{2}|^{2})^{1/2}$. We denote $\mathbb{S}_{\omega}^{0}$ by $\mathbb{C}_{\omega}$ and $\mathbb{S}_{\omega}^{1}$ by $\mathbb{S}_{\omega}$. In the following lemma, we recall some embeddings and multiplicative properties of Mellin-Sobolev spaces; see \cite[Corollaries 2.8 and 2.9]{RS2}, \cite[Corollary 3.3]{RS3}, \cite[Lemma 5.2]{RS3} and \cite[Lemmas 6.2 and 6.3]{RS3}.

\begin{lemma}\label{abr}
\label{c0} Let $1<p,q<\infty$ and $r>1/p$. Then:
\begin{enumerate}[\rm (a)]
\item A function $u$ in $\mathcal{H}^{r,\mu}_p(\mathbb{B})$, $\mu\in\mathbb{R}$, is continuous on $(0,1)$ and close to each $j\in\{0,1\}$ satisfies 
\begin{eqnarray*}
|u(x)|\le c |x-j|^{\mu-1/2} \|u\|_{\mathcal{H}^{r,\mu}_p(\mathbb{B})},
\end{eqnarray*}
for a constant $c>0$ depending only on $\mathbb{B}$ and $p$. If $\mu\geq1/2$, then 
$$
\mathcal{H}^{r,\mu}_p(\mathbb{B})\hookrightarrow C(\mathbb{B}). 
$$
Moreover, if $u_{1}\in \mathcal{H}^{r,1/2}_p(\mathbb{B})$, $u_{2}\in \mathcal{H}^{r,\gamma}_p(\mathbb{B})$ and $\gamma\in\mathbb{R}$, then
$$
\|u_{1}u_{2}\|_{\mathcal{H}^{r,\gamma}_p(\mathbb{B})}\le C\|u_{1}\|_{\mathcal{H}^{r,1/2}_p(\mathbb{B})} \|u_{2}\|_{\mathcal{H}^{r,\gamma}_p(\mathbb{B})},
$$ 
for suitable $C>0$; in particular, up to the choice of an equivalent norm, the space $\mathcal{H}^{r,\gamma}_p(\mathbb{B})$ is a Banach algebra whenever $\gamma\ge 1/2$. 
\smallskip
\item Multiplication by an element in $\mathcal{H}^{\sigma+1/q,1/2}_q(\mathbb{B})$, $\sigma>0$, defines a bounded map on $\mathcal{H}^{\rho,\mu}_p(\mathbb{B})$, $\mu\in\mathbb{R}$, for each $\rho\in(-\sigma,\sigma)$.
\smallskip
\item If $\gamma> 1/2$ and $u\in \mathcal{H}_{p}^{r,\gamma}(\mathbb{B})\oplus{\mathbb C}_{\omega}$ is nowhere zero, then
$$1/u\in \mathcal{H}_{p}^{r,\gamma}(\mathbb{B})\oplus{\mathbb C}_{\omega},$$
that is
$\mathcal{H}_{p}^{r,\gamma}(\mathbb{B})\oplus{\mathbb C}_{\omega}$ is spectrally invariant in the space $C(\mathbb{B})$, and therefore closed under holomorphic functional calculus. In addition, if $U$ is a bounded open subset of $\mathcal{H}_{p}^{r,\gamma}(\mathbb{B})\oplus{\mathbb C}_{\omega}$ consisting of functions $v$ such that 
$\mathrm{Re}(v)\geq\alpha>0$ for some fixed $\alpha$, then the subset $\{1/v\, |\, v\in U\}$ of $\mathcal{H}_{p}^{r,\gamma}(\mathbb{B})\oplus{\mathbb C}_{\omega}$ is also bounded; its bound can be estimated by the bound of $U$ and by the constant $\alpha$. 
\smallskip
\item Let $\ell\in\mathbb{R}$ and $\eta\in(0,1)$. Then, the following embeddings hold
\begin{eqnarray*}
\lefteqn{\mathcal{H}^{\ell+2-2\eta+\varepsilon,\gamma+2-2\eta+\varepsilon}_p(\mathbb{B})\oplus\mathbb{C}_{\omega}}\\
&& \hookrightarrow(\mathcal{H}^{\ell+2,\gamma+2}_p(\mathbb{B})\oplus\mathbb{C}_{\omega},\mathcal{H}^{\ell,\gamma}_p(\mathbb{B}))_{\eta,q}\\
&&\hookrightarrow \mathcal{H}^{\ell+2-2\eta-\varepsilon,\gamma+2-2\eta-\varepsilon}_p(\mathbb{B})\oplus\mathbb{C}_{\omega},
\end{eqnarray*}
for every $\varepsilon>0$.
\end{enumerate}
\end{lemma}

\begin{lemma}\label{sumSn}
Let $n\in\mathbb{N}\cup\{0\}$, $1<p<\infty$, $r>1/p$ and $\gamma>n+1/2$. The space
$$
Z_{n}=\mathcal{H}^{r,\gamma}_{p}(\mathbb{B})\oplus\bigoplus_{k=0}^{n}\mathbb{S}_{\omega}^{k}
$$
is a Banach algebra, spectrally invariant in the space $C(\mathbb{B})$ and closed under holomorphic functional calculus. In addition, if $U$ is a bounded open subset of $Z_{n}$ consisting of functions $v$ such that $\mathrm{Re}(v)\geq\alpha>0$ for some fixed $\alpha$, then the subset $\{1/v\, |\, v\in U\}$ of $Z_{n}$ is also bounded; its bound can be estimated by the bound of $U$ and by the constant $\alpha$. 
\end{lemma}
\begin{proof}
Due to Lemma \ref{abr} (a) and the fact that multiplication by an element of $\mathbb{S}_{\omega}^{k}$, $k\in\mathbb{N}\cup\{0\}$, induces a bounded map from $\mathcal{H}^{r,\gamma}_{p}(\mathbb{B})$ to $\mathcal{H}^{r,\gamma+k}_{p}(\mathbb{B})$, it follows that the space $Z_{n}\hookrightarrow C(\mathbb{B})$ is a Banach algebra, for any $n\in\mathbb{N}\cup\{0\}$. To show the spectrally invariance property in $C(\mathbb{B})$ and the closedness under holomorphic functional calculus, we proceed by induction in $n$. For $n=0$, the statement is true due to Lemma \ref{abr} (c). Assume that the result holds for some $n\in\mathbb{N}\cup\{0\}$. Let $u\in Z_{n+1}$ be a nowhere zero function and 
$$
m_{j}=\{\min k\in\mathbb{N}\, | \, \text{the $\mathbb{S}_{\omega}^{k}$-term of $u$ is not equal to zero at $j$}\}. 
$$
On the support of $\omega_{j}$ we may write
$$u=c_{j}\omega_{j}+|x-j|^{m_{j}}h_{j},$$
where $c_{j}\in \mathbb{C}$ and 
$$
\omega_{j}h_{j}\in \mathcal{H}^{r,\gamma-m_{j}}_{p}(\mathbb{B})\oplus\bigoplus_{k=0}^{n-m_{j}}\mathbb{S}_{\omega}^{k}.
$$
By Lemma \ref{abr} (a), we obtain
$$u(j)=c_{j}\neq0.$$ 
Let $\{\phi_{k}\}_{k\in\{1,\dots,N\}}$ be a partition of unity on $\mathbb{B}$. Let $\widetilde{c}_{j}\omega_{j}\neq0$ be the $\mathbb{S}_{\omega}^{0}$-component of $\omega_{j}h_{j}$. Again by Lemma \ref{abr} (a), it follows that $h_{j}-\widetilde{c}_{j}\omega_{j}$ tends to zero as $x$ tends to $j$. Hence, there exists $\delta>0$ such that $h_{0}$ is pointwise invertible in $[0,\delta)$ and $h_{1}$ is pointwise invertible in $(1-\delta,1]$. Assume that $\phi_{1}$ is supported on $[0,\delta)$ and $\phi_{N}$ on $(1-\delta,1]$. On the support of $\phi_{k}$, $k\in\{0,N\}$, we have
$$
u^{-1}=(c_{j}h_{j}^{-1}+|x-j|^{m_{j}})^{-1}h_{j}^{-1}.
$$
By the induction hypothesis, 
$$
\phi_{k}\{h_{j}^{-1}, (c_{j}h_{j}^{-1}+|x-j|^{m_{j}})^{-1}\}\in \mathcal{H}^{r,\gamma-m_{j}}_{p}(\mathbb{B})\oplus\mathbb{S}_{\omega}^{m_{j}}\oplus\bigoplus_{k=0}^{n-m_{j}}\mathbb{S}_{\omega}^{k},
$$
which implies that $\phi_{k}u^{-1}\in Z_{n+1}$, for $k\in\{0,N\}$. In the interior, by Lemma \ref{abr} (c), we have
$$
\phi_{k}u^{-1}\in \mathcal{H}^{r,\gamma}_{p}(\mathbb{B})\hookrightarrow Z_{n+1},\quad k\in\{1\dots,N-1\}.
 $$
Therefore, from
$$
1=u\sum_{k=0}^{N}\phi_{k}u^{-1},
$$
we obtain that $u^{-1}\in Z_{n+1}$. The closedness under holomorphic functional calculus follows by the formula
$$
f(u)=\frac{1}{2\pi i}\int_{\Gamma}f(-\lambda)(u+\lambda)^{-1}d\lambda,
$$
where $\Gamma$ is a closed simple path around $\mathrm{Ran}(-u)$, within the holomorphicity area of $f$. The boundedness of the set $\{1/u\, |\, u\in U\}$ follows by the above construction.
\end{proof}

Cone differential operators act naturally on scales of weighted Mellin-Sobolev spaces, i.e. such an operator $A$ of order $\mu$ induces a bounded map
$$
A: \mathcal{H}^{r+\mu,\gamma+\mu}_p(\mathbb{B}) \rightarrow \mathcal{H}^{r,\gamma}_p(\mathbb{B}),
$$
for any $p\in(1,\infty)$ and $r,\gamma\in\mathbb{R}$.
However, we will regard $A$ as an unbounded operator in $\mathcal{H}^{r,\gamma}_p(\mathbb{B})$, $p\in(1,\infty)$, $r,\gamma\in\mathbb{R}$, with domain $C_{c}^{\infty}(\mathbb{B}^{\circ})$. If $A$ is $\mathbb{B}$-elliptic, the domain of its minimal extension $\underline{A}_{\min}$ (that is the closure of $A$) is given by 
$$
\mathcal{D}(\underline{A}_{\min})=\Big\{u\in {\bigcap}_{\varepsilon>0}\mathcal{H}^{r+\mu,\gamma+\mu-\varepsilon}_p(\mathbb{B}) \, |\, Au\in \mathcal{H}^{r,\gamma}_p(\mathbb{B})\Big\}.
$$ 
If the conormal symbol of $A$ has no zeros on the line
\begin{equation}\label{linedom}
\Big\{\lambda\in\mathbb{C}\,|\, \mathrm{Re}(\lambda)= 1/2-\gamma-\mu\Big\},
\end{equation}
then we have
$$
\mathcal{D}(\underline{A}_{\min})=\mathcal{H}^{r+\mu,\gamma+\mu}_p(\mathbb{B}).
$$ 
The domain $\mathcal{D}(\underline{A}_{\max})$ of the maximal extension $\underline{A}_{\max}$ of $A$ is defined by
\begin{equation*}
\mathcal{D}(\underline{A}_{\max})=\Big\{u\in\mathcal{H}^{r,\gamma}_p(\mathbb{B}) \, |\, Au\in \mathcal{H}^{r,\gamma}_p(\mathbb{B})\Big\}.
\end{equation*}
It turns out that $\mathcal{D}(\underline{A}_{\max})$ and $\mathcal{D}(\underline{A}_{\min})$ differ by an $r,p$-independent finite dimensional space $ \mathcal{E}_{A,\gamma}$,
i.e.
\begin{equation*}\label{dmax1}
\mathcal{D}(\underline{A}_{\max})=\mathcal{D}(\underline{A}_{\min})\oplus\mathcal{E}_{A,\gamma}.
\end{equation*}
The space $\mathcal{E}_{A,\gamma}$ is called {\em asymptotics space} and it consists of linear combinations of $C^{\infty}(\mathbb{B}^\circ)$-functions, that 
around each boundary point $j\in\{0,1\}$, are of the form
$$\omega(x)|x-j|^{-\rho}\log^{\eta}(|x-j|),$$
where $\eta\in\mathbb{N}\cup\{0\}$ and $\rho\in \mathbb{C}$ satisfies
$$
\mathrm{Re}(\rho)\in I_{\gamma}=[1/2-\gamma-\mu,1/2-\gamma).
$$

If the coefficients $a_k$ in \eqref{Aop} close to each boundary point $j\in\{0,1\}$ are constant, then we have precisely
\begin{equation}\label{Espacesharp}
\mathcal{E}_{A,\gamma}=\bigoplus_{j\in\{0,1\}}\mathcal{E}_{A,\gamma,j}= \bigoplus_{j\in\{0,1\}}\bigoplus_{\rho\in I_{\gamma}}\mathcal{E}_{A,\gamma,j,\rho}.
\end{equation}
The index $\rho$ in \eqref{Espacesharp} runs through all possible zeros of $\sigma_{M}(A)(j,\cdot)$ in $I_{\gamma}$. Furthermore, for each such $\rho$, the space $\mathcal{E}_{A,\gamma,j,\rho}$ consists of functions that are zero, at points where $\omega_{j}$ is zero, and on the support of $\omega_{j}$ are of the form
$$c\,\omega_{j}(x)|x-j|^{-\rho}\log^{\eta}(|x-j|).$$
Here $c\in\mathbb{C}$, $\eta\in\{0,\dots,m_{\rho}\}$ and $m_{\rho}\in\mathbb{N}\cup\{0\}$ is the multiplicity of $\rho$.
All possible closed extensions $\underline{A}_s$ of $A$, also called {\em realizations}, correspond to subspaces
$\underline{\mathcal{E}}_{A,\gamma}$ of $\mathcal{E}_{A,\gamma}$.

We associate to $A$ the {\em model cone operators} $A_{j,\wedge}$, $j\in\{0,1\}$, defined by 
$$
A_{j,\wedge}=x^{-\mu}\sum_{k=0}^{\mu}a_{k}(j)(-x\partial_{x})^{k},
$$
which in our setting act on smooth compactly supported functions on the half-line
$\mathbb{R}_{+}=[0,+\infty).$

\begin{definition}
For any $p\in(1,\infty)$ and $r,\gamma\in\mathbb{R}$ define $\mathcal{K}_{p}^{r,\gamma}(\mathbb{R}_{+})$ to be the space of all distributions $u$ satisfying $\omega_{0} u\in \mathcal{H}^{r,\gamma}_p(\mathbb{B})$ and $(1-\omega_{0})u\in H^{r}_{p}(\mathbb{R})$, where $(1-\omega_{0})u$ is naturally extended everywhere on $\mathbb{R}$ by zero and one.
\end{definition}

The operators $A_{j,\wedge}$ act naturally on scales of Sobolev spaces $\mathcal{K}_{p}^{r,\gamma}(\mathbb{R}_{+})$, i.e. 
$$
A_{j,\wedge}\in\mathcal{L}(\mathcal{K}_{p}^{r+\mu,\gamma+\mu}(\mathbb{R}_{+}),\mathcal{K}_{p}^{r,\gamma}(\mathbb{R}_{+})), 
$$
for any $p\in(1,\infty)$ and $r,\gamma\in\mathbb{R}$. In the same way as above, each
$$A_{j,\wedge}: C_{c}^{\infty}(\mathbb{R}_{+})\rightarrow \mathcal{K}_{p}^{r,\gamma}(\mathbb{R}_{+}),$$
considered now as an unbounded operator, admits several closed extensions. More precisely,
\begin{equation}\label{dommodcone1}
\mathcal{D}(\underline{A}_{j,\wedge,\max})=\mathcal{D}(\underline{A}_{j,\wedge,\min})\oplus\mathcal{F}_{A_{j,\wedge},\gamma}.
\end{equation}
Here,
$$
\mathcal{D}(\underline{A}_{j,\wedge,\min})=\big\{u\in {\bigcap}_{\varepsilon>0}\mathcal{K}^{r+\mu,\gamma+\mu-\varepsilon}_p(\mathbb{R}_{+}) \, |\, Au\in \mathcal{K}^{r,\gamma}_p(\mathbb{R}_{+})\big\}
$$ 
and
\begin{equation}\label{dommodcone2}
\mathcal{D}(\underline{A}_{j,\wedge,\min})=\mathcal{K}^{r+\mu,\gamma+\mu}_p(\mathbb{R}_{+})
\end{equation}
if and only if $\sigma_{M}(A)(j,\cdot)$ has no zero on the line \eqref{linedom}. In addition, $\mathcal{F}_{A_{j,\wedge},\gamma}$ coincides with
$\mathcal{E}_{A,\gamma,j}$ in \eqref{Espacesharp} by replacing $\omega_{j}$ with a fixed cut-off function
$\omega_{\wedge}\in C^{\infty}(\mathbb{R}_{+})$ near zero. In fact, there exists a natural isomorphism 
\begin{equation}\label{isomthheta}
\Theta: \mathcal{E}_{A,\gamma}\rightarrow {\bigoplus}_{j\in\{0,1\}}\mathcal{F}_{A_{j,\wedge},\gamma},
\end{equation}
see \cite[Theorem 4.7]{GKM2} for details. Clearly, each closed extension $\underline{A}_{j,\wedge,r}$ of $A_{j,\wedge}$ is obtained by a particular choice of a subspace of $\mathcal{F}_{A_{j,\wedge},\gamma}$.

Let $p\in(1,\infty)$, $r,\gamma\in\mathbb{R}$ and let $\underline{A}_{r}$ be the closed extension of a $\mathbb{B}$-elliptic operator $A$ in $\mathcal{H}^{r,\gamma}_p(\mathbb{B})$ with domain 
\begin{equation}\label{closedextofA}
\mathcal{D}(\underline{A}_{r})=\mathcal{D}(\underline{A}_{\min})\oplus\underline{\mathcal{E}}_{A,\gamma}, 
\end{equation}
where $\underline{\mathcal{E}}_{A,\gamma}$ is a subspace of $\mathcal{E}_{A,\gamma}$. For each $\theta\in[0,\pi)$, consider the following ellipticity conditions for the operator $\underline{A}_{r}$:
\begin{enumerate}[\rm (E1)]
\item Both $\sigma_{\psi}(A)+\lambda$ and $\widetilde{\sigma}_{\psi}(A)+\lambda$ are invertible for each $\lambda\in \mathbb{C}\backslash S_{\theta}$, where $S_{\theta}$ is the sector in Definition \ref{defsec}.
\medskip
\item For each $j\in\{0,1\}$ the polynomial $\sigma_{M}(A)(j,\cdot)$ has no zeroes on
$$\big\{\lambda\in\mathbb{C}\, |\,\mathrm{Re}(\lambda)=1/2-\gamma-\mu\,\, \text{ or}\,\,\, \mathrm{Re}(\lambda)=1/2-\gamma\big\}.$$
\item There exists a constant $R>0$ such that, for each $j\in\{0,1\}$, the resolvent $(\lambda-\underline{A}_{j,\wedge,0})^{-1}$ is defined for all
$\lambda\in S_{\theta}$ with $|\lambda|\geq R$, where 
$$
\mathcal{D}(\underline{A}_{j,\wedge,0})=\mathcal{K}^{\mu,\gamma+\mu}_{2}(\mathbb{R}_{+})\oplus \Theta\big(\underline{\mathcal{E}}_{A,\gamma}\big).
$$
Additionally, we assume that $\mathcal{D}(\underline{A}_{j,\wedge,0})$ is invariant under dilations, i.e. if $u\in \Theta(\underline{\mathcal{E}}_{A,\gamma})$ then for each $r>0$,
$$x\mapsto u(rx)\in \Theta\big(\underline{\mathcal{E}}_{A,\gamma}\big).$$
\end{enumerate}

The above mentioned condition (E2) is technical and guarantees that the minimal domain is equal to a Mellin-Sobolev space. The condition (E1) and a weaker version of (E3) were used in \cite[Theorems 6.9 and 6.36]{GKM} and \cite[Theorem 9.1]{GKM2} in order to show sectoriality for $\mathbb{B}$-elliptic cone operators. The same result was extended to boundary value problems; see \cite[Theorems 8.1 and 8.26]{Kra}. Furthermore, in \cite[Theorem 2]{CSS}, \cite[Theorem 3.3]{RS2} and \cite[Theorem 4.3]{SS} it was shown that (E1) and (E3) imply even bounded imaginary powers for $\underline{A}_{r}$. For our purposes, we will use the following improvement.

\begin{theorem}\label{hinftyth}
Let $\theta\in[0,\pi)$, $p\in(1,\infty)$, $r,\gamma\in\mathbb{R}$ and let $\underline{A}_{r}$ be the closed extension of $A$ in $\mathcal{H}^{r,\gamma}_p(\mathbb{B})$ with domain \eqref{closedextofA}. If the conditions {\em (E1)}, {\em(E2)}, {\em(E3)} are satisfied, then there exists $c>0$ such that
$c-\underline{A}_{r}\in \mathcal{H}^{\infty}(\theta).$
\end{theorem}
\begin{remark}
The above result is proved in \cite[Theorem 5.2]{SS1}. We point out that, due to \cite[Lemma 2.5 and (7.1)]{GKM2}, the ellipticity condition (E3) in \cite[p. 1405]{SS1} is equivalent to (E3) above. 
\end{remark}

\section{Short time existence and regularity of the flow}\label{sec6}
In this section, we are going to investigate evolution by curvature flow of curves, lying in singular Riemann surfaces and joining two singularities.

\subsection{Short time existence}\label{assumptongama}
Let $\gamma:\mathbb{B}\to(\Sigma,\gind)$ be a regular curve such that $\gamma(0)$ and $\gamma(1)$ are conformal conical points of orders $\beta_0$ and $\beta_1$, respectively. Assume that:
\begin{enumerate}[\rm(a)]
\item The orders satisfy $-1<\beta_{0}\leq \beta_{1}<0$.
\medskip
\item The image $\gamma((0,1))$ does not meet any singular point of $(\Sigma,\gind)$.
\medskip
\item
There exist conformal coordinate charts $(U_{j},\psi_{j})$ around the points $\gamma(j)$, $j\in\{0,1\}$, and $\varepsilon>0$ such
that the curves $\psi_0\circ\gamma:(0,\varepsilon)\to(\C,|dz|^2)$
and $\psi_1\circ\gamma:(1-\varepsilon,1)\to(\C,|dz|^2)$ are parametrized by arc-length, 
$(\psi_{j}\circ\gamma)(j)=(0,0)$ and $(\psi_{j}\circ \gamma)_{s}(j)=(1,0).$ Moreover, we assume that 
\begin{equation}\label{mainassumption}
\omega_{j}(\psi_{j}\circ\gamma)\in \mathcal{H}_{p}^{b,\alpha}(\mathbb{B})\oplus\mathbb{S}_{\omega}
\end{equation}
for certain $p\in(1,\infty)$, where
$$b\in[4,\infty)\cup\{\infty\}\,\,\text{and}\,\,\alpha> \max\{3/2-\beta_{0},5/2+\beta_{1}\}>3/2,$$
and $\omega_{j}:\mathbb{B}\rightarrow \mathbb{R}$ is a cut-off function near $j\in\{0,1\}$.
\medskip
\item For any $s_0\in\mathbb{B}^{\rm o}$ and a coordinate chart $(U_{s_0},\psi_{s_0})$ around the point $\gamma(s_0)\in\Sigma$, it holds
$\psi_{s_0}\circ\gamma\in{H}^b_{p,loc}(\mathbb{B}^{\rm o})$.
\end{enumerate}
A regular curve $\gamma\in C^{\infty}(\mathbb{B};\Sigma)$ satisfy the regularity conditions (a)-(d). To simplify the notation, from now on we omit compositions with the maps $\psi_{j}$.

We will consider the evolution equation
\begin{equation}\label{Ioannina1}
(\varGamma_{t})^{\perp}=\bf k_{\gind},
\end{equation}
where $(\cdot)^{\perp}$ stands for the normal projection with respect to $\gind$ on the normal space of $\varGamma$. As in Section \ref{CRSSM}, we
may represent the evolved curves as graphs over $\gamma$ of a time-dependent function $w$ on $\mathbb{B}$. 
In the chart $(U_{j},\psi_{j})$, the flow is represented in the form
\begin{equation}\label{gflow}
\sigma(s,t)=\gamma(s)+w(s,t)\xi(s),
\end{equation}
where $\xi$ is the unit normal along $\gamma$ with respect to the euclidean metric.
Denote by $\xi^{\gind}_\sigma$ the unit normal with respect to $\gind$ of $\sigma$. From \eqref{s4}, we get that
$$
(\sigma_t)^{\perp}=(w_t\xi)^\perp=w_t\gind(\xi,\xi^{\gind}_{\sigma})\xi^{\gind}_{\sigma}=w_t\langle\xi,\xi_{\sigma}\rangle\xi_{\sigma}
=\frac{1-kw}{\sqrt{(1-kw)^2+w^2_s}} w_t\,\xi_\sigma.
$$
From the equations \eqref{s1}, \eqref{s2}, \eqref{s3}, and \eqref{s5} we deduce that the curve
$\varGamma$ is evolving by \eqref{Ioannina1} if and only if $w$ satisfies a quasilinear parabolic equation, which near the endpoints has the following degenerate form 
\begin{eqnarray}\label{kwer1}
\lefteqn{w_t=|\gamma+w\xi|^{-2\beta_j}e^{-2h}\Big(\frac{w_{ss}}{(1-kw)^2+w_s^2}+\frac{w_s\langle Dh,\gamma_s\rangle}{1-kw}
-\langle Dh,\xi\rangle}\\
&&\hspace{-20pt}+\frac{2kw^2_s+k_sw_sw+k(1-kw)^2}{(1-kw)\big((1-kw)^2+w_s^2\big)}
+\beta_j\frac{w_s\langle\gamma,\gamma_s\rangle-(1-kw)\langle\gamma,\xi\rangle-w(1-kw)}{|\gamma+w\xi|^2(1-kw)}\Big)\nonumber
\end{eqnarray}
with initial data
\begin{equation}\label{kwer2}
w_0(\cdot)=w(\cdot, 0)=0.
\end{equation}

Since at the charts we have $\gamma(j)=(0,0)$ and $\gamma_s(j)=(1,0)$, we may write $\gamma$ in the form
$$\gamma=|s-j|\varphi$$
for values of $s$ close to $j$, where $\varphi(j)=(1,0)$. 
Denote by $\gamma_{1}$, $\gamma_{2}$, $\varphi_{1}$, $\varphi_{2}$, $\xi_{1}$ and $\xi_{2}$ the components of $\gamma$, $\varphi$ and $\xi$ respectively. Due to \eqref{mainassumption}, we have
\begin{equation}\label{inreg1}
\omega\{\varphi_{1},\varphi_{2}\}\in \mathcal{H}_{p}^{b,\alpha-1}(\mathbb{B})\oplus\mathbb{C}_{\omega} \quad \text{and} \quad \omega\{\xi_{1},\xi_{2}\}\in \mathcal{H}_{p}^{b-1,\alpha-1}(\mathbb{B})\oplus\mathbb{C}_{\omega}.
\end{equation}
Recall that $k=\langle \gamma_{ss},\xi\rangle$. From the computations in the proof of Lemma \ref{lemm ana p} and Lemma \ref{abr} (a) and (b), we have
\begin{equation}\label{inreg2}
\omega k\in \mathcal{H}_{p}^{b-2,\alpha-2}(\mathbb{B}), \quad \omega g\in \mathcal{H}_{p}^{b,\alpha-2}(\mathbb{B}), \quad \text{and} \quad \omega\rho\in \mathcal{H}_{p}^{b-1,\alpha-2}(\mathbb{B}),
\end{equation}
where the functions $g$, $\rho$ are defined for values of $s$ close to $j$ and given by 
\begin{equation}\label{greqexp}
g=|s-j|^{-1}(1-|\varphi|^2) \quad \text{and} \quad \rho=|s-j|^{-2}\langle\gamma,\xi\rangle.
\end{equation}
Let 
\begin{equation}\label{nalpha}
n=\max\{m\in \mathbb{N}\cup\{0\}\, |\, m<\alpha-3/2\}.
\end{equation}
By the analyticity of $h$ and Lemma \ref{sumSn}, there exists an $R>0$ such that for each
$$\lambda\in\{z\in\mathbb{C}\, |\, |z|<R\}$$
it holds
\begin{equation}\label{inreg3}
\omega\{\vartheta(\cdot,\lambda),\zeta_{2}(\cdot,\lambda),\zeta_{2}(\cdot,\lambda)\}\in \mathcal{H}_{p}^{b-1,\alpha-1}(\mathbb{B})\oplus\bigoplus_{m=0}^{n}\mathbb{S}_{\omega}^{m},
\end{equation}
where
$$
\vartheta(\cdot,\lambda)=h(\gamma(\cdot)+\lambda\xi(\cdot))
$$
and $\zeta_{1}$, $\zeta_{2}$ are the components of
$$
\zeta(\cdot,\lambda)=Dh(\gamma(\cdot)+\lambda\xi(\cdot)).
$$

Let $\varepsilon\in(0,{1}/{2})$ and let $d:[0,1]\rightarrow [0,1]$
be a $C^{\infty}$-function such that
\begin{equation}
d(s)=\left\{
\begin{array}{ll}
s&\quad \text{on} \quad [0,\varepsilon), \\
\ge\min\{\varepsilon,1-\varepsilon\}&\quad \text{on}\quad[\varepsilon,1-\varepsilon], \\
1-s &\quad \text{on}\quad(1-\varepsilon,1].
\end{array}
\right.\label{tangent}
\end{equation}
We claim that there is a solution of the form
$$w(s,t)=d(s)u(s,t)$$
for any $(s,t)\in\mathbb{B}\times [0,T)$, where $T>0$ and
$u$ is continuous in space-time and smooth for values $s\in(0,1)$.
If this holds, then we deduce immediately that the flow will stay fixed at each singular point.

The evolution equation of $w$, close to a boundary point $j\in\{0,1\}$, can be written equivalently in terms of $u$ in the form
\begin{eqnarray*}
u_t&=&|s-j|^{-2\beta_j-1}e^{-2\vartheta(|s-j|,|s-j|u)}|\varphi+u\xi|^{-2\beta_j}\bigg(\frac{(-1)^{j}\beta_j |\varphi|^2u_s}{|\varphi+u\xi|^{2}(1-|s-j|ku)}\\
&& +\frac{|s-j|u_{ss}+(-1)^{j}2u_s}{(1-|s-j|ku)^2+(u+(-1)^{j}|s-j|u_s)^2}-\frac{\beta_{j}\rho}{|\varphi+u\xi|^{2}}\\
&&+\frac{k(1-|s-j|ku)^{2}+2k(u+(-1)^{j}|s-j|u_s)^2}{(1-|s-j|ku)((1-|s-j|ku)^2+(u+(-1)^{j}|s-j|u_s)^2)}\\
&&+\frac{(-1)^{j}|s-j|k_s(u^2+(-1)^{j}|s-j|uu_s)}{(1-|s-j|ku)((1-|s-j|ku)^2+(u+(-1)^{j}|s-j|u_s)^2)}\\
&&+\beta_j\,\frac{-gu+(-1)^{j}\langle \varphi,\varphi_s\rangle u+ku^2+|s-j|\langle\varphi,\varphi_s\rangle u_s}{|\varphi+u\xi|^2(1-|s-j|ku)}-\langle\zeta,\xi\rangle\\
&&+\frac{(u+(-1)^{j}|s-j|u_s) \big(\langle\zeta,\varphi\rangle
+(-1)^{j}\langle\zeta,|s-j|\varphi_s\rangle\big)}{1-|s-j|ku}\bigg).
\end{eqnarray*}

Around each boundary point $j\in\{0,1\}$, we perform the following change 
$$
|s-j|^{1+\beta_{j}}=(1+\beta_{j})|x-j|.
$$
Then
\begin{equation}\label{xtosderivartves}
|s-j|\partial_ s=(1+\beta_{j})|x-j|\partial_x \quad \text{and} \quad (1+\beta_{j})|s-j|^{-1}ds=|x-j|^{-1}dx.
\end{equation}
For each $m\in\mathbb{N}\cup\{0\}$, define the space
$$
\mathbb{X}_{\omega}^{m}=\big\{c_{0}\omega_{0}x^{m/(1+\beta_{0})}+c_{1}\omega_{1}(1-x)^{m/(1+\beta_{1})}\, |\, c_{0},c_{1}\in\mathbb{C}\big\},
$$ 
endowed with the norm $v\mapsto (|c_{1}|^{2}+|c_{2}|^{2})^{1/2}$. Note that $\mathbb{X}_{\omega}^{0}=\mathbb{C}_{\omega}$. Moreover, let us introduce the space
$$
\mathbb{X}_{\omega}=\bigoplus_{m=0}^{n}\mathbb{X}_{\omega}^{m},
$$
where $n$ is the fixed natural number given in \eqref{nalpha}. Due to \eqref{inreg1}-\eqref{inreg3}, for each $\lambda\in\{z\in\mathbb{C}\, |\, |z|<R\}$ we have
 \begin{equation}\label{xregdata}
 \begin{cases}
\omega k\in \mathcal{H}_{p}^{b-2,\frac{1}{2}+\min_{j\in\{0,1\}}\{\frac{\alpha-5/2}{1+\beta_{j}}\}}(\mathbb{B}), \\
 \omega g\in \mathcal{H}_{p}^{b,\frac{1}{2}+\min_{j\in\{0,1\}}\{\frac{\alpha-5/2}{1+\beta_{j}}\}}(\mathbb{B}), \\
\omega\rho\in \mathcal{H}_{p}^{b-1,\frac{1}{2}+\min_{j\in\{0,1\}}\{\frac{\alpha-5/2}{1+\beta_{j}}\}}(\mathbb{B}),\\
\omega\{\varphi_{1},\varphi_{2}\}\in \mathcal{H}_{p}^{b,\frac{1}{2}+\frac{\alpha-3/2}{1+\beta_{1}}}(\mathbb{B})\oplus\mathbb{C}_{\omega} ,\\
 \omega\{\xi_{1},\xi_{2}\}\in \mathcal{H}_{p}^{b-1,\frac{1}{2}+\frac{\alpha-3/2}{1+\beta_{1}}}(\mathbb{B})\oplus\mathbb{C}_{\omega}, \\
\omega\{\vartheta(\cdot,\lambda),\zeta_{2}(\cdot,\lambda),\zeta_{2}(\cdot,\lambda)\}\in \mathcal{H}_{p}^{b-1,\frac{1}{2}+\frac{\alpha-3/2}{1+\beta_{1}}}(\mathbb{B})\oplus\mathbb{X}_{\omega},
\end{cases}
\end{equation}
where all spaces are now considered in $x$-variable. 

Let $\kappa:[0,1]\rightarrow [0,1]$ be a $C^{\infty}$-smooth function such that $\kappa=j$ on the support of the cut-off function $\omega_{j}$. After a straightforward computation, the evolution equation for the function $u$ can be written in the form
\begin{eqnarray}\label{flqz1}
u_t+A(u)u&=&F(u),\\\label{flqz2}
u(0)&=&0.
\end{eqnarray}
The leading term is expressed globally as 
\begin{eqnarray}\nonumber
A(v)u&=&\omega\big(Q_{3}(v)+Q_{1}(v)\beta_{\kappa}(1+\beta_{\kappa})^{-1}(-1)^{\kappa+1}|x-\kappa|^{-1}
\big)u_{x}\\\label{newexpA}
&&\quad\quad\quad\quad+Q_{1}(v)Au+(1-\omega)(Q_{2}(v)+Q_{3}(v))u_{x},
\end{eqnarray}
where
\begin{equation}\label{Alinear}
A=\partial_{x}^{2}+(-1)^{\kappa}2\omega |x-\kappa|^{-1}\partial_{x},
\end{equation}
and, on the support of $\omega$, we have
\begin{eqnarray*}
\eta(v)&=&\vartheta\big(((1+\beta_{j})|x-j|)^{\frac{1}{1+\beta_{j}}},((1+\beta_{j})|x-j|)^{\frac{1}{1+\beta_{j}}}v\big),\\
Q(v)&=&\big(1-k((1+\beta_{j})|x-j|)^{\frac{1}{1+\beta_{j}}}v\big)^2+\big(v+(-1)^{j}(1+\beta_{j})|x-j|v_x\big)^2,\\
Q_{1}(v)&=&-e^{-2\eta(v)}|\varphi+v\xi|^{-2\beta_{j}}(Q(v))^{-1},\\
Q_{2}(v)&=&(-1)^{j+1}(2+\beta_{j})\frac{e^{-2\eta_{j}(v)}|\varphi+v\xi|^{-2\beta_{j}}}{(1+\beta_{j})|x-j|Q(v)},\\
Q_{3}(v)&=&(-1)^{j+1}\frac{e^{-2\eta(v)}|\varphi|^{2}|\varphi+v\xi|^{-2(1+\beta_{j})}}{1-k((1+\beta_{j})|x-j|)^{\frac{1}{1+\beta_{j}}}v} \frac{\beta_{j}}{1+\beta_{j}}\frac{1}{|x-j|}.
\end{eqnarray*}
The term $F$ on the support of $\omega$ is given by
\begin{equation}\label{Fterm}
F(u)=G(u)\big(F_1(u)+F_2(u)+F_3(u)+F_4(u)+F_5(u)\big),
\end{equation}
where
$$
G(u)=\frac{e^{-2\eta_{j}(u)}|\varphi+u\xi|^{-2\beta_{j}}}{1-k((1+\beta_{j})|x-j|)^{\frac{1}{1+\beta_{j}}}u},
$$
and
\begin{eqnarray*}
F_{1}(u)&=&f_{1}(1-k((1+\beta_{j})|x-j|)^{{1}/(1+\beta_{j})}u),\\
F_{2}(u)&=&f_{2}|\varphi+u\xi|^{-2}(1-k((1+\beta_{j})|x-j|)^{1/(1+\beta_{j})}u),\\
F_{3}(u)&=&Q^{-1}(u)(f_{3,1}+f_{3,2}u+f_{3,3}u^{2}+f_{3,4}|x-j|uu_{x}+f_{3,5}(xu_{x})^{2}),\\
F_{4}(u)&=&|\varphi+u\xi|^{-2}(f_{4,1}u+f_{4,2}|x-j|u_{x}+f_{4,3}u^{2}),\\
F_{5}(u)&=&f_{5,1}u+f_{5,2}|x-j|u_{x}.
\end{eqnarray*}
Using Lemma \ref{abr} and \eqref{xregdata}, after long but straightforward computations, we obtain that
\begin{equation}\label{ftermsA}
\omega\{f_{1},f_{5,\cdot}\} \in \bigcap_{\varepsilon>0}\mathcal{H}_{p}^{b-1,\frac{1}{2}-\frac{1+2\beta_{1}}{1+\beta_{1}}-\varepsilon}(\mathbb{B})
\end{equation}
and
\begin{equation}\label{ftermsB}
\omega\{f_{2}, f_{3,\cdot},f_{4,\cdot}\}\in \mathcal{H}_{p}^{b-3,\frac{1}{2}+\min_{j\in\{0,1\}}\{\frac{\alpha-5/2}{1+\beta_{j}},\frac{\alpha-7/2-2\beta_{j}}{1+\beta_{j}}\}}(\mathbb{B}).
\end{equation}

The operator $A$ from \eqref{Alinear} is a second order cone differential operator. The homogeneous principal pseudodifferential symbol $\sigma_{\psi}(A)$ and the principal rescaled symbol $\widetilde{\sigma}_{\psi}(A)$ of $A$ are given by
\begin{equation}\label{symbols}
\sigma_{\psi}(A)(x,\tau)=\widetilde{\sigma}_{\psi}(A)(j,\tau)=-\tau^{2},
\end{equation}
where $x$ is close to $j\in\{0,1\}$ and $\tau\in\mathbb{R}$. Therefore $A$ is $\mathbb{B}$-elliptic. The conormal symbol $\sigma_{M}(A)$ of $A$ is given by 
$$
\sigma_{M}(A)(j,\lambda)=\lambda^{2}-\lambda, \quad \lambda\in\mathbb{C},
$$
which has two zeros, $0$ and $1$. Hence, if $\underline{A}_{r,\max}$ is the maximal extension of $A$ in $\mathcal{H}_{p}^{r,\ell}(\mathbb{B})$, where $r,\ell\in\mathbb{R}$ and $\ell\notin\{-{3}/{2},-{5}/{2}\}$, then
\begin{equation}\label{maxdomA}
\mathcal{D}(\underline{A}_{r,\max})=\mathcal{H}_{p}^{r+2,\ell+2}(\mathbb{B})\bigoplus_{\eta\in\{0,1\}\cap I_{\ell}^{\circ}}\mathcal{N}_{\eta},
\end{equation}
where
$I_{\ell}^{\circ}=(-{3}/{2}-\ell,{1}/{2}-\ell),$
$\mathcal{N}_{0}=\mathbb{C}_{\omega}$ and 
$$
\mathcal{N}_{1}=\{c_{1}\omega_{1}x^{-1}+c_{2}\omega_{2}(1-x)^{-1}\}.
$$ 
According to \eqref{maxdomA}, let $\underline{A}_{r}$ be the closed extension of $A$ in $X_{0}^{r}=\mathcal{H}_{p}^{r,\ell}(\mathbb{B})$ with domain 
\begin{equation}\label{domA}
\mathcal{D}(\underline{A}_{r})=X_{1}^{r}=\mathcal{H}_{p}^{r+2,\ell+2}(\mathbb{B})\oplus \mathbb{C}_{\omega}.
\end{equation}

\begin{lemma}\label{maxregA}
Let $r\in\mathbb{R}$ and $\ell\in(-1/2,1/2)$. Then, for each $c>0$ and $\theta\in[0,\pi)$ we have $c-\underline{A}_{r}\in \mathcal{H}^{\infty}(\theta)$; in particular, $c-\underline{A}_{r}\in \mathcal{R}(\theta)$.
\end{lemma}
\begin{proof}
We will show that the ellipticity conditions (E1), (E2) and (E3) in Section \ref{secconebas} are satisfied for the closed extension $\underline{A}_{r}$, and then the result will follow from Theorem \ref{hinftyth}. 

{\it Condition} (E1). By \eqref{symbols}, for each $\lambda\in \mathbb{C}\backslash S_{\theta}$, both $\sigma_{\psi}(A)+\lambda$ and $\widetilde{\sigma}_{\psi}(A)+\lambda$ are pointwise invertible. Hence, (E1) is satisfied. 

{\it Condition} (E2). This condition is fulfilled if $\ell\neq-1/2$ and $\ell\neq-5/2$, which is satisfied by our choice of $\ell$.

{\it Condition} (E3). For each $j\in\{0,1\}$ the model cone operators of $A$ are given by
$$
A_{j,\wedge}=A_{\wedge}=y^{-2}\big((y\partial_{y})^{2}+(y\partial_{y})\big):C_{c}^{\infty}(\mathbb{R}_{+})\rightarrow \mathcal{K}_{2}^{0,\ell}(\mathbb{R}_{+}).
$$
By \cite[(3.11)-(3.12)]{SS1}, the image of the $\omega_{j}$ component of $\mathbb{C}_{\omega}$ under the isomorphism $\Theta$ given in \eqref{isomthheta} is the set 
$$
\mathbb{C}_{\omega_{\wedge}}=\{c\omega_{1}\, |\, c\in\mathbb{C}\}. 
$$
Denote by $\underline{A}_{\wedge,0}$ the closed extension of the operator $A_{\wedge}$ in
$\mathcal{K}_{2}^{0,\ell}(\mathbb{R}_{+})$ with domain $\mathcal{K}_{2}^{2,\ell+2}(\mathbb{R}_{+})\oplus\mathbb{C}_{\omega_{\wedge}}$.
Take $\lambda\in\mathbb{C}\backslash(-\infty,0]$ and $u\in \mathcal{D}(\underline{A}_{\wedge,0})$ such that
$$(\lambda-\underline{A}_{\wedge,0})u=0.$$
This is equivalent to
$$yu_{yy}+2u_{y}-\lambda yu=0.$$
Therefore,
$$
u=y^{-1}(c_{1}e^{-y\sqrt{\lambda}}+c_{2}\lambda^{-\frac{1}{2}}e^{y\sqrt{\lambda}}),
$$
for certain $c_{1},c_{2}\in\mathbb{C}$. By Lemma \ref{abr} (a), the function $u$ is continuous up to zero, which implies that
 $$c_{1}\sqrt{\lambda}+c_{2}=0.$$
On the other hand, since the real part of $\sqrt{\lambda}$ is not equal to zero, from the fact that $u\in\mathcal{D}(\underline{A}_{\wedge,0})$, we deduce that at least one of $c_{1}$, $c_{2}$ must be zero. Hence $u=0$, which implies that $\lambda-\underline{A}_{j,\wedge,0}$ is injective. 

The inner product $\langle\cdot,\cdot\rangle_{0,0}$ of $\mathcal{K}_{2}^{0,0}(\mathbb{R}_{+})$ yields an identification of the dual space of $\mathcal{K}_{2}^{0,\ell}(\mathbb{R}_{+})$ with $\mathcal{K}_{2}^{0,-\ell}(\mathbb{R}_{+})$. Notice that $A_{\wedge}$ is not symmetric since its formal adjoint is given by
$$
A_{\wedge}^{\ast}=y^{-2}\big((y\partial_{y})^{2}-3(y\partial_{y})+2\big).
$$
The operator $A_{\wedge}^{\ast}$ is the model cone of a $\mathbb{B}$-elliptic cone differential operator with conormal symbol given by
$
\lambda^{2}+3\lambda+2,
$
which vanishes at $-2$ and $-1$. Note that
$$-1\in I_{-\ell}=[-3/2+\ell,1/2+\ell)$$
and $-2\notin I_{-\ell}$. Therefore,
by \eqref{dommodcone1}-\eqref{dommodcone2}, the maximal domain of $A_{\wedge}^{\ast}$ in $\mathcal{K}_{2}^{0,-\ell}(\mathbb{R}_{+})$ is given by
\begin{equation*}
\mathcal{D}(\underline{A}_{\wedge,\max}^{\ast})=\mathcal{K}_{2}^{2,2-\ell}(\mathbb{R}_{+})\oplus\mathcal{K}_{-1},
\end{equation*}
where $\mathcal{K}_{-1}$ is the space of functions of the form
$$y\mapsto c\omega_{1}y,$$ with $c\in\mathbb{C}$.

The adjoint $\underline{A}_{\wedge}^{\ast}$ of $\underline{A}_{\wedge}$ is defined by the action of $A_{\wedge}^{\ast}$ in $\mathcal{K}_{2}^{0,-\ell}(\mathbb{R}_{+})$ and its domain is given by
\begin{eqnarray*}
\lefteqn{\mathcal{D}(\underline{A}_{\wedge}^{\ast})= \Big\{v\in \mathcal{K}_{2}^{0,-\ell}(\mathbb{R}_{+})\, |
\, \text{there exists } f\in \mathcal{K}_{2}^{0,-\ell}(\mathbb{R}_{+})}\\
&&\quad\quad\quad\text{such that}\,\, \langle v,\underline{A}_{\wedge}u\rangle_{0,0}=\langle f, u\rangle_{0,0}\,\, \text{for all } u\in \mathcal{D}(\underline{A}_{\wedge})\Big\}.
\end{eqnarray*}
Choosing $u=c_{3}\omega_{1}\in \mathbb{C}_{\omega_{\wedge}}$ and $v=c_{4}\omega_{1}y\in \mathcal{K}_{-1}$, the condition 
$$
\langle v,\underline{A}_{\wedge}u\rangle_{0,0}=\langle \underline{A}_{\wedge}^{\ast}v, u\rangle_{0,0}
$$
becomes $\bar{c}_{3}c_{4}=0$. Therefore, $c_{4}=0$ and the domain of $\underline{A}_{\wedge}^{\ast}$ is the minimal one, i.e. 
\begin{equation}\label{maxdAstar}
\mathcal{D}(\underline{A}_{\wedge}^{\ast})=\mathcal{K}_{2}^{2,2-\ell}(\mathbb{R}_{+}).
\end{equation}

Now we shall show that $\underline{A}_{\wedge}^{\ast}$ is an injection. Suppose that $v\in \mathcal{D}(\underline{A}_{\wedge}^{\ast})$ satisfies
$$(\lambda-\underline{A}_{\wedge}^{\ast})v=0,$$
for certain $\lambda\in\mathbb{C}\backslash(-\infty,0]$. This is equivalent to
$$y^{2}v_{yy}-2yv_{y}+(2-\lambda y^{2})v=0.$$
Hence
$$
v=y(c_{5}e^{-y\sqrt{\lambda}}+c_{6}\lambda^{-\frac{1}{2}}e^{y\sqrt{\lambda}}),
$$
for some $c_{5},c_{6}\in\mathbb{C}$. By \eqref{maxdAstar} and Lemma \ref{abr} (a), we have that $v\in C(\mathbb{B})$ and 
$
|\omega_{1}v(y)|\leq c_{7}y^{3/2-\ell}, 
$
for certain $c_{7}>0$. Since $3/2-\ell>1$, it follows that $$c_{5}\sqrt{\lambda}+c_{6}=0.$$ Because the real part of $\sqrt{\lambda}$ is not equal to zero and $u\in\mathcal{D}(\underline{A}_{\wedge,0})$, at least one of $c_{5}$, $c_{6}$ must be zero. Thus $c_{5}=c_{6}=0$ and so $\lambda-\underline{A}_{\wedge}^{\ast}$ is injective.

If $\lambda\in\mathbb{C}\backslash(-\infty,0]$, by \cite[Remark 5.26 (i)]{GKM}, the operator $\lambda-\underline{A}_{\wedge}^{\ast}$ is Fredholm and therefore it has closed range. This fact, together with the injectivity, implies that $\lambda-\underline{A}_{\wedge}^{\ast}$ is bounded below and therefore $\bar{\lambda}-\underline{A}_{\wedge}$ is surjective. Hence, for each $\lambda\in\mathbb{C}\backslash(-\infty,0]$, the operator $\lambda-\underline{A}_{\wedge}$ is bijective. This completes the proof. 
\end{proof}

\begin{lemma}\label{wperturb}
Let $r\in\{0\}\cup [1,\infty)$, $\ell\in(-1/2,1/2)$ and let $\underline{A}_{r}$ be the closed extension of $A$ defined in \eqref{domA}. Moreover, let $\nu>0$ and
$$
v\in \mathcal{H}_{p}^{r+\frac{1}{p}+\nu,\frac{1}{2}+\nu}(\mathbb{B})\oplus\mathbb{C}_{\omega}
$$ 
be real-valued positive function bounded away from zero. Then, there exist $c>0$ and $\theta\in({\pi}/{2},\pi)$ such that $c-v\underline{A}_{r}\in \mathcal{R}(\theta)$.
\end{lemma}
\begin{proof}
By Lemma \ref{maxregA}, for each $\theta_{0}\in(\pi/2,\pi)$, there exists a $c_{0}>0$ such that the operator $c_{0}-\underline{A}_{r}:X_{1}^{r}\rightarrow X_{0}^{r}$, is $R$-sectorial of angle $\theta_{0}$. From Lemma \ref{abr} (b), we have $v\underline{A}_{r}\in\mathcal{L}(X_{1}^{r},X_{0}^{r})$. For $r=0$ we will employ the freezing-of-coefficients method to show that, for certain $c>0$, the operator
\begin{equation}\label{opfprs0}
c-v\underline{A}_{0}:X_{1}^{0}\rightarrow X_{0}^{0},
\end{equation}
is $R$-sectorial of angle $\theta_{0}$. We split the proof into several parts according to different values of $r$ and follow the same steps as in the proof of \cite[Theorem 6.1]{RS3}. 

{\em Case $r=0$}. We will show the following: for each real-valued positive function bounded away from zero $v\in C(\mathbb{B})$ and each $\theta_{0}\in(\pi/2,\pi)$, there exists a $c>0$ such that the operator \eqref{opfprs0} is $R$-sectorial of angle $\theta_{0}$; note that, by Lemma \ref{abr} (a), we have 
$$
\mathcal{H}_{p}^{\frac{1}{p}+\nu,\frac{1}{2}+\nu}(\mathbb{B})\oplus\mathbb{C}_{\omega} \hookrightarrow C(\mathbb{B}).
$$
Choose $c>c_{0}\max\{v(x)\, |\, x\in\mathbb{B}\}$ and fix $s\in\mathbb{B}$. For each $\lambda\in S_{\theta_{0}}\backslash\{0\}$, we have 
\begin{eqnarray*}
\lefteqn{\lambda\big(\lambda+c-v(s)\underline{A}_{0}\big)^{-1}}\\
&=&\frac{\lambda}{v(s)}\Big(\frac{\lambda}{v(s)}+\frac{c}{v(s)}-c_{0}+c_{0}-\underline{A}_{0}\Big)^{-1}\\
&=&\frac{\frac{\lambda}{v(s)}}{\frac{\lambda}{v(s)}+\frac{c}{v(s)}-c_{0}}\,\Big(\frac{\lambda}{v(s)}
+\frac{c}{v(s)}-c_{0}\Big)\Big(\frac{\lambda}{v(s)}+\frac{c}{v(s)}-c_{0}+c_{0}-\underline{A}_{0}\Big)^{-1}.
\end{eqnarray*}
Once we have the above expression, we proceed as in \cite[(6.4)]{RS3}, using the $R$-sectoriality of $c_{0}-\underline{A}_{0}$ and Kahane's contraction principle \cite[Proposition 2.5]{KuW1}, to deduce that
$$c-v(s)\underline{A}_{0}:X_{1}^{0}\rightarrow X_{0}^{0}$$
is $R$-sectorial of angle $\theta_{0}$ and its $R$-sectorial bound is uniformly bounded with respect to $s\in \mathbb{B}$.

Let $\varepsilon>0$ be arbitrary small and let $\{B_{k}\}_{k\in\{0,\dots,n\}}$, $n\in\mathbb{N}$, $n\geq2$, be an open cover of $\mathbb{B}$ consisting of the collar parts $B_{0}=[0,\varepsilon)$ and $B_1=(1-\varepsilon,1]$, together with a collection of open intervals $B_{k}=(s_{k}-\varepsilon,s_{k}+\varepsilon)$, $k\in\{2,\dots,n\}$,
where $\{s_{k}\}_{k\in\{2,\dots,n\}}$ is a set of points in $(0,1)$. We assume that
$$
(s_{k}-3\varepsilon/2,s_{k}+3\varepsilon/2)\cap\{0,1\}=\emptyset
$$
for all $k=\{2,\dots,n\}$. Moreover, let $\widetilde{\omega}$ be a smooth function on $\mathbb{R}$ such that $\widetilde{\omega}(x)=1$ when
$x\leq{1}/{2}$ and $\widetilde{\omega}(x)=0$ when $x\geq {3}/{4}$. Define
$$
v_{k}(x)=\widetilde{\omega}(|x-s_{k}|/2\varepsilon)v(x)+\big(1-\widetilde{\omega}(|x-s_{k}|/2\varepsilon)\big)v(s_{k}),
$$
where $x\in\mathbb{B}$, $k\in\{0,\dots,n\}$, $s_{0}=0$ and $s_1=1$. Taking $\varepsilon>0$ small enough and by possibly enlarging $n$, the norms $\|\widetilde{\omega}(|x-s_{k}|/2\varepsilon)(v-v_{k})\|_{C(\mathbb{B})}$ becomes arbitrary small. As a consequence, the norm of each $\omega(|x-s_{k}|/2\varepsilon)(v-v_{k})$, regarded as a multiplier on $X_{0}^{0}$, becomes arbitrary small. By the perturbation result of Kunstmann and Weis \cite[Theorem 1]{KuWe}, it follows that
$$
c-v_{k}\underline{A}_{0}=c-v(x_{k})\underline{A}_{0}+(v(x_{k})-v_{k})\underline{A}_{0}\in\mathcal{R}(\theta_{0}),
$$
for each $k=\{0,\dots,n\}$.

Let $\{\phi_{k}\}_{k\in\{0,\dots,n\}}$ be a partition of unity subordinate to $\{B_{k}\}_{k\in\{0,\dots,n\}}$, and let $\psi_{k}\in C^{\infty}(\mathbb{B})$,
$k=\{0,\dots,n\}$, such that $\mathrm{supp}(\psi_{k})\subset B_{k}$ and $\psi_{k}=1$ on $\mathrm{supp}(\phi_{k})$. By \cite[(I.2.5.2),(I.2.9.6)]{Am} and \cite[Lemma 5.2]{RS3}, for any $\eta\in(0,1)$ the fractional powers of $\underline{A}_{0}$ satisfy
$$
\mathcal{H}_{p}^{2\eta+\epsilon,\ell+2\eta+\epsilon}(\mathbb{B})\oplus \mathbb{C}_{\omega} \hookrightarrow \mathcal{D}((c_{0}-\underline{A}_{0})^{\eta})\hookrightarrow\mathcal{H}_{p}^{2\eta-\epsilon,\ell+2\eta-\epsilon}(\mathbb{B})\oplus \mathbb{C}_{\omega}
$$
for all $\epsilon>0$. Therefore, if $\lambda \in S_{\theta_{0}}$, by taking $c>0$ sufficiently large, similarly to \cite[pp. 1457-1458]{RS3}, we can construct an inverse for $\lambda+c-v\underline{A}_{0}$ in the space $\mathcal{L}(X_{0}^{0},X_{1}^{0})$, given by 
\begin{equation}\label{resolventvA}
\big(\lambda+c-v\underline{A}_{0}\big)^{-1}=\sum_{k=0}^{\infty}(-1)^{k}B^{k}(\lambda)R(\lambda),
\end{equation}
where 
$$
B(\lambda)=\sum_{i=0}^{n}\psi_{i}\big(\lambda+c-v_{i}\underline{A}_{0}\big)^{-1}[v\underline{A}_{0},\phi_{i}] 
$$
and
$$
R(\lambda)=\sum_{i=0}^{n}\psi_{i}\big(\lambda+c-v_{i}\underline{A}_{0}\big)^{-1}\phi_{i}.
$$
Using equation \eqref{resolventvA}, and following the same lines as in \cite[pp. 1458-1459]{RS3}, we obtain the $R$-sectoriality of $c-v\underline{A}_{0}$.
 
{\bf Claim}. {\em For any $r\geq0$, the resolvent of $v\underline{A}_{r}$ is the restriction of the resolvent of $v\underline{A}_{0}$ to $X_{0}^{r}$.}

{\em Proof of claim.} It is sufficient to show that
$$
(\lambda-v\underline{A}_{0})^{-1}(\lambda-v\underline{A}_{r})=I\quad\text{and} \quad (\lambda-v\underline{A}_{r})(\lambda-v\underline{A}_{0})^{-1}=I
$$
hold on $X_{1}^{r}$ and $X_{0}^{r}$ respectively. The first one is trivial. Concerning the second one, let $u\in X_{1}^{0}$ such that
$$(\lambda-v\underline{A}_{r})u\in X_{0}^{r}.$$
If $r\in[0,2]$, then $v\underline{A}_{r}u\in X_{0}^{r}$ and by Lemma \ref{abr} (b), (c), we obtain that $\underline{A}_{r}u\in X_{0}^{r}$. Therefore, $u$ belongs to the maximal domain of $A$ in $X_{0}^{r}$. From the structure of the maximal domain given by \eqref{maxdomA}, the choice of $\ell$, and the fact that $u\in X_{1}^{0}$, we deduce that $u\in X_{1}^{r}$. Iteration then shows the result for arbitrary $r$. This proves the claim.

{\em Case $r\in\mathbb{N}$}. Taking into account the above claim, the $R$-sectoriality of $c-v\underline{A}_{r}$ for arbitrary $r\in\mathbb{N}$ and suitable $c>0$, follows by induction as in \cite[pp. 1460-1461]{RS3}.

{\em Case $r>1$}. 
By the interpolation results \cite[Theorem 3.19]{KaSa} and \cite[Lemma 3.7]{RS3}, there exists a $c>0$ such that the operator 
\begin{equation}\label{rsecAlin}
c-v\underline{A}_{r_{0}}:X_{1}^{r_{0}}\rightarrow X_{0}^{r_{0}}
\end{equation}
is $R$-sectorial of angle $\theta_{0}$, where $r_{0}=r-1$. Let now $q\in(1,\infty)$, $T>0$, $f\in L^{q}(0,T;X_{0}^{1+r_{0}})$ and consider the problem
\begin{eqnarray}\label{atestP1}
u_{t}(t)-v\underline{A}_{r_{0}}u(t)&=&f(t), \quad t\in(0,T),\\\label{atestP2}
u(0)&=&0.
\end{eqnarray}
By Theorem \ref{KalWeiTh}, there exists a unique 
\begin{equation}\label{firstregu}
u\in E_{1}^{r_{0}}=H^{1,q}(0,T;X_{0}^{r_{0}})\cap L^{q}(0,T;X_{1}^{r_{0}})
\end{equation}
solving \eqref{atestP1}-\eqref{atestP2}. More precisely, denote by $B_{r_{0}}$ the operator $\partial_{t}$ in the space $L^{q}(0,T;X_{0}^{r_{0}})$ with domain 
$$
\mathcal{D}(B_{r_{0}})=\{u\in H_{q}^{1}(0,T;X_{0}^{r_{0}})\, |\, u(0)=0\}.
$$ 
Since $X_{0}^{r_{0}}$ is UMD, by \cite[Theorem 8.5.8]{Ha}, for each $\phi\in(0,\pi/2)$ we have that $B_{r_{0}}\in\mathcal{H}^{\infty}(\phi)$. Because $B_{r_{0}}$ and $\Lambda_{r_{0}}=c-v\underline{A}_{r_{0}}$ are resolvent commuting in the sense of \cite[(III.4.9.1)]{Am}, by \cite[Theorem 6.3]{KW1}, the inverse $(B_{r_{0}}+\Lambda_{r_{0}})^{-1}$ exists as a bounded map from $L^{q}(0,T;X_{0}^{r_{0}})$ to $H_{q}^{1}(0,T;X_{0}^{r_{0}})\cap L^{q}(0,T;X_{1}^{r_{0}})$. In particular, after changing $u$ to $e^{ct}u$ in \eqref{atestP1}-\eqref{atestP2}, we obtain
\begin{equation}\label{uexpr}
u(t)=e^{ct}(B_{r_{0}}+\Lambda_{r_{0}})^{-1}(e^{-c(\cdot)}f(\cdot))(t), \quad t\in[0,T].
\end{equation}
By classical results of Da Prato-Grisvard \cite[Theorem 3.7]{DG} and Kalton-Weis \cite[Theorem 6.3]{KW1}, the solution is expressed by the formula \cite[(3.11)]{DG} for the inverse of the closure of the sum of two closed operators, i.e. 
$$
(B_{r_{0}}+\Lambda_{r_{0}})^{-1}=\frac{1}{2\pi i}\int_{\Gamma_{\theta_{0}}} (\Lambda_{r_{0}}+\lambda)^{-1}(B_{r_{0}}-\lambda)^{-1}d\lambda,
$$
where the path $\Gamma_{\theta_{0}}$ is defined in \eqref{gammapath}. Since 
$$
(\Lambda_{r_{0}+1}+\lambda)^{-1}=(\Lambda_{r_{0}}+\lambda)^{-1}|_{X_{0}^{r_{0}+1}}, \quad \lambda\in S_{\theta_{0}},
$$
we obtain that
$$
(B_{r_{0}+1}+\Lambda_{r_{0}+1})^{-1}=(B_{r_{0}}+\Lambda_{r_{0}})^{-1}|_{E_{0}^{r_{0}+1}},
$$
where
$$
E_{0}^{r_{0}}=L^{q}(0,T;X_{0}^{r_{0}}).
$$
Consequently, $(B_{r_{0}}+\Lambda_{r_{0}})^{-1}$ maps $e^{-c(\cdot)}f(\cdot)$ to $E_{1}^{r_{0}+1}$, and so 
$$
\omega_{0}x\partial_{x}(B_{r_{0}}+\Lambda_{r_{0}})^{-1}(e^{-c(\cdot)}f(\cdot))\in E_{1}^{r_{0}}.
$$
Therefore, after applying $x\partial_{x}$ to \eqref{uexpr}, we obtain
\begin{eqnarray}\nonumber
(\omega_{0}xu_{x})(t)&=&e^{ct}(B_{r_{0}}+\Lambda_{r_{0}})^{-1}(e^{-c(\cdot)}\omega_{0}xf_{x}(\cdot))(t)\\\nonumber
&&+e^{ct}[\omega_{0}x\partial_{x},(B_{r_{0}}+\Lambda_{r_{0}})^{-1}](e^{-c(\cdot)}f(\cdot))(t)\\\label{commtaxdx}
&=&e^{ct}(B_{r_{0}}+\Lambda_{r_{0}})^{-1}(e^{-c(\cdot)}\omega_{0}xf_{x}(\cdot))(t)\\\nonumber
&&+e^{ct}(B_{r_{0}}+\Lambda_{r_{0}})^{-1}[\Lambda_{r_{0}},\omega_{0}x\partial_{x}](B_{r_{0}}+\Lambda_{r_{0}})^{-1}(e^{-c(\cdot)}f(\cdot))(t),
\end{eqnarray}
for each $t\in [0,T]$. Since $\omega_{0}xf_{x}\in E_{0}^{r_{0}}$, the first term on the right hand side of \eqref{commtaxdx} belongs to $E_{1}^{r_{0}}$. Moreover, 
$$
(B_{r_{0}}+\Lambda_{r_{0}})^{-1}(e^{-c(\cdot)}f(\cdot))\in E_{1}^{r_{0}}
$$ 
as well. By straightforward calculation, we get that $[\Lambda_{r_{0}},\omega_{0}x\partial_{x}]$ close to zero equals to
$$
(xv_{x}-2v)\partial_{x}^{2}-4x^{-1}\partial_{x}.
$$
Since 
$$
\partial_{x}^{2}, x^{-1}\partial_{x} : L^{q}(0,T;X_{1}^{r_{0}}) \rightarrow E_{0}^{r_{0}}
$$
and 
$$
\omega_{0}\{v, xv_{x}\}\in \mathcal{H}_{p}^{r_{0}+\frac{1}{p}+\nu,\frac{1}{2}+\nu}(\mathbb{B})\oplus\mathbb{C}_{\omega},
$$ 
by Lemma \ref{abr} (b), we obtain that 
$$
[\Lambda_{r_{0}},\omega_{0}x\partial_{x}](B_{r_{0}}+\Lambda_{r_{0}})^{-1}(e^{-c(\cdot)}f(\cdot))\in E_{0}^{r_{0}}.
$$
Consequently, the second term on the right hand side of \eqref{commtaxdx} belongs to $E_{1}^{r_{0}}$. Therefore, $\omega_{0}xu_{x}\in E_{1}^{r_{0}}$. Similarly, we show that
$\omega_{1}(x-1)u_{x}\in E_{1}^{r_{0}}$. This fact combined with \eqref{firstregu}, yields
$$
u\in E_{1}^{r_{0}+1} = H^{1,q}(0,T;X_{0}^{1+r_{0}})\cap L^{q}(0,T;X_{1}^{1+r_{0}}).
$$
Hence, by \eqref{atestP1}-\eqref{atestP2} we deduce that the operator $v\underline{A}_{r}:X_{1}^{r}\rightarrow X_{0}^{r}$ has maximal $L^{q}$-regularity. By the result of Weis \cite[Theorem 4.2]{Weis}, we conclude that there exist a $c>0$ and $\theta \in(\pi/2,\pi)$ such that $c-v\underline{A}_{r}\in\mathcal{R}(\theta)$.
\end{proof}

\begin{theorem}\label{short1d}
Let $\gamma$ be a curve satisfying the conditions {\rm(a)-(d)} of Section {\rm \ref{assumptongama}}, and let
\begin{equation}\label{ell0}
\ell_{0}=\min_{_{j\in\{0,1\}}}\Big\{0,-\frac{1+2\beta_{1}}{1+\beta_{1}},\frac{\alpha-5/2}{1+\beta_{j}},\frac{\alpha-7/2-2\beta_{j}}{1+\beta_{j}}\Big\}.
\end{equation}
For any, $q\in(2,\infty)$ and $r\geq1$ satisfying 
\begin{equation}\label{pqgammachoice0}
\frac{2}{q}+ \frac{1}{p}<1, \quad 
 r\in \bigg\{\begin{array}{lll}
(b-4+2/q, b-3-1/p), & \text{if} & b<\infty, \\ 
{[}1,\infty), & \text{if} & b=\infty,
\end{array}
\end{equation}
and $\ell\in\mathbb{R}$ such that
\begin{equation}\label{pqgammachoice}
-\frac{1}{2}+\frac{2}{q}<\ell<\frac{1}{2}+\ell_{0},
\end{equation}
there exists a $T>0$ and a unique 
\begin{equation}\label{exanduniqofu}
u\in H^{1,q}(0,T; \mathcal{H}_{p}^{r,\ell}(\mathbb{B}))\cap L^{q}(0,T; \mathcal{H}_{p}^{r+2,\ell+2}(\mathbb{B})\oplus\mathbb{C}_{\omega})
\end{equation}
solving \eqref{flqz1}-\eqref{flqz2}. The function $u$ also satisfies
\begin{eqnarray}\label{extrareginspace1}
\lefteqn{\hspace{-5mm} u\in C([0,T]; \mathcal{H}_{p}^{r+2-2/q-\varepsilon,\ell+2-2/q-\varepsilon}(\mathbb{B})\oplus\mathbb{C}_{\omega})}\\\label{extrareginspace0} 
&&\hspace{-5mm}\cap \, C([0,T]; (\mathcal{H}_{p}^{r+2,\ell+2}(\mathbb{B})\oplus\mathbb{C}_{\omega},\mathcal{H}_{p}^{r,\ell}(\mathbb{B}))_{1/q,q})\cap C([0,T];C(\mathbb{B})) 
\end{eqnarray}
for any $\varepsilon>0$. In particular, if $b=\infty$, then for any $\tau\in(0,T)$, $\nu,\varepsilon>0$ and $n\in\mathbb{N}$, the function $u$ satisfies
\begin{eqnarray}\label{extrareginspace2}
\lefteqn{u\in C([\tau,T]; \mathcal{H}_{p}^{\nu,\ell+2-2/q-\varepsilon}(\mathbb{B})\oplus\mathbb{C}_{\omega})}\\\label{extrareginspace3}
&&\cap \, H^{n,q}(\tau,T; \mathcal{H}_{p}^{\nu,\ell-2(n-1)}(\mathbb{B}))\cap L^{q}(\tau,T; \mathcal{H}_{p}^{\nu,\ell+2}(\mathbb{B})\oplus\mathbb{C}_{\omega}).
\end{eqnarray}
\end{theorem}
\begin{proof}
We apply Theorem \ref{ClementLi} of Cl\'ement and Li to \eqref{flqz1}-\eqref{flqz2} with 
$$
X_{0}=X_{0}^{r}=\mathcal{H}_{p}^{r,\ell}(\mathbb{B}) \quad \text{and} \quad X_{1}=X_{1}^{r}=\mathcal{H}_{p}^{r+2,\ell+2}(\mathbb{B})\oplus\mathbb{C}_{\omega},
$$ 
$A(\cdot)$ as in \eqref{newexpA}, $F$ given by \eqref{Fterm} and $u_{0}=0$. Instead of taking $U$ in Theorem \ref{ClementLi} a subset of $(X_{1}^{r},X_{0}^{r})_{1/q,q}$, we choose $U$ to be an open ball $B_{\mu}^{r}(\delta)$ in $(X_{1}^{r},X_{0}^{r})_{1/(q-\mu),q}$ centered at $u_{0}$ and of radius $\delta>0$, where
\begin{equation}\label{sigmachoice}
0\leq\mu<q-2\max\Big\{\frac{p}{p-1},\frac{2}{2\ell+1}\Big\}.
\end{equation}
The parameter $\mu$ will be used later to show the smoothness of the solution. We would like to point out that Mellin-Sobolev space regularity in \eqref{ftermsA}, and the weights in \eqref{ftermsA}-\eqref{ftermsB}, determine the parameters $r$ and $\ell$. Define
$$
\nu_{0}=\bigg\{
\begin{array}{lll} b-3 & \text{if}& b<\infty,\\
r+1-\frac{2}{q-\mu} & \text{if}& b=\infty,
\end{array}
$$
and fix 
$$
0<\varepsilon_{0}<\min\Big\{1,\frac{2\ell+1}{2}-\frac{2}{q-\mu}\Big\}.
$$
Then, by Lemma \ref{abr} (a) we have
$$
\mathbb{X}_{\omega}\hookrightarrow \mathcal{H}_{p}^{\nu_{0},\frac{1}{2}+\varepsilon_{0}}(\mathbb{B})\oplus\mathbb{C}_{\omega}\hookrightarrow C(\mathbb{B}).
$$
Additionally, by \eqref{xregdata} we get
\begin{equation}\label{unireg1}
\omega|x-\kappa|^{\frac{1}{1+\beta_{\kappa}}}k \in \mathcal{H}_{p}^{\nu_{0},\frac{1}{2}+\varepsilon_{0}}(\mathbb{B})
\end{equation}
and
\begin{equation}\label{uniregA}
\omega\{\varphi_{1},\varphi_{2}, \xi_{1},\xi_{2}, \vartheta(\cdot,\lambda)\}\in \mathcal{H}_{p}^{\nu_{0},\frac{1}{2}+\varepsilon_{0}}(\mathbb{B})\oplus\mathbb{C}_{\omega}.
\end{equation}
By Lemma \ref{abr} (a), (d), we have
\begin{equation}\label{contembofint}
(X_{1}^{r},X_{0}^{r})_{\frac{1}{q-\mu},q}\hookrightarrow \mathcal{H}_{p}^{r+2-\frac{2}{q-\mu}-\varepsilon, \ell+2-\frac{2}{q-\mu}-\varepsilon}(\mathbb{B})\oplus\mathbb{C}_{\omega},
\end{equation}
for any sufficiently small $\varepsilon>0$. In particular,
\begin{equation}\label{uniregB}
(X_{1}^{r},X_{0}^{r})_{\frac{1}{q-\mu},q}\hookrightarrow \mathcal{H}_{p}^{\nu_{0},\frac{1}{2}+\varepsilon_{0}}(\mathbb{B})\oplus\mathbb{C}_{\omega}.
\end{equation}

{\em Lipschitz property of $A(\cdot)$}. We will show this property locally on the support of $\omega$. Let $v,v_{1},v_{2}\in B_{\mu}^{r}(\delta)$. Set 
$$
Q_{4}(v)=1-k((1+\beta_{\kappa})|x-\kappa|)^{\frac{1}{1+\beta_{\kappa}}}v
$$
and
$$
Q_{5}(v)=e^{-2\eta(v)}|\varphi|^{2}|\varphi+v\xi|^{-2(1+\beta_{\kappa})}(Q_{4}(v))^{-1}.
$$
By \eqref{newexpA}, it suffices to show that there exists a $C_{1}>0$ such that
\begin{equation}\label{QLips}
\|\omega(Q_{i}(v_{1})-Q_{i}(v_{2}))\|_{\mathcal{L}(X_{0})}\leq C_{1}\|v_{1}-v_{2}\|_{(X_{1},X_{0})_{\frac{1}{q-\mu},q}},
\end{equation}
where $i\in\{1,4\}$ and each $Q_{i}(v_{l})$, $l\in\{1,2\}$, acts as a multiplication operator. In the interior of $(0,1)$, the expressions of $Q_{1}$, $Q_{2}$ and $Q_{3}$ are obtained similarly by taking the orders $\beta_{0}$ and $\beta_{1}$ equal to zero, which makes the argument simpler. 

For each $y\in\mathbb{R}$, we have
\begin{equation}\label{multwithpower}
\omega |x-\kappa|^{y}: \mathbb{C}_{\omega}\rightarrow \bigcap_{\rho>1, \nu>0,\varepsilon>0}\mathcal{H}_{\rho}^{\nu,\frac{1}{2}+y-\varepsilon}(\mathbb{B}),
\end{equation}
and the induced operator is bounded for every fixed $\rho\in(1,\infty)$, $\nu$ and $\varepsilon>0$. Therefore, for each $\varepsilon>0$ small enough, the map 
\begin{eqnarray}\nonumber
\lefteqn{\omega |x-\kappa|^{\frac{1}{1+\beta_{\kappa}}}:(X_{1}^{r},X_{0}^{r})_{\frac{1}{q-\mu},q}}\\\label{oakndbff}
&&\rightarrow \mathcal{H}_{p}^{\nu_{0},\frac{1}{2}+\frac{1}{1+\beta_{1}}-\varepsilon}(\mathbb{B}) \hookrightarrow \mathcal{H}_{p}^{\nu_{0},\frac{3}{2}-\varepsilon}(\mathbb{B}) \hookrightarrow C(\mathbb{B})
\end{eqnarray}
is also bounded. Hence, there exists a $C_{2}>0$ such that
\begin{eqnarray*}
\lefteqn{\omega\big((1+\beta_{\kappa})|x-\kappa|\big)^{\frac{1}{1+\beta_{\kappa}}}||v_{1}|-|v_{2}||}\\
&\leq& \|\omega\big((1+\beta_{\kappa})|x-\kappa|\big)^{\frac{1}{1+\beta_{\kappa}}}(v_{1}-v_{2})\|_{C(\mathbb{B})}\leq C_{2}\|v_{1}-v_{2}\|_{(X_{1}^{r},X_{0}^{r})_{\frac{1}{q-\mu},q}}.
\end{eqnarray*}
Thus, after taking $\delta>0$ sufficiently small, we can choose a closed smooth path $\Gamma_{1}$ within $\{z\in \mathbb{C}\, |\, |z|<R\}$ containing the sets
$$
\bigcup_{v\in B_{\mu}^{r}(\delta)}\mathrm{Ran}(\omega_{j} |x-j|^{1/(1+\beta_{j})}v).
$$
By the analyticity of $\vartheta$, we write
$$
\omega\eta(v)=\frac{1}{2\pi i}\int_{\Gamma_{1}}\frac{\omega\vartheta\big(((1+\beta_{\kappa})|x-\kappa|)^{\frac{1}{1+\beta_{\kappa}}},\lambda\big) }{\lambda-\big((1+\beta_{\kappa})|x-\kappa|\big)^{\frac{1}{1+\beta_{\kappa}}}v}d\lambda.
$$
From the above expression, \eqref{uniregA}, \eqref{oakndbff}, Lemma \ref{abr} (a), (c) and Lemma \ref{sumSn}, we deduce that
\begin{equation}\label{expregv1}
\omega \eta(v)\in \mathcal{H}_{p}^{\nu_{0},\frac{1}{2}+\varepsilon_{0}}(\mathbb{B})\oplus\mathbb{C}_{\omega}.
\end{equation}
Additionally, 
\begin{eqnarray*}
&&\hspace{-20pt}\omega(\eta(v_{1})-\eta(v_{2}))\\
&&\hspace{-17pt}=\frac{1}{2\pi i}\int_{\Gamma_{1}}\frac{\omega \vartheta\big(((1+\beta_{\kappa})|x-\kappa|)^{\frac{1}{1+\beta_{\kappa}}},\lambda\big) \big((1+\beta_{\kappa})|x-\kappa|\big)^{\frac{1}{1+\beta_{\kappa}}}(v_{1}-v_{2}) }{\big(\lambda-((1+\beta_{\kappa})|x-\kappa|)^{\frac{1}{1+\beta_{\kappa}}}v_{1}\big)\big(\lambda-((1+\beta_{\kappa})|x-\kappa|)^{\frac{1}{1+\beta_{\kappa}}}v_{2}\big)}d\lambda.
\end{eqnarray*}
By \eqref{uniregA}, \eqref{oakndbff}, Lemma \ref{abr} (a), (c) and Lemma \ref{sumSn}, we have
\begin{equation}\label{etatermlip}
\|\omega(\eta(v_{1})-\eta(v_{2}))\|_{\mathcal{H}_{p}^{\nu_{0},\frac{1}{2}+\varepsilon_{0}}(\mathbb{B})\oplus\mathbb{C}_{\omega}}\leq C_{3}\|v_{1}-v_{2}\|_{(X_{1}^{r},X_{0}^{r})_{\frac{1}{q-\mu},q}},
\end{equation}
for some $C_{3}>0$. Hence, there exists a $C_{4}>0$ such that
$$
|\omega(\eta(v_{1})-\eta(v_{2}))|\leq \|\omega(\eta(v_{1})-\eta(v_{2}))\|_{C(\mathbb{B})} \leq C_{4}\|v_{1}-v_{2}\|_{(X_{1}^{r},X_{0}^{r})_{\frac{1}{q-\mu},q}}.
$$
Consequently, we can choose a closed smooth path $\Gamma_{2}$ containing the sets
$$
\bigcup_{v\in B_{\mu}^{r}(\delta)}\mathrm{Ran}(\omega_{j}\eta(v)).
$$
Then,
$$
\omega e^{-2\eta(v)}=\frac{1}{2\pi i}\int_{\Gamma_{2}}\frac{\omega e^{-2\lambda}}{\lambda-\eta(v)}d\lambda,
$$
and from \eqref{expregv1} and Lemma \ref{sumSn}, we deduce
\begin{equation}\label{expregv}
\omega e^{-2\eta(v)}\in \mathcal{H}_{p}^{\nu_{0},\frac{1}{2}+\varepsilon_{0}}(\mathbb{B})\oplus\mathbb{C}_{\omega}.
\end{equation}
Moreover,
$$
\omega(e^{-2\eta(v_{1})}-e^{-2\eta(v_{2})})=\frac{1}{2\pi i}\int_{\Gamma_{2}}\frac{\omega e^{-2\lambda}(\eta(v_{1})-\eta(v_{2}))}{(\lambda-\eta(v_{1}))(\lambda-\eta(v_{2}))}d\lambda.
$$
Again by \eqref{etatermlip} and Lemma \ref{sumSn}, we obtain
\begin{equation}\label{explipsch}
\|\omega(e^{-2\eta(v_{1})}-e^{-2\eta(v_{2})})\|_{\mathcal{H}_{p}^{\nu_{0},\frac{1}{2}+\varepsilon_{0}}(\mathbb{B})\oplus\mathbb{C}_{\omega}}\leq C_{5}\|v_{1}-v_{2}\|_{(X_{1}^{r},X_{0}^{r})_{\frac{1}{q-\mu},q}},
\end{equation}
for some $C_{5}>0$.

Due to \eqref{unireg1}, \eqref{uniregB} and Lemma \ref{abr} (a), we obtain
$$
\omega Q_{4}(v) \in\mathcal{H}_{p}^{\nu_{0},\frac{1}{2}+\varepsilon_{0}}(\mathbb{B})\oplus\mathbb{C}_{\omega};
$$
in particular
\begin{equation}\label{Q5lip}
\|\omega(Q_{4}(v_{1})-Q_{4}(v_{2}))\|_{\mathcal{H}_{p}^{\nu_{0},\frac{1}{2}+\varepsilon_{0}}(\mathbb{B})\oplus\mathbb{C}_{\omega}}
 \leq C_{6}\|v_{1}-v_{2}\|_{(X_{1}^{r},X_{0}^{r})_{\frac{1}{q-\mu},q}},
\end{equation}
for certain $C_{6}>0$. Hence, by the relation 
$$
|\omega-\omega Q_{4}(v)|\leq\|\omega(Q_{4}(u_{0})-Q_{4}(v))\|_{C(\mathbb{B})}\leq C_{4}\|u_{0}-v\|_{(X_{1}^{r},X_{0}^{r})_{\frac{1}{q-\mu},q}},
$$
after choosing $\delta>0$ sufficiently small, we get that there exists a $C_{7}>0$ such that
\begin{equation}\label{q5nonzero}
|Q_{4}(v)|\geq C_{7}
\end{equation}
on the support of $\omega$. As a consequence, by Lemma \ref{abr} (a), (c), we deduce
\begin{equation}\label{qdenom2}
\omega \{Q_{4}(v), (Q_{4}(v))^{-1}, Q_{4}^{2}(v)\} \in\mathcal{H}_{p}^{\nu_{0},\frac{1}{2}+\varepsilon_{0}}(\mathbb{B})\oplus\mathbb{C}_{\omega}.
\end{equation}

By \eqref{uniregA}, \eqref{uniregB} and Lemma \ref{abr} (a), we have
\begin{equation}\label{phiterm1}
 \omega |\varphi+v\xi|^{2}\in \mathcal{H}_{p}^{\nu_{0},\frac{1}{2}+\varepsilon_{0}}(\mathbb{B})\oplus\mathbb{C}_{\omega}
\end{equation}
and
\begin{equation}\label{phivxi}
\|\omega (|\varphi+v_{1}\xi|^{2}-|\varphi+v_{2}\xi|^{2})\|_{\mathcal{H}_{p}^{\nu_{0},\frac{1}{2}+\varepsilon_{0}}(\mathbb{B})\oplus\mathbb{C}_{\omega}}\hspace{-6pt}\leq \hspace{-1pt}C_{8}\|v_{1}-v_{2}\|_{(X_{1}^{r},X_{0}^{r})_{\frac{1}{q-\mu},q}}
\end{equation}
for certain $C_{8}>0$. Hence, by the inequality
\begin{eqnarray*}
\lefteqn{\omega ||\varphi|^{2}-|\varphi+v\xi||^{2}}\\
&\leq& \|\omega (|\varphi+u_{0}\xi|^{2}-|\varphi+v\xi|^{2})\|_{C(\mathbb{B})}\leq C_{4}\|u_{0}-v\|_{(X_{1}^{r},X_{0}^{r})_{\frac{1}{q-\mu},q}},
\end{eqnarray*}
after choosing $\delta>0$ appropriately, we deduce that there exists a constant $C_{9}>0$ such that
\begin{equation}\label{phivxibound}
|\varphi+v\xi|^{2}\geq C_{9}
\end{equation}
on the support of $\omega$. From \eqref{uniregA}, \eqref{uniregB}, \eqref{phiterm1} and Lemma \ref{abr} (a), (c), we obtain 
\begin{equation}\label{phiterm}
\omega\{ \langle \varphi, \xi\rangle, |\varphi|^{2}, |\varphi+v\xi|^{-2}, |\varphi+v\xi|^{-2\beta_{\kappa}}, |\varphi|^{2} |\varphi+v\xi|^{-2(\beta_{\kappa}+1)}\}\hspace{-2pt}\in\hspace{-2pt}\mathcal{H}_{p}^{\nu_{0},\frac{1}{2}+\varepsilon_{0}}(\mathbb{B})\oplus\mathbb{C}_{\omega}.
\end{equation}

According to \eqref{contembofint} and Lemma \ref{abr} (a), we have
\begin{equation}\label{xcxv}
\omega |x-\kappa|\partial_{x}:(X_{1}^{r},X_{0}^{r})_{\frac{1}{q-\mu},q}\rightarrow \mathcal{H}_{p}^{\nu_{0},\frac{1}{2}+\varepsilon_{0}}(\mathbb{B})\hookrightarrow C(\mathbb{B}),
\end{equation}
and
\begin{equation}\label{qdenom1}
\omega\{ v+(1+\beta_{j})|x-\kappa|v_x, (v+(1+\beta_{j})|x-\kappa|v_x)^{2} \}\in\mathcal{H}_{p}^{\nu_{0},\frac{1}{2}+\varepsilon_{0}}(\mathbb{B})\oplus\mathbb{C}_{\omega }.
\end{equation}
Combining \eqref{q5nonzero}, \eqref{qdenom2}, \eqref{qdenom1} and Lemma \ref{abr} (a), (c), we have 
\begin{equation}\label{Qtermreg}
\omega \{Q(v), (Q(v))^{-1}\}\in \mathcal{H}_{p}^{\nu_{0},\frac{1}{2}+\varepsilon_{0}}(\mathbb{B})\oplus\mathbb{C}_{\omega }.
\end{equation}
Hence, by \eqref{expregv}, \eqref{phiterm}, \eqref{Qtermreg} and Lemma \ref{abr} (a), it holds
\begin{equation}\label{Q1reg}
\omega Q_{1}(v)\in \mathcal{H}_{p}^{\nu_{0},\frac{1}{2}+\varepsilon_{0}}(\mathbb{B})\oplus\mathbb{C}_{\omega}.
\end{equation}
Similarly, by \eqref{expregv}, \eqref{qdenom2}, \eqref{phiterm} and Lemma \ref{abr} (a), we have
$$
\omega Q_{5}(v)\in \mathcal{H}_{p}^{\nu_{0},\frac{1}{2}+\varepsilon_{0}}(\mathbb{B})\oplus\mathbb{C}_{\omega}.
$$
Hence, by Lemma \ref{abr} (b) each of $Q_{i}(v)$, $i\in\{1,4\}$, acts by multiplication as a bounded map on $X_{0}^{r}$, i.e. $A(v)$ is a well-defined element in $\mathcal{L}(X_{1}^{r},X_{0}^{r})$.

Due to \eqref{phivxi} and \eqref{phivxibound}, there exists a closed smooth path $\Gamma_{3}$ within the domain $\{\lambda\in\mathbb{C}\, |\, \mathrm{Re}(\lambda)>0\}$ containing the sets
$$
\bigcup_{v\in B_{\mu}^{r}(\delta)}\mathrm{Ran}(\omega _{j}|\varphi+v\xi|).
$$
Concerning the term $e^{2\eta}Q_{1}$, we estimate
\begin{eqnarray}\nonumber
\lefteqn{\omega (|\varphi+v_{1}\xi|^{-2\beta_{\kappa}}(Q(v_{1}))^{-1}-|\varphi+v_{2}\xi|^{-2\beta_{\kappa}}(Q(v_{2}))^{-1})}\\\nonumber
&=&\omega (Q(v_{2}))^{-1}(|\varphi+v_{1}\xi|^{-2\beta_{\kappa}}-|\varphi+v_{2}\xi|^{-2\beta_{\kappa}})\\\nonumber
&&+\omega ((Q(v_{1}))^{-1}-(Q(v_{2}))^{-1})|\varphi+v_{1}\xi|^{-2\beta_{\kappa}}\\\nonumber
&=&(Q(v_{2}))^{-1}\frac{\omega}{2\pi i}\int_{\Gamma_{3}}\lambda^{-2\beta_{\kappa}}((\lambda-|\varphi+v_{1}\xi|)^{-1}-(\lambda-|\varphi+v_{2}\xi|)^{-1})d\lambda\\\label{ghias}
&&+\omega |\varphi+v_{1}\xi|^{-2\beta_{\kappa}}(Q(v_{1}))^{-1}(Q(v_{2}))^{-1}(Q(v_{2})-Q(v_{1})).
\end{eqnarray}
Note that 
\begin{eqnarray}\nonumber
\lefteqn{(\lambda-|\varphi+v_{1}\xi|)^{-1}-(\lambda-|\varphi+v_{2}\xi|)^{-1}}\\\nonumber
&=&(\lambda-|\varphi+v_{1}\xi|)^{-1}(\lambda-|\varphi+v_{2}\xi|)^{-1}(|\varphi+v_{2}\xi|+|\varphi+v_{1}\xi|)^{-1}\\\label{ahfsfdfd}
&&\times (v_{1}+v_{2}+2\langle \varphi, \xi\rangle)(v_{1}-v_{2})
\end{eqnarray}
and
\begin{eqnarray}\nonumber
Q(v_{1})-Q(v_{2})&=&\big(2-k((1+\beta_{\kappa})|x-\kappa|)^{\frac{1}{1+\beta_{\kappa}}}(v_{1}+v_{2})\big)\\\nonumber
&&\times k\big((1+\beta_{\kappa})|x-\kappa|\big)^{\frac{1}{1+\beta_{\kappa}}}(v_{2}-v_{1})\\\nonumber
&&+\big(v_{1}+v_{2}+(-1)^{\kappa}(1+\beta_{\kappa})|x-\kappa|\partial_{x}(v_{1}+v_{2})\big)\\\label{Qterm}
&&\times\big(1+(-1)^{\kappa}(1+\beta_{\kappa})|x-\kappa|\partial_{x}\big)(v_{1}-v_{2}).
\end{eqnarray}
Similarly, for the term $e^{2\eta}Q_{5}$, we have
\begin{eqnarray*}
\lefteqn{\omega \big(|\varphi+v_{1}\xi|^{-2(\beta_{\kappa}+1)}(Q_{4}(v_{1}))^{-1}-|\varphi+v_{2}\xi|^{-2(\beta_{\kappa}+1)}(Q_{4}(v_{2}))^{-1}\big)}\\
&=&\omega |\varphi+v_{1}\xi|^{-2(\beta_{\kappa}+1)}(Q_{4}(v_{1}))^{-1}(Q_{4}(v_{2}))^{-1}((Q_{4}(v_{2})-Q_{4}(v_{1}))\\
&+&\frac{\omega(Q_{4}(v_{2}))^{-1}}{2\pi i}\int_{\Gamma_{3}}\lambda^{-2(\beta_{\kappa}+1)}\big((\lambda-|\varphi+v_{1}\xi|)^{-1}-(\lambda-|\varphi+v_{2}\xi|)^{-1}\big)d\lambda.
\end{eqnarray*}
By \eqref{explipsch}, \eqref{Q5lip}, \eqref{phiterm}, \eqref{Qtermreg}, \eqref{ghias}, \eqref{ahfsfdfd}, \eqref{Qterm} and Lemma \ref{abr} (a), (c), arguing as in \cite[p. 1463]{RS3}, the maps below are Lipschitz
\begin{eqnarray}\nonumber
&&\hspace{-50pt}(X_{1}^{r},X_{0}^{r})_{\frac{1}{q-\mu},q} \supset B_{\mu}^{r}(\delta)\ni v\\\label{TzP}
&&\mapsto\omega\{Q^{-1}(v), |\varphi+v\xi|^{-2\beta_{\kappa}}, Q_{i}(v)\} \in \mathcal{H}_{p}^{\nu_{0},\frac{1}{2}+\varepsilon_{0}}(\mathbb{B})\oplus\mathbb{C}_{\omega},
\end{eqnarray}
where $i\in\{1,4,5\}$.
Therefore, the map 
\begin{equation}\label{LipforA}
A(\cdot):B_{\mu}^{r}(\delta)\rightarrow \mathcal{L}(X_{1}^{r}, X_{0}^{r})
\end{equation}
is also Lipschitz.
 
{\em Lipschitz property of $F$}. To show that the map $F:B_{\mu}^{r}(\delta)\rightarrow X_{0}^{r}$ is well defined and Lipschitz, we proceed
arguing as in \cite[p. 1464]{RS3}. To illustrate, let us restrict as before on the support of the cut-off function $\omega$ and examine the terms 
\begin{equation}\label{twoterms}
\omega G(v)(Q(v))^{-1}f_{3,4}|x-\kappa|vv_{x} \quad \text{and} \quad \omega G(v)f_{5,2}|x-\kappa|v_{x}. 
\end{equation}
By \eqref{uniregB}, \eqref{expregv}, \eqref{qdenom2}, \eqref{phiterm}, \eqref{Qtermreg} and Lemma \ref{abr} (a) we obtain 
$$
\omega G(v)(Q(v))^{-1}|x-\kappa|vv_{x}\in \mathcal{H}_{p}^{\nu_{0},\frac{1}{2}+\varepsilon_{0}}(\mathbb{B})\oplus\mathbb{C}_{\omega }. 
$$
Moreover,
\begin{eqnarray}\nonumber
&&\hspace{-60pt}|x-\kappa|(G(v_{1})(Q(v_{1}))^{-1}v_{1}(v_{1})_{x}-G(v_{2})(Q(v_{2}))^{-1}v_{2}(v_{2})_{x})\\\nonumber
&=&(G(v_{1})-G(v_{2}))(Q(v_{1}))^{-1}v_{1}|x-\kappa|(v_{1})_{x}\\\nonumber
&&+G(v_{2})\big(Q(v_{1}))^{-1}-(Q(v_{2}))^{-1}\big)v_{1}|x-\kappa|(v_{1})_{x}\\\nonumber
&&+G(v_{2})(Q(v_{2}))^{-1}(v_{1}-v_{2})|x-\kappa|(v_{1})_{x}\\\label{mixedterm}
&&+G(v_{2})(Q(v_{2}))^{-1}v_{2}|x-\kappa|(v_{1}-v_{2})_{x}.
\end{eqnarray}
By \eqref{explipsch}, \eqref{q5nonzero}, \eqref{TzP}, \eqref{TzP} and Lemma \ref{abr} (a), (c), the map
\begin{equation}\label{GLip}
(X_{1}^{r},X_{0}^{r})_{\frac{1}{q-\mu},q} \supset B_{\mu}^{r}(\delta)\ni v \mapsto \omega G(v) \in \mathcal{H}_{p}^{\nu_{0},\frac{1}{2}+\varepsilon_{0}}(\mathbb{B})\oplus\mathbb{C}_{\omega }
\end{equation}
is Lipschitz. Hence, by \eqref{xcxv}, \eqref{TzP}, \eqref{mixedterm}, \eqref{GLip} we deduce that the map 
\begin{eqnarray}\nonumber
\lefteqn{(X_{1}^{r},X_{0}^{r})_{\frac{1}{q-\mu},q} \supset B_{\mu}^{r}(\delta)\ni v}\\\label{Lipofrestterm}
&& \mapsto \omega|x-\kappa|(G(v)(Q(v))^{-1}vv_{x} \in \mathcal{H}_{p}^{\nu_{0},\frac{1}{2}+\varepsilon_{0}}(\mathbb{B})\oplus\mathbb{C}_{\omega }
\end{eqnarray}
is Lipschitz. Since
$$f_{3,4}\in X_{0}^{r},$$
the result follows by \eqref{Lipofrestterm} and Lemma \ref{abr} (b). The Lipschitz property for the second term in \eqref{twoterms} follows by \eqref{xcxv}, \eqref{GLip} and the fact that
$$f_{5,2}\in X_{0}^{r}.$$
These imply Lipschitz property for
\begin{equation}\label{LipforF}
F:B_{\mu}^{r}(\delta)\rightarrow X_{0}^{r}.
\end{equation}

{\em Maximal $L^{q}$-regularity for $A(u_{0})$}. We will show that for each $v\in B_{\mu}^{r}(\delta)$ with $\mathrm{Im}(v)=0$, the operator $A(v)\in\mathcal{L}(X_{1}^{r},X_{0}^{r})$ has maximal $L^{q}$-regularity. By Lemma \ref{wperturb}, \eqref{phivxibound} and \eqref{Q1reg}, for each $\theta\in (\pi/2,\pi)$ there exists a $c_{0}>0$ such that $c_{0}+Q_{1}(v)\underline{A}_{r}\in\mathcal{R}(\theta)$. By \cite[(I.2.5.2), (I.2.9.6)]{Am} and Lemma \ref{abr} (d), we have 
\begin{equation}\label{domfracpow}
\mathcal{D}((c_{0}+Q_{1}(v)\underline{A}_{r})^{\varepsilon_{1}})\hookrightarrow \mathcal{H}^{r+2\varepsilon_{1}-\varepsilon,\ell+2\varepsilon_{1}-\varepsilon}_p(\mathbb{B})\oplus\mathbb{C}_{\omega },
\end{equation}
for any $\varepsilon_{1}\in(1/2,1)$ and $\varepsilon>0$.

The first term on the right-hand side of \eqref{newexpA} can be written in the form
$$
\omega \Big(Q_{3}(v)+Q_{1}(v)\frac{\beta_{\kappa}}{1+\beta_{\kappa}}\frac{(-1)^{\kappa+1}}{|x-\kappa|}\Big)\partial_{x}=Q_{6}(v)\frac{(-1)^{\kappa+1}}{|x-\kappa|}\partial_{x},
$$
where
$$
Q_{6}(v)=\omega e^{-2\eta(v)}|\varphi+v\xi|^{-2\beta_{\kappa}}\frac{\beta_{\kappa}}{1+\beta_{\kappa}}\Big(\frac{|\varphi|^{2}}{|\varphi+v\xi|^{2}Q_{4}(v)}-\frac{1}{Q(v)}\Big).
$$
Keeping in mind that
$$\varphi(j)=(1,0)\quad\text{and}\quad\xi(j)=(0,1),$$
using \eqref{expregv}, \eqref{qdenom2}, \eqref{phiterm}, \eqref{Qtermreg} and Lemma \ref{abr}, we get
$$
\omega Q_{6}(v)\in \mathcal{H}_{p}^{\nu_{0},\frac{1}{2}+\varepsilon_{0}}(\mathbb{B}).
$$
For each $\varepsilon_{2}\in (0,\varepsilon_{0})$, we have
$$
\omega |x-\kappa|^{-\varepsilon_{2}}Q_{6}(v) \in \mathcal{H}_{p}^{\nu_{0},\frac{1}{2}+\varepsilon_{0}-\varepsilon_{2}}(\mathbb{B})
$$
and by Lemma \ref{abr} (b) this term acts by multiplication as a bounded map on $X_{0}^{r}$. Therefore, after choosing $\varepsilon_{2}$ and $\varepsilon_{1}$ such that $\varepsilon_{1}>1-\varepsilon_{2}/2$, by \eqref{domfracpow} we conclude that 
\begin{equation}
\omega \Big(Q_{3}(v)+Q_{1}(v)\frac{\beta_{\kappa}}{1+\beta_{\kappa}}\frac{(-1)^{\kappa+1}}{|x-\kappa|}\Big)\partial_{x}\in \mathcal{L}(\mathcal{D}((c_{0}+Q_{1}(v)\underline{A}_{r})^{\varepsilon_{1}}),X_{0}^{r}).
\end{equation}

Moreover, by \eqref{expregv}, \eqref{qdenom2}, \eqref{phiterm}, \eqref{Qtermreg} and Lemma \ref{abr} (a), we have 
$$
(1-\omega )(Q_{2}(v)+Q_{3}(v))\in \bigcap_{\nu>0} \mathcal{H}_{p}^{\nu_{0},\nu}(\mathbb{B}).
$$
By Lemma \ref{abr} (b), the term
$(1-\omega )(Q_{2}(v)+Q_{3}(v))$
acts by multiplication as a bounded map on $X_{0}^{r}$. Hence, 
\begin{equation}\label{ACDC}
(1-\omega )(Q_{2}(v)+Q_{3}(v))\partial_{x}\in \mathcal{L}(\mathcal{D}((c_{0}+Q_{1}(v)\underline{A}_{r})^{\varepsilon_{1}}),X_{0}^{r}).
\end{equation}
Therefore, by writing 
\begin{eqnarray*}
\lefteqn{\Big(\omega \Big(Q_{3}(v)+Q_{1}(v)\frac{\beta_{\kappa}}{1+\beta_{\kappa}}\frac{(-1)^{\kappa+1}}{|x-\kappa|}\Big)\partial_{x}}\\
&&+(1-\omega )(Q_{2}(v)+Q_{3}(v))\partial_{x}\Big)(c_{0}+c+Q_{1}(v)\underline{A}_{r})^{-1}\\
&=&\Big(\omega \Big(Q_{3}(v)+Q_{1}(v)\frac{\beta_{\kappa}}{1+\beta_{\kappa}}\frac{(-1)^{\kappa+1}}{|x-\kappa|}\Big)\partial_{x}\\
&&+(1-\omega )(Q_{2}(v)+Q_{3}(v))\partial_{x}\Big)(c_{0}+Q_{1}(v)\underline{A}_{r})^{-\varepsilon_{1}}\\
&&\times(c_{0}+Q_{1}(v)\underline{A}_{r})^{\varepsilon_{1}}(c_{0}+c+Q_{1}(v)\underline{A}_{r})^{-1},
\end{eqnarray*}
using \cite[Lemma 2.3.3]{Tan} together with the perturbation result of Kunstmann and Weis \cite[Theorem 1]{KuWe} and choosing $c>0$ sufficiently large, we deduce that
\begin{equation}\label{maxregAu}
A(v)+c_{0}+c\in\mathcal{R}(\theta). 
\end{equation}
Note that as $v$ varies in $B_{\mu}^{r}(\delta)$, we only have to choose, if necessary, $c_{0}$ and $c$ larger. Then, maximal $L^{q}$-regularity for each $A(v)$, $v\in \{B_{\mu}^{r}(\delta)\, |\, \mathrm{Im}(v)=0\}$, is obtained by Theorem \ref{KalWeiTh}.

{\em Existence and uniqueness}. From the above conclusions, choosing $\mu=0$, it follows that there exists a $T>0$ and a unique $u$ as in \eqref{exanduniqofu} solving the problem \eqref{flqz1}-\eqref{flqz2}. By \eqref{interpemb}, \eqref{contembofint} and Lemma \ref{abr} (a), we also obtain the regularity \eqref{extrareginspace1}-\eqref{extrareginspace0} of $u$. Moreover, by taking the complex conjugate in \eqref{flqz1}-\eqref{flqz2}, and using the above uniqueness result, we conclude that $u(t)\in \mathbb{R}$, $t\in[0,T]$.

\begin{remark}\label{unifofdata} 
In the proofs of \eqref{LipforA}, \eqref{LipforF} and \eqref{maxregAu}, we choose $\delta>0$ in $B_{\mu}^{r}(\delta)$ sufficiently small, so that the equations \eqref{q5nonzero}, \eqref{phivxibound} are satisfied, and moreover to ensure the existence of the paths $\Gamma_{1}$, $\Gamma_{2}$. As a matter of fact,
\eqref{LipforA}, \eqref{LipforF} and \eqref{maxregAu} still hold true, if we replace
$$\{v\in B_{\mu}^{r}(\delta)\, | \, \mathrm{Im}(v)=0\}$$ with
$$\{v\in B_{\mu_{0}}^{r_{0}}(\delta_{0})\, | \, \mathrm{Im}(v)=0\}\cap B_{\mu}^{r}(\delta),$$
with fixed $\mu_{0}$, $r_{0}$, $\delta_{0}$, where $\delta_{0}$ is sufficiently small and arbitrary $\mu$, $r$, $\delta$.
\end{remark}

{\em Smoothing in space}. Suppose now that $b$ in \eqref{pqgammachoice0} is infinite. We will show that the solution becomes instantaneously smooth in space. We apply \cite[Theorem 3.1]{RS4} to \eqref{flqz1}-\eqref{flqz2} with $A(\cdot)$ and $F$ as before, the Banach scales
$$
Y_{0}^{n}=X_{0}^{r+\frac{n}{q\eta_{0}}}\quad\text{and}\quad Y_{1}^{n}=X_{1}^{r+\frac{n}{q\eta_{0}}},
$$
where $n\in \mathbb{N}\cup\{0\}$ and
$Z=\{v\in B_{0}^{r}(\delta_{0})\, | \, \mathrm{Im}(v)=0\}$
for certain $\delta_{0}>0$ sufficiently small. Moreover, we choose 
\begin{equation}\label{eta0choice}
\eta_{0}>\frac{1}{2}\max\Big\{1,\frac{q-\mu}{\mu}\Big\},
\end{equation}
where $\mu>0$ satisfies \eqref{sigmachoice}. Clearly, for $\varepsilon>0$ small enough, we have
$$
\mathcal{H}_{p}^{r+\frac{n}{q\eta_{0}}+2,\ell+2}(\mathbb{B})\hookrightarrow \mathcal{H}_{p}^{r+\frac{n+1}{q\eta_{0}}+2-\frac{2}{q}+\varepsilon,\ell+2-\frac{2}{q}+\varepsilon}(\mathbb{B}).
$$
Hence, by Lemma \ref{abr} (d), we deduce
$$Y_{1}^{n}\hookrightarrow (Y_{1}^{n+1},Y_{0}^{n+1})_{{1}/{q},q}$$
for each $n\in \mathbb{N}\cup\{0\}$. We examine now the validity of the conditions (i), (ii) and (iii) of \cite[Theorem 3.1]{RS4} separately.

{\em Condition} (i). Due to \eqref{extrareginspace0}, by possibly choosing a smaller $T>0$, we can achieve that $u(t)\in Z$ for all $t\in[0,T]$. Hence, \eqref{extrareginspace0} and \eqref{LipforA} imply continuity of
$$
[0,T]\ni t\mapsto A(u(t)) \in \mathcal{L}(Y_{1}^{0},Y_{0}^{0}).
$$
In addition, by \eqref{maxregAu} and Theorem \ref{KalWeiTh}, we deduce that for each $t\in[0,T]$ the operator
$$
A(u(t))\in \mathcal{L}(Y_{1}^{0},Y_{0}^{0})
$$
has maximal $L^{q}$-regularity. 

{\em Condition} (ii). Let $w\in Z\cap (Y_{1}^{n},Y_{0}^{n})_{{1}/{q},q}$ for some $n\in\mathbb{N}\cup\{0\}$. By Lemma \ref{abr} (d) and \eqref{eta0choice}, we have
\begin{eqnarray}\nonumber
(Y_{1}^{n},Y_{0}^{n})_{\frac{1}{q},q}&\hookrightarrow& \mathcal{H}_{p}^{r+\frac{n}{q\eta_{0}}+2-\frac{2}{q}-\varepsilon,\ell+2-\frac{2}{q}-\varepsilon}(\mathbb{B})\oplus\mathbb{C}_{\omega }\\\nonumber
&\hookrightarrow& \mathcal{H}_{p}^{r+\frac{n+1}{q\eta_{0}}+2-\frac{2}{q-\mu}+\varepsilon,\ell+2-\frac{2}{q-\mu}+\varepsilon}(\mathbb{B})\oplus\mathbb{C}_{\omega } \\\label{intersmoothemb}
&\hookrightarrow& (Y_{1}^{n+1},Y_{0}^{n+1})_{\frac{1}{q-\mu},q},
\end{eqnarray}
for $\varepsilon>0$ small enough. For fixed $n\in\mathbb{N}\cup\{0\}$, we can choose $\delta>0$ such that 
$$
w\in B_{\mu}^{r+(n+1)/(q\eta_{0})}(\delta).
$$ 
Consequently, by Remark \ref{unifofdata} and \eqref{intersmoothemb}, the map
$$
A(w):Y_{1}^{n+1}\rightarrow Y_{0}^{n+1}
$$
is a well defined and has maximal $L^{q}$-regularity. Similarly, if
$$
h\in C([0,T];Z\cap (Y_{1}^{n},Y_{0}^{n})_{{1}/{q},q}),
$$
then by the equation \eqref{intersmoothemb}, for any $n\in\mathbb{N}\cup\{0\}$ we can choose the radius $\delta$ of $B_{\mu}^{r+(n+1)/(q\eta_{0})}(\delta)$ large enough such that 
$$
h(t)\in B_{\mu}^{r+(n+1)/(q\eta_{0})}(\delta),
$$ 
for each $t\in[0,T]$. Therefore, by Remark \ref{unifofdata} and \eqref{intersmoothemb}, it follows that
$$
[0,T]\ni t\mapsto A(h(t)) \in \mathcal{L}(Y_{1}^{n+1},Y_{0}^{n+1})
$$
is continuous for any $n\in\mathbb{N}\cup\{0\}$.

{\em Condition} (iii). Immediately follows by Remark \ref{unifofdata} and \eqref{intersmoothemb}. 

By \cite[Theorem 3.1]{RS4}, besides \eqref{extrareginspace0}-\eqref{extrareginspace1}, the solution $u$ satisfies additionally the regularity \eqref{extrareginspace3} with $n=1$, for each $\tau\in(0,T)$. Then, \eqref{extrareginspace2} follows by \eqref{interpemb} and \eqref{contembofint}.

{\em Smoothing in time}. The space smoothness of the solution, provided by \eqref{extrareginspace2} and \eqref{extrareginspace3} with $n=1$, immediately implies smoothness in time as well. Let us highlight how to prove \eqref{extrareginspace3} for each $n\in\mathbb{N}$. It suffices to prove our claim on the support of $\omega$. By differentiating \eqref{flqz1} with respect to $t$, we obtain
\begin{equation}\label{diffovert}
u_{tt}=-A(u)u_{t}-(\partial_{t}A(u))u+\partial_{t}F(u), \quad t\in(\tau,T).
\end{equation}
Let $\nu>2+1/p$. According to Remark \ref{unifofdata}, we may choose $\delta>0$ sufficiently large such that 
$$
u(t)\in \{v\in B_{0}^{r}(\delta_{0})\, |\, \mathrm{Im}(v)=0\}\cap B_{0}^{\nu}(\delta),
$$ 
for each $t\in[\tau,T]$. Due to \eqref{extrareginspace3} with $n=1$, we have
$$
\omega\{ |x-\kappa|^{-1}u_{xt}, u_{xxt}\}\in L^{q}(\tau,T; \mathcal{H}_{p}^{\nu-2,\ell-2}(\mathbb{B})).
$$
Hence, due to \eqref{interpemb}, \eqref{extrareginspace3} with $n=1$, \eqref{TzP}, \eqref{TzP} and Lemma \ref{abr} (b) we have 
$$
\omega \{|x-\kappa|^{-1}Q_{1}(u)u_{xt}, Q_{1}(u)u_{xxt}, Q_{3}(u)u_{xt}\}\in L^{q}(\tau,T; \mathcal{H}_{p}^{\nu-2,\ell-2}(\mathbb{B})).
$$
Therefore
\begin{equation}\label{treg1}
A(u)u_{t}\in L^{q}(\tau,T; \mathcal{H}_{p}^{\nu-2,\ell-2}(\mathbb{B})).
\end{equation}
Concerning the second term on the right-hand side of \eqref{diffovert}, we have 
$$
\omega \partial_{t}A(u)=\omega \Big(\big(Q_{7}(u)u_{t}+Q_{8}(u)u_{xt}\big)\partial_{x}^{2}+\big(Q_{9}(u)u_{t}+Q_{10}(u)u_{xt}\big)\frac{\partial_{x}}{|x-\kappa|}\Big).
$$
Similarly to \eqref{TzP} and \eqref{TzP}, the maps
\begin{equation}\label{kkajsoc1}
(X_{1}^{\nu},X_{0}^{\nu})_{\frac{1}{q},q} \supset B_{0}^{r}(\delta_{0})\cap B_{0}^{\nu}(\delta)\ni v\mapsto \omega Q_{i}(v)\in \mathcal{H}_{p}^{\nu+1-\frac{2}{q}-\varepsilon,\frac{1}{2}+\varepsilon_{0}}(\mathbb{B})\oplus\mathbb{C}_{\omega},
\end{equation}
where $i\in\{7,8,9,10\}$ and $\varepsilon>0$, are Lipschitz. Moreover
\begin{equation}\label{parto2wdfu}
u_{t}\in L^{q}(\tau,T; \mathcal{H}_{p}^{\nu,\ell}(\mathbb{B})) \quad \text{and} \quad u_{xt}\in L^{q}(\tau,T; \mathcal{H}_{p}^{\nu-1,\ell-1}(\mathbb{B})).
\end{equation}
In addition, by \eqref{extrareginspace2}
\begin{equation}\label{partofssbhu1}
u_{x}\in \bigcap_{\varepsilon>0}C([\tau,T]; \mathcal{H}_{p}^{\frac{1}{\varepsilon},\ell+1-\frac{2}{q}-\varepsilon}(\mathbb{B})) \,\, \text{and} \,\, u_{xx}\in \bigcap_{\varepsilon>0}C([\tau,T]; \mathcal{H}_{p}^{\frac{1}{\varepsilon},\ell-\frac{2}{q}-\varepsilon}(\mathbb{B})).
\end{equation}
Let $d$ be as in \eqref{tangent}. By writing $u_{t}u_{xx}=(u_{t}/d)(du_{xx})$ and $u_{xt}u_{xx}=(u_{xt}/d)(du_{xx})$,
from \eqref{parto2wdfu}, \eqref{partofssbhu1}, and Lemma \ref{abr} (b) we get
$$
u_{t}u_{xx}, u_{xt}u_{xx}, \omega |x-\kappa|^{-1}u_{t}u_{x}\in L^{q}(\tau,T; \mathcal{H}_{p}^{\nu-1,\ell-2}(\mathbb{B})).
$$
Hence, by \eqref{interpemb}, \eqref{extrareginspace3} with $n=1$, \eqref{kkajsoc1} and Lemma \ref{abr} (b), we obtain 
\begin{equation}\label{treg2}
(\partial_{t}A(u))u\in L^{q}(\tau,T; \mathcal{H}_{p}^{\nu-1,\ell-2}(\mathbb{B})).
\end{equation}

For the third term on the right-hand side of \eqref{diffovert}, we have
$$
\partial_{t}F(u)=Q_{11}(u)u_{t}+Q_{12}(u)u_{xt}.
$$
Similarly to \eqref{LipforF}, the maps
\begin{equation}\label{akpslodc}
(X_{1}^{\nu},X_{0}^{\nu})_{\frac{1}{q},q} \supset B_{0}^{r}(\delta_{0})\cap B_{0}^{\nu}(\delta)\ni v\mapsto \omega Q_{i}(v)\in\mathcal{H}_{p}^{\nu,\ell}(\mathbb{B}),
\end{equation}
where $i\in\{11,12\}$, are Lipschitz. By writing 
$$
Q_{11}(u)u_{t}=(dQ_{11}(u))(u_{t}/d) \quad \text{and} \quad Q_{12}(u)u_{xt}=(dQ_{12}(u))(u_{xt}/d),
$$
and taking into account \eqref{interpemb}, \eqref{extrareginspace3} with $n=1$, \eqref{parto2wdfu}, \eqref{akpslodc} and Lemma \ref{abr} (b), we deduce that
\begin{equation}\label{treg3}
\partial_{t}F(u)\in L^{q}(\tau,T; \mathcal{H}_{p}^{\nu-1,\ell-2}(\mathbb{B})).
\end{equation}
By \eqref{diffovert}, \eqref{treg1}, \eqref{treg2} and \eqref{treg3}, we conclude that 
$$
u\in H^{2,q}(\tau,T; \mathcal{H}_{p}^{\nu-2,\ell-2}(\mathbb{B})),
$$
which is \eqref{extrareginspace3} for $n=2$. The proof for arbitrary $n\in\mathbb{N}$, follows by iterating the above procedure and using the following fact: for each $n\in\mathbb{N}$, due to smoothness in space, we can choose $\nu$ sufficiently large and restrict the ground space to the space of smaller weight $\mathcal{H}_{p}^{\nu,\ell-2(n-1)}(\mathbb{B})$, to make sure that Lemma \ref{abr} (b) can be applied to treat the non-linearities; see e.g. \cite[Section 5.3]{RS4} for more details.
\end{proof}

We return now to the initial variable $s$ and recover the regularity of the distance function $w$. From \eqref{xtosderivartves}, Theorem \ref{short1d} and the definition of Mellin-Sobolev spaces we immediately obtain the following.

\begin{theorem}[$s$-variable regularity for $w$]\label{wexreg}
Let $\gamma$ be a curve satisfying the conditions {\rm(a)-(d)} of Section {\rm\ref{assumptongama}} and let $p$, $q$, $r$, $\ell_{0}$, $\ell$ be as in \eqref{ell0}, \eqref{pqgammachoice0} and \eqref{pqgammachoice}. Then, there exists a $T>0$ and a unique
$$
w\in H^{1,q}(0,T; \mathcal{H}_{p}^{r,\frac{3}{2}+(\ell-\frac{1}{2})(1+\beta_{1})}(\mathbb{B})) \cap L^{q}(0,T; \mathcal{H}_{p}^{r+2,\frac{3}{2}+(\ell+\frac{3}{2})(1+\beta_{0})}(\mathbb{B})\oplus\mathbb{S}_{\omega})
$$
solving \eqref{kwer1}-\eqref{kwer2}. The function $w$ also satisfies 
\begin{eqnarray}\label{wreg3}
\lefteqn{w\in C([0,T]; \mathcal{H}_{p}^{r+2-\frac{2}{q}-\varepsilon,\frac{3}{2}+(\ell+\frac{3}{2}-\frac{2}{q})(1+\beta_{0})-\varepsilon}(\mathbb{B})\oplus\mathbb{S}_{\omega})}\\\nonumber
&& \cap \, C([0,T]; (\mathcal{H}_{p}^{r+2,\frac{3}{2}+(\ell+\frac{3}{2})(1+\beta_{0})}(\mathbb{B})\oplus\mathbb{S}_{\omega}, \mathcal{H}_{p}^{r,\frac{3}{2}+(\ell-\frac{1}{2})(1+\beta_{1})}(\mathbb{B}))_{1/q,q})
\end{eqnarray}
for all $\varepsilon>0$. In particular, if $b=\infty$, then for any $\tau\in(0,T)$, $\nu,\varepsilon>0$ and $n\in\mathbb{N}$, the solution $w$ satisfies
\begin{eqnarray}\label{wreg4}
\lefteqn{ w\in C([\tau,T]; \mathcal{H}_{p}^{\nu,\frac{3}{2}+(\ell+\frac{3}{2}-\frac{2}{q})(1+\beta_{0})-\varepsilon}(\mathbb{B})\oplus\mathbb{S}_{\omega})}\\\label{wreg5}
&&\cap \, H^{n,q}(\tau,T; \mathcal{H}_{p}^{\nu,\frac{3}{2}+(\ell+\frac{3}{2}-2n)(1+\beta_{1})}(\mathbb{B}))\\\nonumber
&&\cap \, L^{q}(\tau,T; \mathcal{H}_{p}^{\nu,\frac{3}{2}+(\ell+\frac{3}{2})(1+\beta_{0})}(\mathbb{B})\oplus\mathbb{S}_{\omega}).
\end{eqnarray}
\end{theorem}

Considering $\omega_{j}w$ instead of $w$, we obtain a version of Theorem \ref{wexreg} with sharper weights in the corresponding Mellin-Sobolev spaces. In particular, we get the following asymptotic behavior for the solution and its derivatives.

\begin{corollary}[$s$-regularity and asymptotic behavior of $w$]\label{pb1}
Let $w$ be the unique solution of \eqref{kwer1}-\eqref{kwer2} given by Theorem {\rm\ref{wexreg}}. Then for each endpoint $j\in\{0,1\}$, we have
\begin{eqnarray*}
\lefteqn{\omega_{j}w\in C([0,T]; \mathcal{H}_{p}^{r+2-\frac{2}{q}-\varepsilon,\frac{3}{2}+(\ell+\frac{3}{2}-\frac{2}{q})(1+\beta_{j})-\varepsilon}(\mathbb{B})\oplus\mathbb{S}_{\omega})}\\\nonumber
&& \cap \, C([0,T]; (\mathcal{H}_{p}^{r+2,\frac{3}{2}+(\ell+\frac{3}{2})(1+\beta_{j})}(\mathbb{B})\oplus\mathbb{S}_{\omega}, \mathcal{H}_{p}^{r,\frac{3}{2}+(\ell-\frac{1}{2})(1+\beta_{j})}(\mathbb{B}))_{1/q,q})\\
&&\cap\, H^{1,q}(0,T; \mathcal{H}_{p}^{r,\frac{3}{2}+(\ell-\frac{1}{2})(1+\beta_{j})}(\mathbb{B})) \cap L^{q}(0,T; \mathcal{H}_{p}^{r+2,\frac{3}{2}+(\ell+\frac{3}{2})(1+\beta_{j})}(\mathbb{B})\oplus\mathbb{S}_{\omega})
\end{eqnarray*}
for all $\varepsilon>0$. In particular, if $b=\infty$, then for any $\tau\in(0,T)$, $\nu,\varepsilon>0$ and $n\in\mathbb{N}$, it holds
\begin{eqnarray*}
\lefteqn{\omega_{j}w\in C([\tau,T]; \mathcal{H}_{p}^{\nu,\frac{3}{2}+(\ell+\frac{3}{2}-\frac{2}{q})(1+\beta_{j})-\varepsilon}(\mathbb{B})\oplus\mathbb{S}_{\omega})}\\
&&\cap \, H^{n,q}(\tau,T; \mathcal{H}_{p}^{\nu,\frac{3}{2}+(\ell+\frac{3}{2}-2n)(1+\beta_{j})}(\mathbb{B}))\\
&&\cap \, L^{q}(\tau,T; \mathcal{H}_{p}^{\nu,\frac{3}{2}+(\ell+\frac{3}{2})(1+\beta_{j})}(\mathbb{B})\oplus\mathbb{S}_{\omega}).
\end{eqnarray*}
Moreover, close to the boundary point $j\in\{0,1\}$, we have:
\begin{enumerate}[\rm(a)]
\item For any $t\in [0,T]$ and any $\varepsilon>0$, it holds
\begin{eqnarray*}
|w(s,t)-c_{j}(t)|s-j||&\leq& C_{1,j}|s-j|^{1+(\ell+\frac{3}{2}-\frac{2}{q})(1+\beta_{j})-\varepsilon},\\
|w_s(s,t)-(-1)^{j}c_{j}(t)|&\leq& C_{2,j}|s-j|^{(\ell+\frac{3}{2}-\frac{2}{q})(1+\beta_{j})-\varepsilon},\\
|w_{ss}(s,t)|&\leq& C_{3,j}|s-j|^{-1+(\ell+\frac{3}{2}-\frac{2}{q})(1+\beta_{j})-\varepsilon},
\end{eqnarray*}
where $c_{j}=u(j,\cdot)\in C([0,T];\mathbb{R})$, with $u$ the solution in Theorem {\rm\ref{short1d}}, and $C_{1,j}$, $C_{2,j}$, $C_{3,j}$ are positive constants depending only on $q$, $r$, $\ell$, $T$, $\varepsilon$.
\medskip
\item For any fixed $\varepsilon>0$, after choosing $q$ sufficiently large and $\ell$ sufficiently close to $1/2+\ell_{0}$, there exists a time $0<T_{\varepsilon}\le T$ is as in Theorem {\rm\ref{wexreg}}, such that for any $t\in [0,T_{\varepsilon}]$ it holds
\begin{eqnarray*}
|w(s,t)-c_{j}(t)|s-j||\leq C_{4,j}|s-j|^{3+2\beta_{j}+\ell_{0}(1+\beta_{j})-\varepsilon},\\
|w_s(s,t)-(-1)^{j}c_{j}(t)|\leq C_{5,j}|s-j|^{2+2\beta_{j}+\ell_{0}(1+\beta_{j})-\varepsilon},\\
|w_{ss}(s,t)|\leq C_{6,j}|s-j|^{1+2\beta_{j}+\ell_{0}(1+\beta_{j})-\varepsilon},
\end{eqnarray*}
where $C_{4,j}$, $C_{5,j}$, $C_{6,j}$ are positive constants depending only on $r$, $\varepsilon$, $T_{\varepsilon}$.
\end{enumerate}
\end{corollary}

\begin{corollary}\label{gammaflow}
The equation \eqref{Ioannina1} has a unique solution and the evolved curves have the same regularity as the solution $w$ given in Theorem {\rm \ref{wexreg}} or Corollary {\rm\ref{pb1}}.
\end{corollary}
\begin{proof}
The existence and uniqueness follow by Theorem \ref{wexreg}. By \eqref{gflow}, near the tips, the evolved curves are given by
$$
\mathbb{B}\times [0,T]\ni (s,t) \mapsto \sigma(s,t)=\gamma(s)+w(s,t)\xi(s).
$$
Observe that
$$a-1>3/2+(\ell+3/2)(1+\beta_{j})>1/2.$$
Then for $\gamma$ we have 
$$
\gamma_{1},\gamma_{2}\in \mathcal{H}_{p}^{b,\alpha}(\mathbb{B})\oplus\mathbb{S}_{\omega} \hookrightarrow \mathcal{H}_{p}^{r+2,\frac{3}{2}+(\ell+\frac{3}{2})(1+\beta_{j})}(\mathbb{B})\oplus\mathbb{S}_{\omega}.
$$
On the other hand
$$
\omega\{\xi_{1},\xi_{2}\}\in \mathcal{H}_{p}^{r+2,\alpha-1}(\mathbb{B})\oplus\mathbb{C}_{\omega}\hookrightarrow \mathcal{H}_{p}^{r+2,\frac{3}{2}+(\ell+\frac{3}{2})(1+\beta_{j})}(\mathbb{B})\oplus\mathbb{C}_{\omega}.
$$
The result follows by Lemma \ref{abr} (a), (b) and Corollary \ref{pb1}.
\end{proof}

By Lemma \ref{abr} (a) and \eqref{wreg3}, we find the following asymptotics for the graphical function $w$ and the geodesic curvature $k_g$ of the evolved curve. 

\begin{theorem}\label{kgregularitys}
Let $\gamma$ be a curve satisfying the conditions {\rm(a)-(d)} of Section {\rm\ref{assumptongama}}, let $p$, $q$, $r$, $\ell_{0}$, $\ell$ be as in \eqref{ell0}, \eqref{pqgammachoice0}, \eqref{pqgammachoice},
and $w:\mathbb{B}\times[0,T]\rightarrow\mathbb{R}$ be as in Theorem {\rm\ref{wexreg}}. Then:
\begin{enumerate}[\rm(a)]
\item
The function $w$ is
$C^{1}$-smooth up to the boundary and
\begin{equation*}\label{hhkklase}
\lim_{s\rightarrow j}w_{t}(s,t)=0
\end{equation*}
for each $j\in\{0,1\}$ and $t\in[0,T]$.
\medskip
\item
For any fixed $\varepsilon>0$, after choosing $q$ sufficiently large and $\ell$ sufficiently close to $1/2+\ell_{0}$, there exists a time $T_\varepsilon\in(0,T]$
such that close to $j\in\{0,1\}$, the scalar geodesic curvature $k_{\gind}$ has the following asymptotic behaviour
$$
|k_{\gind}(s,t)|\le c_j|s-j|^{(1+\ell_0)(1+\beta_j)-\varepsilon}
$$
where $c_j> 0$ is a positive constant depending only on $r$, $\varepsilon$ and $T_{\varepsilon}$. In particular, if $\alpha\geq5/2$, then
\begin{equation*}
|k_{\gind}(s,t)|\leq\left\{
\begin{array}{lll}
c_j|s-s_j|^{1+\beta_j-\varepsilon} &\text{if} &\beta_1\leq -1/2,\\
c_j|s-s_j|^{-\beta_1\frac{1+\beta_j}{1+\beta_1}-\varepsilon} &\text{if}& \beta_1>-1/2,
\end{array}
\right.
\end{equation*}
for some constant $c_{j}\geq 0$ depending only on $r$, $\varepsilon$ and $T_{\varepsilon}$.
\end{enumerate}
\end{theorem}
\begin{proof}
Recall at first that close to the boundary points of the interval $\mathbb{B}$, the solution $w$ can be written in the form
$$w=|s-j|u,$$
and the initial curve $\gamma$ in the form
$$\gamma=|s-j|\varphi.$$
(a) For sufficiently small $\varepsilon>0$, by \eqref{wreg3} we have
$$
w_{s},w_{ss}\in C([0,T]; \mathcal{H}_{p}^{r-\frac{2}{q}-\varepsilon,-\frac{1}{2}+(\ell+\frac{3}{2}-\frac{2}{q})(1+\beta_{0})-\varepsilon}(\mathbb{B}))\hookrightarrow C([0,T];C(0,1)),
$$
where we have used the fact that $r\ge 1$ and Lemma \ref{abr} (a) for the continuous embedding. Hence, the right hand side of \eqref{kwer1} belongs to $C([0,T];C(0,1))$, which implies that $w_{t}$ is continuous on $(0,1)\times [0,T]$.

From Corollary \ref{pb1} (b), it follows that $w_{s}$ is continuous in space-time up to the boundary. By \eqref{pqgammachoice}, we have $\ell_{0}>-1$. Consequently,
$$1+\ell_{0}(1+\beta_{j})>-\beta_j>0.$$
Hence, again from Corollary \ref{pb1} (b), we deduce that
$$
\lim_{s\to j}|s-j|^{-2\beta_j}w_{ss}\to 0,
$$
for any $t\in[0,T]$. Moreover, by a straightforward computation, we see that
\begin{eqnarray*}
\mathcal{B}_1&=&\frac{w_s\langle\gamma,\gamma_s\rangle-(1-kw)\langle\gamma,\xi\rangle-w(1-kw)}{|\gamma+w\xi|^2(1-kw)}\\
&=&\frac{-gu+ku^2+(-1)^j\langle\varphi,\varphi_s\rangle u+|s-j|\langle\varphi,\varphi_s\rangle u_s+(-1)^j|\varphi|^2u_s}
{|\varphi+u\xi|^2(1-|s-j|ku)}\nonumber\\
&&-\frac{\rho}{|\varphi+u\xi|^2},\nonumber
\end{eqnarray*}
where $g$ and $\rho$ are defined in \eqref{greqexp}. Taking into account the assumptions on the parameters $b$ and $\alpha$,
we have that
$$
\lim_{s\to j}|s-j|^{-2\beta_j}\mathcal{B}_1= 0,
$$
for any $t\in[0,T]$. Moreover, for the terms $\mathcal{B}_2$ and $\mathcal{B}_3$ given by
\begin{eqnarray*}
\mathcal{B}_2&=&\frac{2kw^2_s+k_sw_sw+k(1-kw)^2}{(1-kw)\big((1-kw)^2+w_s^2\big)},\\
\mathcal{B}_3&=&\frac{w_s\langle Dh|_{\gamma+w\xi},\gamma_s\rangle}{1-kw}-\langle Dh|_{\gamma+w\xi},\xi\rangle,
\end{eqnarray*}
we easily deduce that
$$
\lim_{s\to j}|s-j|^{-2\beta_j}\mathcal{B}_2=0
$$
and
$$
\lim_{s\to j}|s-j|^{-2\beta_j}\mathcal{B}_3=0,
$$
for any $j\in\{0,1\}$.
Consequently, from \eqref{kwer1}, it follows that $w_{t}$ is continuous in space-time up to the boundary and
$$
\lim_{s\to j} w_t(s,t)=0
$$
holds.

(b) Let us examine now the asymptotic behavior of each term on the right hand side of \eqref{s5}. In particular, it suffices
to investigate the asymptotic behavior of the terms
\begin{eqnarray*}
\mathcal{C}_1&=&\frac{1-kw}{\big((1-kw)^2+w_s^2\big)^{3/2}}\,w_{ss},\\
\mathcal{C}_2&=&\frac{2kw^2_s+k_sw_sw+k(1-kw)^2}{\big((1-kw)^2+w_s^2\big)^{3/2}},\\
\mathcal{C}_3&=&\frac{-(1-kw)\langle\gamma,\xi\rangle}{|\gamma+w\xi|^2\big((1-kw)^2+w^2_s\big)^{1/2}},\\
\mathcal{C}_4&=&\frac{w_s\langle\gamma,\gamma_s\rangle-w(1-kw)}{|\gamma+w\xi|^2\big((1-kw)^2+w^2_s\big)^{1/2}},\\
\mathcal{C}_5&=&\frac{w_s\langle Dh|_{\gamma+w\xi},\gamma_s\rangle-(1-kw)\langle Dh|_{\gamma+w\xi},\xi\rangle}{\big((1-kw)^2+w^2_s\big)^{1/2}},
\end{eqnarray*}
close to the boundary point $j\in\{0,1\}$. We would like to mention that, after straightforward computations, it follows that
$$
\mathcal{C}_3=\frac{-\rho(1-kw)}{|\varphi+u\xi|^2\sqrt{(1-kw)^2+w_s^2}}
$$
and
\begin{eqnarray*}
\mathcal{C}_4&=&\frac{w_s\langle\gamma,\gamma_s\rangle-w(1-kw)}{|\gamma+w\xi|^2(1-kw)}\\
&=&\frac{-gu+ku^2+(-1)^j\langle\varphi,\varphi_s\rangle u+|s-j|\langle\varphi,\varphi_s\rangle u_s+(-1)^j|\varphi|^2u_s}
{|\varphi+u\xi|^2\sqrt{(1-kw)^2+w_s^2}}.
\end{eqnarray*}
For sufficiently small $\varepsilon>0$, and after choosing $\ell$ in Corollary \eqref{pb1} sufficiently close to $\ell_{0}+1/2$, we have
\begin{eqnarray*}
\omega_{j}u\in E_{4}&=&C([0,T_{\varepsilon}]; \mathcal{H}_{p}^{r+2-\frac{2}{q}-\varepsilon,\frac{1}{2}+(\ell_{0}+2-\frac{2}{q})(1+\beta_{j})-\varepsilon}(\mathbb{B})\oplus\mathbb{C}_{\omega}),\\
\omega_{j}u_{s}\in E_{3}&=&C([0,T_{\varepsilon}]; \mathcal{H}_{p}^{r+1-\frac{2}{q}-\varepsilon,-\frac{1}{2}+(\ell_{0}+2-\frac{2}{q})(1+\beta_{j})-\varepsilon}(\mathbb{B})),\\
\omega_{j}w\in E_{2}&=&C([0,T_{\varepsilon}]; \mathcal{H}_{p}^{r+2-\frac{2}{q}-\varepsilon,\frac{3}{2}+(\ell_{0}+2-\frac{2}{q})(1+\beta_{j})-\varepsilon}(\mathbb{B})\oplus\mathbb{S}_{\omega}),\\
\omega_{j}w_{s}\in E_{1}&=&C([0,T_{\varepsilon}]; \mathcal{H}_{p}^{r+1-\frac{2}{q}-\varepsilon,\frac{1}{2}+(\ell_{0}+2-\frac{2}{q})(1+\beta_{j})-\varepsilon}(\mathbb{B})\oplus\mathbb{C}_{\omega}),\\
\omega_{j}w_{ss}\in E_{0}&=&C([0,T_{\varepsilon}]; \mathcal{H}_{p}^{r-\frac{2}{q}-\varepsilon,-\frac{1}{2}+(\ell_{0}+2-\frac{2}{q})(1+\beta_{j})-\varepsilon}(\mathbb{B})).
\end{eqnarray*}
Note that,
$$E_{4}\hookrightarrow E_{3}\hookrightarrow E_{0}\quad \text{and} \quad E_{2}\hookrightarrow E_{1}\hookrightarrow E_{0}.$$
Moreover, due to Lemma \ref{abr} (a), we have
$$E_{0}\hookrightarrow C([0,T_{\varepsilon}];C(0,1)).$$
By Lemma \ref{abr} (a), the spaces $E_{i}$, $i\in\{1,2,3,4\}$, are Banach algebras. Moreover, by Lemma \ref{abr} (c), the spaces
$E_{1}$ and $E_{4}$ are closed under holomorphic functional calculus. In addition, Lemma \ref{abr} (b) implies that elements in $E_{i}$, 
$i\in\{1,2,3,4\}$, act by multiplication as bounded maps on $E_{0}$. Additionally, 
$$
\mathcal{H}_{p}^{b-1,\alpha-1}(\mathbb{B})\oplus\mathbb{C}_{\omega}\hookrightarrow E_{4} \quad \text{and} \quad \mathcal{H}_{p}^{b-2,\alpha-2}(\mathbb{B})\hookrightarrow E_{0}.
$$

Therefore, by Lemma \ref{abr} (a)-(c) and \eqref{inreg1}-\eqref{inreg3}, we have that
$$
\omega_{j}\{\gamma_{1}+w\xi_{1},\gamma_{2}+w\xi_{2}\}\in E_{2}\quad\text{and}\quad \omega_{j}|\varphi+u\xi|\in E_{4}.
$$
Moreover,
$$
\omega_{j}\{w_{s},w_{s}^{2},1-kw,(1-kw)^{2}+w_{s}^{2}\}\in E_{1}.
$$
Close to the boundary point $j\in\{0,1\}$, we write
$$k_{s}ww_{s}=|s-j|k_{s}w(|s-j|^{-1}w_{s}),$$
so that
$$\omega_{j}|s-j|k_{s}\in E_{0}\quad\text{and} \quad\omega_{j}w(|s-j|^{-1}w_{s})\in E_{1}.$$
As a consequence,
$$\omega_{j}k_{s}ww_{s}\in E_{0}.$$
Therefore,
$$
\omega_{j}\{g,\rho,w_{ss}, kw_{s}^{2}, k(1-kw)^{2}\}\in E_{0}
$$
and
$$
\omega_j\{gu,ku^2,\langle\varphi,\varphi_s\rangle u,|s-j|\langle\varphi,\varphi_s\rangle u_s,|\varphi|^2u_s\}\in E_0.
$$
In order to control the function $e^{-h(\gamma+w\xi)}$, let us choose a closed smooth path $\Gamma$ within $\{z\in \mathbb{C}\, |\, |z|<R\}$ containing the sets
$$
\bigcup_{t\in[0,T_{\varepsilon}]}\mathrm{Ran}(\omega_{0}w(t)) \quad\text{and}\quad \bigcup_{t\in[0,T_{\varepsilon}]}\mathrm{Ran}(\omega_{1}w(t)).
$$
By the analyticity of $h$, we write
$$
\omega_{j}e^{-h(\gamma+w\xi)}=\frac{\omega_{j}}{2\pi i}\int_{\Gamma}\frac{e^{-h(\gamma_{1}+i\gamma_{2}+\lambda(\xi_{1}+i\xi_{2}))}}{\lambda-w}d\lambda.
$$
Since, for each $\lambda\in \Gamma$, we have 
$$
\omega_{j}\{\gamma_{1}+i\gamma_{2}+\lambda(\xi_{1}+i\xi_{2}), e^{-h(\gamma_{1}+i\gamma_{2}+\lambda(\xi_{1}+i\xi_{2}))}\}\in E_{4}
$$
and 
$$
\lambda-w\in E_{5}=C\big([0,T_{\varepsilon}]; \mathcal{H}_{p}^{r+2-\frac{2}{q}-\varepsilon,\frac{1}{2}+(\ell_{0}+2-\frac{2}{q})(1+\beta_{j})-\varepsilon}(\mathbb{B})\oplus\bigoplus_{m=0}^{n}\mathbb{S}_{\omega}^{m}\big),
$$
by Lemma \ref{sumSn} it follows that 
$$
\omega_{j}e^{-h(\gamma+w\xi)}\in E_{5}.
$$
Observe that $E_{4}\hookrightarrow E_{5}$ and that elements in $E_{5}$ act by multiplication as bounded maps on $E_{0}$. Similarly, we show that the components of
vector valued function $\omega_{j}Dh|_{\gamma+w\xi}$ also belong to $E_{5}$. Consequently, 
$$
\omega_{j}\{\langle Dh|_{\gamma+w\xi},\gamma_s\rangle, \langle Dh|_{\gamma+w\xi},\xi\rangle\}\in E_{5}.
$$
Now we can easily see that 
$$
\omega_{j}\{\mathcal{C}_{1},\mathcal{C}_{2},\mathcal{C}_{3},\mathcal{C}_{4},\mathcal{C}_{5}\}\in E_{0}. 
$$
Now we immediately see that there exists a constant $c_j$, depending only on $r$, $\varepsilon$ and $T_\varepsilon$, such that
close to the boundary point $j\in\{0,1\}$, we have
$$
|k_{\gind}|\le c_j|s-j|^{(1+\ell_0)(1+\beta_j)-\varepsilon}.
$$
In the case where $\alpha=5/2$, then
$$
\ell_0=\min\Big\{0,-\frac{1+2\beta_1}{1+\beta_1}\Big\},
$$
from where we deduce that, close to the boundary point $j\in\{0,1\}$, we have
\begin{equation*}
|k_{\gind}|\leq\left\{
\begin{array}{lll}
c_j|s-s_j|^{1+\beta_j-\varepsilon} &\text{if} &\beta_1\leq -1/2,\\
c_j|s-s_j|^{-\beta_1\frac{1+\beta_j}{1+\beta_1}-\varepsilon} &\text{if}& \beta_1>-1/2.
\end{array}
\right.
\end{equation*}
This completes the proof.
\end{proof}

\begin{remark}\label{equivalence-1-2}
Suppose that for $b=\infty$ we have a solution of \eqref{Ioannina1} and let us represent it as the graph of a function $w$ over the initial curve $\gamma$. From \eqref{s2}, \eqref{s3}
and \eqref{s4}, we obtain that close to the singular points of $\Sigma$ it has the form
\begin{equation}\label{tangentPhi}
\varGamma_t=k_{\gind}\xi_{\gind}+\varPhi\,\tau_{\gind},\quad\text{where}\quad
\varPhi=\frac{|\gamma+w\xi|^{\beta_j}e^{h_j}w_sw_t}{\sqrt{(1-kw)^2+w^2_s}}.
\end{equation}
A standard procedure in geometric flows, known as the DeTurck trick, is to find a continuous family of diffeomorphisms $y:\mathbb{B}\times[0,T]\to\mathbb{B}$, such that
$\widetilde\varGamma:\mathbb{B}\times[0,T_{0}]\to\mathbb{B}$ given by
$$\widetilde\varGamma(s,t)=\varGamma(y(s,t),t)$$
solves
\begin{equation}\label{Ioannina2}
\widetilde\varGamma_t={\bf k}_{\gind}.
\end{equation}
If such a family of diffeomorphisms exists, it should satisfy the initial value problem
\begin{equation}\label{Ioannina3}
\begin{cases}
y_t(s,t)=\varTheta(y(s,t),t),\\
y(s,0)=s,\quad\, s\in\mathbb{B},
\end{cases}
\quad \text{where} \quad \varTheta=\frac{-w_sw_t}{(1-kw)^2+w^2_s}.
\end{equation}
Observe that $\varTheta$ is continuous in $[0,1]\times[0,T]$ and smooth in $(0,1)\times(0,T)$. Consequently, from Peano's theorem, the problem
\eqref{Ioannina3} admits at least one solution. However, one cannot deduce from the asymptotic
behavior of $w$ that $\sup_{s\in(0,1)}\varTheta_s$ is bounded. Hence, we do not have Lipschitz continuity of the function $\varTheta$ with respect to the spatial parameter. This means that the classical Picard-Lindel\"of's theorem cannot be used
to show uniqueness and continuous dependence on the initial data of the solutions
of \eqref{Ioannina3}. It seems that in the singular case, in general, it is unclear if from a solution of \eqref{Ioannina1}
we can obtain a solution of \eqref{Ioannina2}.
\end{remark}

\section{Evolution equations and some geometric consequences}\label{sec7}
Let $\gind$ be a singular metric of order $\beta\in(-1,0)$ at the origin of $\C$, i.e. in a system of local coordinates we have that
$$
\gind=e^{2h}|z|^{2\beta}|dz|^2
$$
where $h$ is an analytic function. Throughout this section, we will always assume that
$\gamma:\mathbb{B}\to\C$ is an embedded curve such that
$\gamma(0)=0=\gamma(1)$ and satisfying the regularity conditions (a)-(d) in Section \ref{assumptongama} with $b=\infty$.

\subsection{Evolution equations}
Let us evolve the curve $\gamma$ by the DCSF, i.e. consider
the map $\varGamma:\mathbb{B}\times[0,T_{\max})\to(\C,\gind)$ which solves
\eqref{Ioannina1}, where from now on we denote by $T_{\max}$ the {\em maximal time of existence of the
solution} given by Theorem \ref{wexreg}. Equivalently, $\varGamma$
satisfies the evolution equation
$$
\varGamma_t=k_{\gind}\xi_{\gind}+\varPhi\,\tau_{\gind},
$$
where $\varPhi$ is given in \eqref{tangentPhi}. From Corollary \ref{pb1}, we have that
$$
\varPhi(0,t)\tau_{\gind}(0,t)=\varPhi(0,t)\xi_{\gind}(0,t)=0=\varPhi(1,t)\tau_{\gind}(1,t)=\varPhi(1,t)\xi_{\gind}(1,t),
$$
for any $t$. Let us see how several important quantities evolve
under the degenerate flow. Denote by $\mu$ the speed function of each evolved curve, i.e.
$$
\mu=|\varGamma_s|_{\gind}=e^{h\circ\varGamma}|\varGamma|^{\beta}|\varGamma_s|.
$$
Denote by $\si(\,\cdot,t)$ be the arc-length parameter of the evolved curve
$\varGamma(\,\cdot,t):\mathbb{B}\to(\C,\gind)$.
Then we have
$$
\si(s,t)=\int_0^s\mu(y,t)\, dy
$$
and $\si_s(s,t)=\mu(s,t)>0.$
Moreover,
\begin{equation}\label{sxy}
\partial_\si=\mu^{-1}\partial_s=e^{-h\circ\varGamma}|\varGamma|^{-\beta}|\varGamma_s|^{-1}\partial_s.
\end{equation}

By straightforward computations, we deduce the following observation.

\begin{lemma}\label{vary}
The induced by $\gind$ Laplacian $\Delta$ on $\mathbb{B}$ is an operator which near to the endpoint $j\in\{0,1\}$ of the interval 
$\mathbb{B}$ has the form
$$
\Delta=\partial^2_\zeta=P|s-j|^{-2\beta}\partial^2_s+(Q-\beta P)|s-j|^{-2\beta-1}\partial_s,
$$
where $P$ and $Q$ are continuous functions on $[0,1]\times[0,T_{\max})$, with $P$ being positive. If around each boundary
point $j\in\{0,1\}$
we perform the coordinate change
$$
|s-j|=\big((1+\beta)|x-j|\big)^{\frac{1}{1+\beta}},
$$
then the Laplacian operator $\Delta$ takes the form
\begin{eqnarray*}
\Delta=P\partial^2_x+Q\big((1+\beta)|x-j|\big)^{-1}\partial_x,
\end{eqnarray*}
where the leading term is uniformly parabolic.
\end{lemma}

The following evolution equations are easily obtained by direct computations, keeping in mind that
$$k_{\gind}(0,t)=k_{\gind}(1,0)=0\quad\text{and}\quad \varPhi(0,t)\tau_{\gind}(0,t)=\varPhi(1,t)\tau_{\gind}(1,t)=0$$
for any $t$; compare for example with \cite{Gr1} and \cite[Lemma 2.6]{Mantegazza2}.

\begin{lemma}\label{lemaevol}
The following evolution equations hold:
\begin{enumerate}[\rm (a)]
\item The speed $\mu$ and the length element $d\si$ evolve according to
$$
\mu_t=-(k^2_{\gind}-\varPhi_\si)\mu\quad\text{and}\quad\nabla_{\partial_t}d\si=-(k^2_{\gind}-\varPhi_\si) d\si.
$$
\item The Lie bracket of $\partial_\si$ and $\partial_t$ is given by
$$
\big[\partial_t,\partial_\si\big]=(k^2_{\gind}-\varPhi_\zeta)\partial_\si.
$$
\item The tangent $\tau_{\gind}$ and the normal $\xi_{\gind}$ satisfy the evolution equations
$$
\nabla_{\partial_t}\tau_{\gind}=\big((k_{\gind})_{\si}+\varPhi k_{\gind}\big)\xi_{\gind}
\quad\text{and}\quad\nabla_{\partial_t}\xi_{\gind}=-\big((k_{\gind})_{\si}+\varPhi k_{\gind}\big)\tau_{\gind},
$$
respectively.
\smallskip
\item The evolution of the length $L$ and of the enclosed area $\mathcal{A}$ is given by
$$
L_t=\lim_{s\to 1}\varPhi(s,t)-\lim_{s\to 0}\varPhi(s,t)-\int_0^1 k^2_{\gind}d\si\quad\text{and}
\quad {\mathcal A}_t=-\int_0^1 k_{\gind}d\si.
$$
\item The commutator between $\partial_\si$ and $\partial_t$ satisfy the identity
$$
\nabla_{\partial_t}\nabla_{\partial_\si}Z
=\nabla_{\partial_\si}\nabla_{\partial_t}Z+(k^2_{\gind}-\varPhi_{\si})\nabla_{\partial_\si}Z
-k_{\gind}\mathcal{R}_{\gind}(\tau_{\gind},\xi_{\gind},Z),
$$
where $Z$ is a time-dependent vector field along the evolving curves and $\mathcal{R}_{\gind}$ the Riemann curvature tensor of $\gind$.
\smallskip
\item The curvature $k_{\gind}$ evolves in time according to
$$
\big(k_{\gind}\big)_t= (k_{\gind})_{\si\si}+\varPhi(k_{\gind})_{\zeta}+k_{\gind}^3+k_{\gind}K_{\gind},
$$
where $K_{\gind}$ is the Gauss curvature of the metric $\gind$.
\end{enumerate}
\end{lemma}

\subsection{The maximum principle}
A useful tool to control the behavior of various quantities is the comparison principle;
see \cite{hamilton}, \cite[Chapter 7]{andrews} or \cite[Chapter 2]{Mantegazza}. For
our purposes let us state the one-dimensional version of this principle, adapting it to our situation.

\begin{theorem}\label{maxprinc}
Consider the open set $\Omega=(0,1)\times(0,T)$, where $T>0$, and
assume that $u\in C^2(\Omega)\cap C(\overline\Omega)$ is a solution of the differential equation
$$
u_t=Pu_{xx}+Q u_x+\varPsi(u)
$$
where $P$ is a continuous and positive on $[0,1]\times[0,T)$, $Q$ is continuous on $(0,1)\times[0,T)$ and
$\varPsi$ a locally Lipsichtz function. Suppose that for every $t\in(0,T)$ there exists a value $\delta> 0$
and a compact subset $K\subset(0,1)$, such that at every time $t'\in(t-\delta,t+\delta)\cap[0,T)$ the maximum
$u_{\max}(t')=\max_{s\in[0,1]}u(\cdot,t')$ {\rm(}resp. the minimum $u_{\min}(t')=\min_{s\in[0,1]}u(\cdot,t')${\rm)}
is attained at least at one point of $K$. Then, the following conclusions hold:
\begin{enumerate}[\rm (a)]
\item
The function $u_{\max}$ {\rm(}resp. $u_{\min}${\rm)} is locally Lipschitz in $(0,T)$ and, at every differentiability 
time $t\in (0,T)$, we have
$$u'_{\max}(t)\le \varPsi(u_{\max}(t))\quad \big(\text{resp. } u'_{\min}(t)\ge \varPsi(u_{\min}(t))\big).$$
\item If $\phi:[0,T)\to\R$ {\rm(}resp. $\psi:[0,T)\to\R${\rm)} is the solution of the associated ODE
$$
\left\{
\begin{array}{lll}
\phi'(t) &= &\varPsi(\phi(t)) \\
\phi(0) &= &u_{\max}(0)
\end{array}
\right.
\,\,\left(\text{resp.}\,\,
\left\{
\begin{array}{lll}
\psi'(t) &= &\varPsi(\psi(t)) \\
\psi(0) &= & u_{\min}(0)
\end{array}
\right.\right)
$$
then
$$
u(x,t)\le \phi(t)\,\,\,\big(\text{resp. } u(x,t)\ge\psi(t)\big)
$$
for all $(x,t)\in[0,1]\times [0,T)$.
\end{enumerate}
\end{theorem}

\subsection{Some geometric consequences}
Applying the parabolic maximum principle we can estimate the maximal time of solution and
prove that several geometric properties are preserved during the flow.

\begin{lemma}\label{t1}
The distance-type function $d:\mathbb{B}\times[0,T_{\max})$ given by
$$d(s,t)=e^{2h\circ\varGamma(s,t)}|\varGamma(s,t)|^{2+2\beta}+2t,$$
satisfies the equation
\begin{equation}\label{311032021}
d_t= d_{\si\si}+\varPhi d_\si.
\end{equation}
In particular, the maximal time of solution satisfies
$$
T_{\max}\le\frac{1}{2}\max_{\mathbb{B}}\big( e^{h\circ\gamma}|\gamma|^{2+2\beta}\big).
$$
\end{lemma}
\begin{proof}
Note at first that, because $\beta+1>0$, the function $d$ is well defined and continuous on $\mathbb{B}$.
Observe now that
$$
d=e^{2h\circ\varGamma}|\varGamma|^{2+2\beta}+2t=e^{2h}|\varGamma|^{2\beta}\langle \varGamma,\varGamma\rangle+2t
=\langle \varGamma,\varGamma\rangle_{\gind}+2t.
$$
Differentiating with respect to $t$, we have
$$
d_t=2k_{\gind}\langle\varGamma, \xi_{\gind}\rangle_{\gind}+2\varPhi\langle\varGamma,\tau_{\gind}\rangle_{\gind}+2.
$$
Moreover, differentiating with respect to $\si$, we obtain
$$
d_\si=2\langle\varGamma,\tau_{\gind}\rangle_{\gind}\quad\text{and}\quad d_{\si\si}=2+2k_{\gind}\langle\varGamma,\xi_{\gind}\rangle_{\gind}.
$$
Hence, the function $d$ satisfies the equation \eqref{311032021}. Recall now that, according to the conclusion of the previous section,
the flow stay fixed at the origin. Hence, $d(0,t)=2t=d(1,t)$, for any $t\in[0,T_{\max})$. Hence, $d$ satisfies the assumptions of Theorem \ref{maxprinc}. From the parabolic
maximum principle
we obtain that
\begin{equation}\label{distancesmprini}
2t\le d(s,t)=e^{2h\circ\varGamma(s,t)}|\varGamma(s,t)|^{2+2\beta}+2t\le{\max}_{\mathbb{B}}\big( e^{h\circ\gamma}|\gamma|^{2+2\beta}\big).
\end{equation}
Therefore, taking limit as $t$ tends to $T_{\max}$, we deduce that
$$
T_{\max}\le\frac{1}{2}\max_{\mathbb B}\big( e^{h\circ\gamma}|\gamma|^{2+2\beta}\big).
$$
This completes the proof.
\end{proof}

\begin{lemma}\label{conv1}
Let $\gamma:\mathbb{B}\to(\C,\gind=|z|^{2\beta}|dz|^2)$ be a locally uniformly convex {\rm(}with respect to the Riemannian metric{\rm)} curve. Then,
this property is preserved under the degenerate curvature flow.
\end{lemma}

\begin{proof}
Since $\beta<0$, from Lemma \ref{kg} and Theorem \ref{wexreg} we deduce that, for each $t\in[0,T_{\max})$, the function $k_{\gind}(\cdot,t)$ is
smooth on $(0,1)$, continuous on $\mathbb{B}$ and $k_{\gind}(0,t)=0=k_{\gind}(1,t)$.
We will show that the non-negativity of $k_{\gind}$ is preserved under the degenerate curvature flow. To show this, suppose to
the contrary that there are points in space-time where $k_{\gind}$ becomes negative. This means that there is a time $t_1\ge 0$ such that
$\min k_{\gind}(\cdot,t)=0$ for all $t\le t_1$ and $\min k_{\gind}(\cdot,t)<0$ for any $t\in(t_1,t_2)$, where $t_2\le T_{\max}$.
Thus, the restriction of $k_{\gind}$ on $\mathbb{B}\times(t_1,t_2)$ satisfies the conditions of Theorem \ref{maxprinc}. Hence,
from Lemma \ref{lemaevol} (f) and Theorem \ref{maxprinc} we deduce that
$$
0=\frac{\min_{x\in\mathbb{B}} k_{\gind}(s,t_1)}{\big(1-2t\min_{x\in\mathbb{B}} k^2_{\gind}(s,t_1)\big)^{1/2}}
\le k_{\gind}(s,t),
$$
for all $(s,t)\in\mathbb{B}\times(t_1,t_2)$, which leads to a contradiction. Hence, $k_{\gind}$ stays non-negative
along the evolution. This completes the proof.
\end{proof}

\begin{lemma}\label{isometry2}
Let $L$ be the closed half-line of non-negative real numbers in $\C$ and $F:\C\backslash L\to \varLambda_{\theta}$ the isometry described in \eqref{isometry}.
Suppose that
$\gamma:\mathbb{B}\to\mathbb{C}$ is an embedded smooth regular curve passing through the origin $O$ of $\mathbb{C}$
such that $\gamma(\mathbb{B}^{\rm o})\cap L=\emptyset$ and let $\varGamma:\mathbb{B}\times[0,T_{\max})\to\C$ be the solution of the degenerate curvature flow defined in a maximal time interval
$[0,T_{\max})$. Then,
$F\circ\varGamma:\mathbb{B}^{\rm o}\times[0,T_{\max})\to \varLambda_{\theta}$ provides a solution of the standard CSF in the plane. Additionally, for any $j\in\{0,1\}$, we have that
$$
\lim_{s\to j}F\circ\varGamma(s,t)=0\quad\text{and}\quad \lim_{s\to j}k({F\circ\varGamma}(s,t))=0,
$$
where $k$ denotes the curvature of the curve $F\circ\varGamma(\cdot,t):\mathbb{B}^{\rm o}\to \R^2$.
\end{lemma}
\begin{proof}
Since $\gamma(\mathbb{B}^{\rm o})\cap L=\emptyset$, by the avoidance principle and the shape of geodesics of $(\C,\gind)$ described in Lemma \ref{geod},
we easily see that
$$\varGamma(\mathbb{B}^{\rm o}\times[0,T_{\max}))\cap L=\emptyset;$$
see Fig. \ref{Simera} below.
\begin{figure}[ht]
	\includegraphics[scale=0.9]{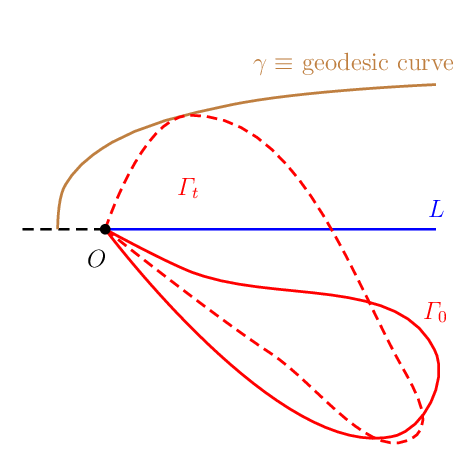}\caption{Avoidance.}\label{Simera}
	\end{figure}
Hence, $F\circ\varGamma:\mathbb{B}^{\rm 0}\times[0,T_{\max})\to \varLambda_{\theta}$ is well defined and, since $F$ is an isometry, it provides a solution of the CSF in the sector $\varLambda_{\theta}$ of $\R^2$.
Moreover, from \eqref{isometry}, we immediately
see that
$$
\lim_{s\to j}F\circ\varGamma(s,t)=0,
$$
for any $j\in\{0,1\}$. Since by Theorem \ref{kgregularitys} (b) the geodesic curvature of evolved curves vanishes at the origin, it follows that
$$
\lim_{s\to j}k({F\circ\varGamma}(s,t))=0,
$$
for any $t\in[0,T_{\max})$ and $j\in\{0,1\}$. This completes the proof.
\end{proof}

\section{Evolution of enclosed areas in flat singular spaces}\label{sec8}
Let us discuss briefly here the following interesting situation. Fix $n$ points $\{p_1,\dots,p_n\}$ in the complex plane and consider the singular
metric
$$
\gind=|z-p_1|^{2\beta_1}\cdots|z-p_n|^{2\beta_n}|dz|^2,
$$
where the orders $-1<\beta_1\le \cdots\le \beta_n<0$. Clearly the Gauss curvature of $\gind$ is zero.
Using exactly the same computations as in the proof of Lemma \ref{lemaevol},
we see that the enclosed areas of the evolved curves satisfy
$$
\mathcal{A}_t(t)=-\int_{\varGamma} k_{\gind}.
$$
Using the Gauss-Bonnet formula in Theorem \ref{SGB}, we get the following result.

\begin{lemma}\label{AGB}
Let $(\C,\gind)$ be the complex plane equipped with a flat metric $\gind$
with conical singular points $\{p_1,\dots,p_n\}$ of orders $-1<\beta_1\le\cdots\le \beta_n<0$,
respectively. Suppose that $\gamma\colon\mathbb{S}^1\to\C$ is
a $C^\infty$-smooth closed embedded curve passing through the singular points $\{p_1,\dots,p_m\}$, $m\le n$, and containing the
rest into its interior.
Then,
the enclosed areas $\mathcal{A}$ of the evolved curves
satisfy
$$
\mathcal{A}_t(t)=-2\pi-\sum_{j=1}^m(\pi-\alpha_j(t))\beta_j+\sum_{j=1}^m\alpha_j(t)-2\pi\sum_{j=m+1}^n\beta_j,
$$
where $\{\alpha_1,\dots,\alpha_m\}$ are the {\rm(}time dependent{\rm)} exterior 
angles of the evolved curves at the vertices $\{p_1,\dots,p_m\}$.
\end{lemma}

\section{Proofs of the main theorems}\label{sec9}
In this section, we will prove the Theorems \ref{THMA}, \ref{THMB}, \ref{THMC}, \ref{THMD} and \ref{THME}.

{\bf Proof of Theorems \ref{THMA} and \ref{THMB}:}
The results follow from the Theorem \ref{wexreg}, Corollary \ref{pb1}, Corollary \ref{gammaflow} and the Theorem \ref{kgregularitys}.

{\bf Proof of Theorem \ref{THMC}:} The results follow from Lemma \ref{AGB}.

{\bf Proof of Theorem \ref{THMD}:} From Lemma \ref{isometry2}, the curves $\{F\circ\varGamma(\cdot,t)\}_{t\in[0,T_{\max})}$ are moving by
the CSF in a sector $\varLambda_{\theta}$ which is contained in a half-plane. Moreover, the evolved curves stay fixed at the origin of $\R^2$ and, for each fixed $t$, the curve
$F\circ\varGamma(\cdot,t)$
has zero curvature at $O\in\R^2$. Let us extend now each $F\circ\varGamma(\cdot,t)$, $t\in[0,T_{\max})$, to a planar curve $\{\varGamma_{\theta}(\cdot,t)\}$, $t\in[0,T_{\max})$, by reflecting with respect to the origin $O$ of $\R^2$.
In this way, we obtain a CSF, of point symmetric {\em balanced figure-eight curves}; see
\cite[Definition 4]{Drugan}. In fact, $\{\varGamma_{\theta}(\cdot,t)\}_{t\in(0,T_{\max})}$ instantaneously becomes a family of real analytic curves. According to the uniqueness part of Theorem \ref{wexreg},
the DCSF on $(\C,\gind)$ with initial data a curve satisfying the assumptions of our theorem, is in to a one-to-one correspondence
with the CSF of point symmetric curves in the euclidean plane. By a result of Grayson \cite[Lemma 3]{grayson3}
(see also \cite[Corollary 2]{Drugan})
the flow will contract in finite time into the origin $O$.

{\bf Proof of Theorem \ref{THME}:} Since the initial curve is not passing through singular points
of the surface, its evolution is done by the standard CSF. By the avoidance
principle and the structure of the geodesics near the conical tips in Lemma \ref{geod}, the evolved curves will not
approach the singular points of the surface. Embeddedness of the
evolved curves is guarantied by results of Gage \cite[Theorem 3.1]{Ga1} and Angenent \cite{Ang4}.
Since the evolving curves are staying in a compact region of the surface, the result follows by Grayson
\cite[Theorem 0.1]{Gr1}. 

\begin{remark}
The long-time behaviour of general figure-eight curves in the euclidean space $\R^2$ under the curve shortening
flow, is not fully understood. For some new results in this direction we refer to \cite{Drugan}. Finally, we refer to \cite{Ma} for results
similar to Theorem E. Our methods are different from those in \cite{Ma}.
\end{remark}

{\bf Acknowledgements:} We would like to express our sincere gratitude to the anonymous  referee for the careful reading of the manuscript and the thoughtful comments.

\end{document}